\documentclass[11pt]{amsart}
\usepackage{amssymb}
\usepackage{pgfplots}
\usepackage{csvsimple}
\usepackage{siunitx}
\usepackage{hyperref}
\usepackage{graphicx,color}
\usepackage[normalem]{ulem}
\usepackage{mathrsfs}
\def\R{\mathbb{R}}
\def\C{\mathbb{C}}
\def\N{\mathbb{N}}

\def\cA{\mathcal{A}}

\def\cE{\mathcal{E}}
\def\cF{\mathcal{F}}

\def\cI{\mathcal{I}}

\def\cM{\mathcal{M}}
\def\cN{\mathcal{N}}
\def\cO{\mathcal{O}}

\def\cS{\mathcal{S}}
\def\cT{\mathcal{T}}

\def\a{\alpha}
\def\b{\beta}
\def\g{\gamma}
\def\d{\delta}

\def\l{\lambda}

\def\p{\partial}
\def\o{\omega}
\def\veps{\varepsilon}
\def\vrho{\varrho}

\def\O{\Omega}

\def\G{\Gamma}
\def\GD{{\Gamma_{\rm D}}}

\def\pO{{\p\O}}

\def\DD{{\rm D}}

\def\wto{\rightharpoonup}
\def\Id{I}
\def\transp{{\sf T}}
\def\hx{\widehat{x}}
\def\hy{\widehat{y}}

\def\ty{\widetilde{y}}

\def\hf{\widehat{f}}

\def\tu{\widetilde{u}}

\newcommand{\dv}[1]{\,{\mathrm d}#1}

\newcommand{\wcheck}[1]{#1\hspace{-.8ex}\mbox{\huge {\lower.45ex \hbox{$\textstyle \check{}$}}} \hspace{.5ex}}

\DeclareMathOperator{\dist}{dist}

\DeclareMathOperator{\trace}{tr}

\DeclareMathOperator{\hull}{span}

\let\oldmarginpar\marginpar
\renewcommand\marginpar[1]{
  \oldmarginpar[\raggedleft\footnotesize #1]
  {\raggedright\footnotesize #1}}
\newtheorem{definition}{Definition}

\newtheorem{proposition}[definition]{Proposition}
\newtheorem{theorem}[definition]{Theorem}

\newtheorem{remark}[definition]{Remark}
\newtheorem{remarks}[definition]{Remarks}
\newtheorem{example}[definition]{Example}

\newtheorem{algorithm}[definition]{Algorithm}

\numberwithin{definition}{section}

\def\fulld{{\rm 3d}}
\def\rod{{\rm rod}}
\def\plate{{\rm plate}}

\def\tb{\widetilde{b}}

\def\hcI{\widehat{\cI}}
\def\gD{{\g_\DD}}
\def\hA{\widehat{A}}
\def\hM{\widehat{M}}
\DeclareMathOperator{\sym}{sym}
\def\dkt{{\rm dkt}}
\def\tor{{\rm tor}}

\def\bc{{\rm bc}}
\def\stop{{\rm stop}}
\def\iso{{\rm iso}}
\def\bil{{\rm bil}}
\def\tot{{\rm tot}}
\def\TP{{\rm TP}}
\def\fvk{{\rm fvk}}
\def\cb{c_{\rm b}}
\def\ct{c_{\rm t}}
\def\cspont{c_{\rm sc}}

\def\tc{\widetilde{c}}
\def\ic{{\rm ic}}
\def\hP{\widehat{P}}
\def\hm{{\rm hm}}
\def\hh{h_{\rm min}}
\def\fv{{\rm fvk}}
\def\tveps{\widetilde{\veps}}
\def\fvk{{F}\"oppl--von {K}\'arm\'an}
\def\verti{\star}
\def\hor{\dagger}
\def\Hom{H}

\begin{document}
\title[Approximation of nonlinear bending phenomena]{Finite element
simulation of nonlinear bending models for thin elastic rods and plates}
\author[S. Bartels]{S\"oren~Bartels}
 \address{Abteilung f\"ur Angewandte Mathematik,  
 Albert-Ludwigs-Universit\"at Freiburg, Hermann-Herder-Str.~10, 
 79104 Freiburg i.~Br., Germany}
 \email{bartels@mathematik.uni-freiburg.de}
\date{\today}
\renewcommand{\subjclassname}{%
\textup{2010} Mathematics Subject Classification}
\begin{abstract}
Nonlinear bending phenomena of thin elastic structures arise in various 
modern and classical applications. Characterizing low energy states of
elastic rods has been investigated by Bernoulli in 1738 and related 
models are used to determine configurations of DNA strands. 
The bending of a piece of paper has 
been described mathematically by Kirchhoff in 1850 and extensions 
of his model arise in nanotechnological applications such as the 
development of externally operated microtools. A rigorous mathematical 
framework that identifies 
these models as dimensionally reduced limits from three-dimensional hyperelasticity 
has only recently been established. It provides a solid basis for developing and 
analyzing numerical approximation schemes. The fourth order character of 
bending problems and a pointwise isometry constraint for large deformations 
require appropriate discretization techniques which are discussed in this article. 
Methods developed for 
the approximation of harmonic maps are adapted to discretize the isometry constraint and 
gradient flows are used to decrease the bending energy. For the case of elastic rods, 
torsion effects and a self-avoidance potential that guarantees
injectivity of deformations are incorporated. The devised and rigorously analyzed 
numerical methods are 
illustrated by means of experiments related to the relaxation of elastic knots, the 
formation of singularities in a M\"obius strip, and the simulation of 
actuated bilayer plates.  
\end{abstract}
 

\maketitle
 
\setcounter{tocdepth}{1}
\tableofcontents

\section{Introduction}\label{sec:intro}
Thin elastic structures occur in various practical applications and in fact
truly three-dimensional objects are hardly ever used. Important reasons
for this are the reduction of weight and cost but also the special mechanical 
features of rods and plates. Correspondingly, their numerical treatment is expected
to be more efficient when such structures can be described as lower-dimensional 
objects. Because of the different mechanical behavior they cannot be treated like
three-dimensional objects and new discretization
techniques are needed. Typical large bending deformations of rods and plates are
distinct from those of three-dimensional objects and are illustrated
in Figure~\ref{fig:sketches}.

\begin{figure}
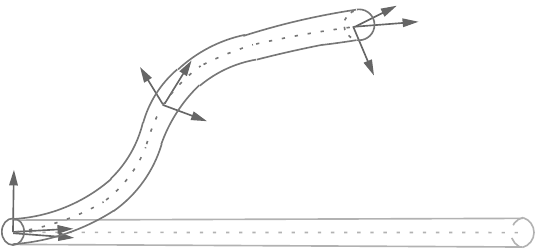 \hspace*{3mm}
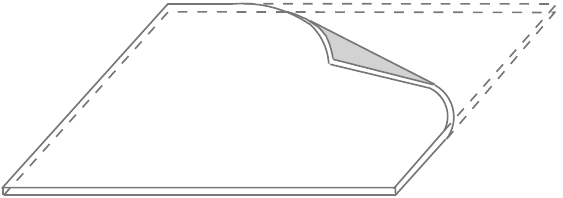
\caption{\label{fig:sketches} The mathematical description of large bending
deformations of thin objects requires the use of appropriate geometric 
quantities: deformed rod with circular cross-section together with an orthonormal
frame that allows to measure bending and torsion effets (left); deformation
of a flat plate that preserves angle and length relations (right). }
\end{figure}

In this article we address the numerical approximation of dimensionally 
reduced models for describing large deformations of thin elastic rods and plates. 
These models result from rigorous limiting processes of general 
three-dimensional hyperelastic material descriptions when the diameter of a 
circular rod or the thickness of a plate is small compared to the
length or diameter and when the acting forces lead to deformations with
energies comparable to the third power of the diameter or thickness. Examples
of such situations are the bending of a springy wire or sheet of paper. 

Characteristic for large nonlinear bending phenomena is that nearly no shearing 
or stretching effects of the object occur and that curvature quantities 
define the amount of energy required for particular deformations. These
aspects become explicitly apparant in the dimensionally reduced models: 
the energy functionals depend on curvature quantities and an isometry
condition arises in the vanishing thickness or diameter limit. In particular,
this condition implies that length and angle relations remain unchanged by
a deformation. 

The employed models for elastic rods and plates result from dimension
reductions of general descriptions for hyperelastic material behavior.
We thus consider an energy density $W:\R^{3\times 3} \to \R$ 
and a corresponding energy minimization of
\[
I_\fulld^\d[y] = \int_{\O_\d} W(\nabla y) \dv{x}
\]
in a set $\cA \subset W^{1,p}(\O_\d;\R^3)$ of admissible deformations 
$y:\O_\d \to \R^3$ that includes boundary conditions. The parameter 
$\d>0$ indicates a small diameter or thickness of 
the reference configuration $\O_\d\subset \R^3$, e.g., $\O_\d = (0,L)\times \d S$
for a thin rod with cross-section $\d S\subset\R^2$ containing zero 
or $\O_\d = \o \times (-\d/2,\d/2)$ for a thin plate with midplane $\o\subset \R^2$.  
Assuming that the minimal energies are comparable to $\d^3$, i.e.,
\[
\min_{y \in \cA} I_\fulld^\d[y]  = \cO(\d^3)
\]
as $\d\to 0$, and following the contributions~\cite{FrJaMu02a,FrJaMu02b,Pant03,MorMul03},
it is possible to identify limiting, dimensionally reduced
theories that determine the corresponding limits of solutions as $\d \to 0$.
The particular cubic scaling characterizes bending phenomena of 
the elastic body and excludes membrane effects, we refer the reader 
to~\cite{FrJaMu06,ConMag08,FHMP16} for discussions of models corresponding to other
scaling regimes. We outline the 
numerical methods that have been developed for simulating nonlinear 
bending behavior of rods and plates after a 
discussion of the related literature. 

Throughout this article we use energy minimization 
principles to determine
deformations subject to boundary conditions and external forces. For other 
approaches to the modeling of rods and plates via equilibria of
forces or conservation of momentum we refer the reader 
to~\cite{Antm05-book,Ciar97-book,AudPom10-book}.
Only a few numerical methods have been discussed mathematically for the 
numerical solution of nonlinear bending models with inextensibility or 
isometry constraint. The articles~\cite{WBHZG07,Bergetal08} devise various 
methods to compute discrete curvature quantities. 
The focus of this article is on the reliability of methods, i.e., the 
accuracy of finite element discretizations and the convergence of
iterative solution methods for the discrete problems. 
The methods discussed here
use techniques developed for the approximation of harmonic 
maps into surfaces in the articles~\cite{Alou97,Bart05,Bart16}. 
We review the numerical treatment of rods following~\cite{Bart13c,BarRei19-in-prep-pre}
and plates as proposed in~\cite{Bart13b,Bart15-book}. For 
the efficient iterative solution we adopt ideas 
from~\cite{NasSof96,KPPRS18-pre}. We discuss the treatment of bilayer
plates following~\cite{BaBoNo17,BBMN18}, illustrate a method
that enforces injectivity of deformations in the case of rods
following~\cite{BaReRi18,BarRei18-pre}, and propose methods for 
the numerical solution of bending deformations with shearing effects
following ideas from~\cite{Bart17}. The problems considered in this
article have similarities with problems related to the length-preserving
elastic flow of curves and the surface area and volume preserving
Willmore--Helfrich flow of closed surfaces but require different numerical methods. 
For contributions related to those problems we refer the reader 
to~\cite{DzKuSc02,DeDzEl05,BaGaNu07,BaGaNu12,SaNeBi16,PozSti17,BoNoNt-in-prep-pre}; 
for examples of modern applications 
of nonlinear bending phenomena including the construction of micromachining
fingers, the fabrication of nanotubes, the occurence of wrinkling in
plastic sheets, and the description of certain properties of DNA molecules, 
we refer the reader to~\cite{Smela-etal93,SchEbe01,ShRoSw07,ThCoSw04}.

\subsection{Bending of elastic rods}
We consider an
elastic rod, e.g., a springy wire, which in its reference configuration 
occupies the region $(0,L)\times \{0\}\subset \R^3$. A low energy deformation 
\[
y: (0,L) \to \R^3
\]
leaves distances of pairs of points on the rod unchanged. This is
described by the inextensibility (and incompressibility) condition
\[
|y'(x_1)|= 1
\]
for almost every $x_1\in (0,L)$. For appropriate boundary conditions a
deformation then minimizes the bending energy
\[
I_\rod[y] = \frac12 \int_0^L |y''(x_1)|^2 \dv{x_1}.
\]
The inextensibility condition implies that $y$ defines an 
arclength parametrization of the deformed rod and hence its curvature
is given by the second derivative of $y$. The simple energy functional $I_\rod$,
which has been proposed by Bernoulli in 1738, 
ignores torsion effects and arises as a special case of the dimension
reduction from three-dimensional hyperelasticity. 

It is interesting to see that the dimension reduction leads to significant
changes in the nature of the energy functionals. The three-dimensional
model depends on strains, is not constrained, and often provides existence of
unique solutions. The dimensionally reduced functional $I_\rod$ depends on
curvature, is constrained, and is singular
in the sense that the set of admissible deformations may be empty, e.g., for
extensive boundary conditions, and that solutions may be non-unique, e.g., for
simple compressive boundary conditions. 
These aspects are related to the presence of a critical nonlinearity 
via a Lagrange multiplier for the
inextensibility constraint in 
the Euler--Lagrange equations for critical points of $I_\rod$, i.e.,
\[
(y'',w'') = (\lambda y',w') \quad 
\Longleftrightarrow \quad y^{(4)} = (\lambda y')',
\]
where the scalar function $\lambda$ depends nonlinearly on $y$. 
The explicit presence of a 
Lagrange multiplier can be avoided if only test functions are considered
that satisfy the linearized inextensibility condition $y'\cdot w' = 0$. This corresponds
to normal, i.e., non-tangential perturbations of a curve in the energy minimization. 

The inextensibility condition requires a suitable numerical treatment to 
avoid locking phenomena or other artifacts. For a partitioning of the interval
$(0,L)$ with nodes 
\[
0 = z_0 < z_1 < \dots < z_N = L
\]
and a subordinated conforming finite element space $\cA_h\subset H^2(0,L;\R^3)$
we impose the inextensibility condition only at these nodes, i.e.,
\[
|y_h'(z_i)| = 1
\]
for $i=0,1,\dots,N$. Since $y_h\in H^2(0,L;\R^3)$ we obtain
linear convergence with respect to the 
meshsize $h$ of the constraint violation error away from the nodes. 
The discrete minimization problem then seeks a minimizer
$y_h \in \cA_h$ for the functional
\[
y_h \mapsto I_\rod^{h} [y_h] = \frac12 \int_0^L |y_h''(x_1)|^2 \dv{x_1},
\]
subject to the nodal constraints $|y_h'(z_i)|=1$ for $i=0,1,\dots,N$. 
A possible choice of a finite element space uses piecewise cubic, 
continuously differentiable functions. This space has the advantage that
its degrees of freedom are the positions and tangent vectors at the nodes,
i.e., 
\[
y_h \equiv \big(y_h(z_i),y_h'(z_i)\big)_{i=0,\dots,N}.
\]
The discretized inextensibility condition can thus be explicitly imposed
on certain degrees of freedom, cf.~Figure~\ref{fig:pw_cubics}.

\begin{figure}
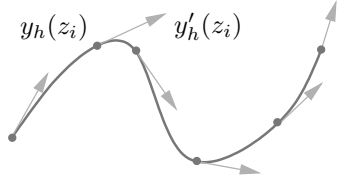
\caption{\label{fig:pw_cubics} Continuously differentiable, piecewise
cubic curves are defined by positions and tangent vectors at nodes 
$z_0<z_1<\dots<z_N$.}
\end{figure}

We iteratively solve the discrete minimization problem by using a 
gradient flow, i.e., on the continuous level we consider a family
$y:[0,T]\times (0,L)\to \R^3$ of deformations that solve the evolution
equation
\[
(\p_t y,w)_\star = - (y'',w'')_{L^2}, \quad y(0) = y_0,
\]
subject to the linearized inextensibility conditions 
\[
\p_t y'(t,x_1) \cdot y'(t,x_1) = 0, \quad w'(x_1) \cdot y'(t,x_1) = 0.
\]
Provided that we have $|y_0'(x_1)|^2=1$ it follows that
\[
|y'(s,x_1)|^2  - 1 = \int_0^s \frac{d}{dt} |y'(t,x_1)|^2 \dv{t}
= 2 \int_0^s \p_t y'(t,x_1) \cdot y'(t,x_1) \dv{t} = 0,
\]
i.e., the inextensibility condition is satisfied. 
We use an implicit discretization of the evolution equation
and a semi-implicit treatment of the linearized constraint, i.e.,
with the backward difference quotient operator $d_t$ we consider 
the time-stepping scheme 
\[
(d_t y_h^k, w_h)_\star = - ([y_h^k]'',[w_h]'')_{L^2}, \quad y_h^0 = y_{0,h},
\]
subject to the linearized constraints evaluated at the nodes, i.e., 
\[
[d_t y_h^k]'(x_i) \cdot [y_h^{k-1}]'(x_i) = 0, 
\quad [w_h]'(x_i) \cdot [y_h^{k-1}]'(x_i) = 0
\]
for $i=0,1,\dots,N$. The scheme is unconditionally energy-decreasing and
convergent to a stationary configuration,  
i.e., choosing the admissible test function $w_h = d_t y_h^k$ directly shows
\[
I_\rod^h[y_h^k] + \tau \|d_t y_h^k\|_\star^2  \le I_\rod^h[y_h^{k-1}].
\]
The inextensibility constraint will not be satisfied exactly at the nodes
but its violation is controlled by the step size $\tau$ and the initial
energy. A proof in the discrete setting imitates the continuous argument given
above. Using the orthogonality $[d_t y_h^k]'(z_i) \cdot [y_h^{k-1}]'(z_i) = 0$ 
and the property $|[y_h^0]'(z_i)|^2 = 1$ we have  
\[\begin{split}
|[y_h^k]'(z_i)|^2 -1  & = |[y_h^{k-1}](z_i)'|^2 + \tau^2 |[d_t y_h^k]'(z_i)|^2 -1 \\
& =  \dots   
= \tau^2 \sum_{\ell=1}^k |[d_t y_h^\ell]'(z_i)|^2.
\end{split}\]
Because of the unconditional energy stability the term on the right-hand side is
of order $\cO(\tau)$. We will show below that these properties are also valid
if torsion effects are taken into account. 

\subsection{Elastic plates}
The mathematical description and numerical treatment of elastic plates 
generalizes that of elastic rods. In the dimensionally reduced model we consider
deformations of a two-dimensional midplane 
\[
y:\o \to \R^3
\]
that leave angle and area relations unchanged, i.e., they satisfy the isometry
condition
\[
(\nabla y\big)^\transp \nabla y = I_2
\]
almost everywhere in $\o \subset \R^2$ with the identity matrix $I_2\in \R^{2\times 2}$. 
This is equivalent to saying
that the tangent vectors $\p_1 y$ and $\p_2 y$ of the deformed plate and
the normal vector $b = \p_1 y \times \p_2 y$ define an orthonormal basis for $\R^3$
in almost every point $x'\in \o$. The actual deformation for appropriate boundary conditions
minimizes the bending energy proposed by Kirchhoff in 1850,
\[
I_\plate[y] = \frac12 \int_\o |D^2 y|^2 \dv{x'}.
\]
Because of the isometry condition, the integrand coincides with the mean curvature
of the deformed plate while its Gaussian curvature vanishes. For a minimizing
or critical isometry $y$ we have that 
\[
(D^2 y, D^2 w) = 0
\]
for all test fields $w$ satisfying appropriate homogeneous boundary conditions
and the linearized isometry condition
\[
(\nabla y)^\transp \nabla w + (\nabla w)^\transp \nabla y = 0.
\]
A finite element discretization uses a possibly nonconforming finite element
space such as so-called discrete Kirchhoff triangles and imposes the isometry
condition in the set of nodes $\cN_h$, i.e., for all $z\in \cN_h$ we have 
\[
\big(\nabla y_h(z)\big)^\transp \nabla y_h(z) = I_2.
\]
With a discrete Hessian $D_h^2$ the numerical minimization is then realized
for the functional 
\[
I_\plate^h [y_h] = \frac12 \int_\o |D_h^2 y_h|^2 \dv{x'}.
\]
In the case of the discrete Kirchhoff triangle, which may be seen as a 
natural generalization of the space of one-dimensional cubic $C^1$ functions, 
the degrees of freedom
are the deformations and the deformation gradients in the nodes, i.e., the
quantities 
\[
\big(y_h(z),\nabla y_h(z)\big)_{z\in \cN_h}.
\]
An image of a discrete Kirchhoff deformation is depicted in 
Figure~\ref{fig:dkt_surf_sketch}.

\begin{figure}
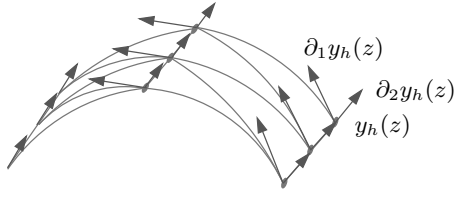
\caption{\label{fig:dkt_surf_sketch} Discrete deformations defined by
discrete Kirchhoff triangles are defined by positions of nodes and 
tangent vectors at the displaced nodes.}
\end{figure}

The isometry constraint is thus imposed directly on certain degrees of freedom.
The iterative numerical minimization follows closely the approach used
in the one-dimensional situation. For an initial isometry $y_0$, we consider
the continuous evolution problem
\[
(\p_t y,w)_\star = -(D^2y,D^2 w), \quad y(0)= y_0,
\]
for appropriate test functions $w\in H^2(\o;\R^3)$ subject to the 
linearized isometry condition
\[
L_{\nabla y}^\iso[\p_t \nabla y] = 0, \quad L_{\nabla y}^\iso[\nabla w] = 0,
\]
with the linearized isometry operator 
\[
L_A^\iso[B] = A^\transp B + B^\transp A. 
\]
A semi-implicit discretization of this constrained evolution problem
leads to a sequence of linearly constrained problems: given an admissible
$y_h^0 \in \cA_h$ compute the sequence $(y_h^k)_{k=0,1,\dots}$ via
$y_h^k = y_h^{k-1}+\tau d_t y_h^k$, where $d_t y_h^k$ solves 
\[
(d_t y_h^k,w_h)_\star = -(D_h^2 y_h^k,D_h^2 w_h) 
\]
subject to the conditions 
\[
L_{\nabla y_h^{k-1}}^\iso[d_t \nabla y_h^k] = 0, \quad
L_{\nabla y_h^{k-1}}^\iso[\nabla w_h] = 0.
\]
Again, straightforward calculations show that the iteration is energy
decreasing and convergent, and that the constraint violation is 
of order $\cO(\tau)$. 

\subsection{Outline of the article}
The article is organized as follows. In Section~\ref{sec:dim_red} we discuss
the arguments that lead to dimensionally reduced models for elastic rods and plates
in the case of small energies. We review partial $\Gamma$-convergence results
and explain the occurrence of the inextensibility and isometry constraints. 
Section~\ref{sec:fe_discr} is devoted to the convergent and practical finite 
element discretization of the one- and two-dimensional minimization problems 
describing the elastic deformation of rods and plates. The rigorous justification
of the finite element methods will be established via showing $\Gamma$-convergence
of the discretized functionals to the continuous one as discretization parameters
tend to zero. Difficulties arise in the appropriate treatment of nonlinear 
constraints and higher order derivatives. The practical minimization of the
discretized energy functionals is addressed in Section~\ref{sec:grad_flow}. 
We use appropriately discretized gradient flows that lead to sequences of linear 
systems of equations together with guaranteed energy decrease. Moreover, we verify
that they converge to stationary configurations. In view of the nonuniqueness
and limited additional regularity properties 
of solutions this appears to be the best attainable result if no further assumptions
are made. The special 
saddle-point structure of the linear systems of equations that arise in the
time steps is investigated in Section~\ref{sec:saddle_solve}. It turns out that
the nodal constraints can be incorporated in the solution space
leading to reduced linear systems
with symmetric and positive definite system matrices.  
Section~\ref{sec:extensions} is concerned with extensions and modifications of
the models and solution methods. In particular, we discuss the 
numerical treatment of bilayer bending problems, the inclusion
of a self-avoidance potential, and a bending problem allowing for the 
formation of wrinkles. The final section provides a summary and conclusions of our 
considerations. We end the introduction with an overview of employed notation. 

\subsection{Notation}
Throughout this article we use standard notation for derivatives and integrals,
matrices and inner products, Lebesgue, Sobolev, and finite element
spaces. The list in Table~\ref{tab:notation} provides an overview of
the most important symbols.

\begin{table}[p]
{\small
\begin{tabular}{|p{33mm}|p{86mm}|} \hline
 $(0,L)$, $\o$, $\O$   & one-, two-, and three-dimensional domains \\[1mm]
 $L^p(A;\R^\ell)$ & Lebesgue functions with values in $\R^\ell$  \\[1mm]
 $W^{k,p}(A,\R^\ell)$ & Sobolev functions with values in $\R^\ell$ \\[1mm]
 $H^k(A,\R^\ell)$ & Sobolev space with $p=2$ \\[1mm]
 $x=(x_1,x_2,x_3)$ & spatial variable \\[1mm]
 $x'=(x_1,x_2)$ & planar component of spatial variable \\[1mm]
 $|\cdot|$ & Euclidean or Frobenius norm of a vector or matrix \\[1mm]
 $(\cdot,\cdot)$, $\|\cdot\|$ & scalar product and norm in $L^2$ \\[1mm]
 $x\cdot y$, $A:B$ & scalar products of vectors and matrices \\[1mm]
 $I_\ell$ & identity matrix in $\R^{\ell\times \ell}$ \\[1mm]
 $\sym(A)$, $\trace(A)$ & symmetric part and trace of a matrix \\[1mm]
 $SO(3)$ & orthogonal matrices with positive determinant \\[1mm]
 $y'$ & one-dimensional derivative \\[1mm]
 $\nabla y = [\p_1 y,\p_2 y,\p_3 y]$ & gradient of a vector field \\[1mm]
 $\nabla' y$ & planar component of gradient \\[1mm]
 $G'$, $I'$  & total or Fr\'echet derivative \\[1mm]
 $P_k(A)$ & polynomials of degree at most $k$ on a set $A$ \\[1mm]
 $h$, $\hh$  & maximal and minimal mesh-sizes \\[1mm]
 $\cT_h$  & triangulation with intervals or triangles \\[1mm]
 $\cN_h$, $\cS_h$ & nodes and sides in a triangulation \\[1mm]
 $z$, $z_T$, $z_S$ & vertices and midpoints of elements, midpoints of sides \\[1mm] 
 $\cS^{k,\ell}(\cT_h)$ & elementwise degree $k$ polynomials in $C^\ell$ \\[1mm]
 $\cI_h^{1,0}$, $\hcI_h^{1,0}$ & global and elementwise nodal $P1$ interpolants \\[1mm]
 $Q_h$ & elementwise averaging operator \\[1mm]
 $\|\cdot\|_{L^p_h}$, $\|\cdot\|_h$, $(\cdot,\cdot)_h$ & discrete $L^p$ norms, case $p=2$, discrete $L^2$ product \\[1mm]
 $\tau$ & step size \\[1mm]
 $d_ta^k = (a^k-a^{k-1})/\tau$ & backward difference quotient for step size $\tau>0$ \\[1mm]
 $\d$ & small thickness parameter \\[1mm]
 $I[y]$ & energy functional \\[1mm]
 $L[y]$ & linear operator \\[1mm]
 $W$ & energy density \\[1mm]
 $Q_3$, $Q_\rod$, $Q_\plate$ & quadratic forms \\[1mm]
 $\l$, $\mu$ & Lam\'e parameters \\[1mm]
 $\cb$, $\ct$ & bending and torsion rigidity \\[1mm]
 $\cA$, $\cA_h$ & sets of admissible deformations \\[1mm]
 $\cF[y]$, $\cF_h[y_h]$ & (shifted) tangent spaces \\
 $(\cdot,\cdot)_\star$, $(\cdot,\cdot)_\dagger$ & metrics used to define gradient flows \\
 $c$, $c'$, $c''$, ... & generic constants \\\hline
\end{tabular} 
}
\caption{\label{tab:notation} Frequently used notation.}
\end{table}

\section{Formal dimension reductions}\label{sec:dim_red}
Following the articles~\cite{FrJaMu02a,FrJaMu02b,MorMul03}
we illustrate in this section how the dimensionally reduced minimization 
problems can be obtained from a general three-dimensional hyperelastic 
energy minimization problem:
\[
\tag{${\rm P}_\fulld$} 
\left\{ \, 
\begin{array}{l}
\text{Minimize} \quad \displaystyle{I_\fulld [y] = \int_\O W(\nabla y) \dv{x}} \\[2mm]
\text{in the set } \cA \subset W^{1,p}(\O;\R^3).
\end{array} \right.
\]
The set of admissible deformations $\cA$ is assumed to be a weakly closed 
subset of a Sobolev
space $W^{1,p}(\O;\R^3)$ and required to include appropriate 
boundary conditions which imply a coercivity property. In an abstract way
these are defined by a bounded linear operator 
\[
L_\bc:W^{1,p}(\O;\R^3)\to Y
\]
and given data
$\ell_\bc \in Y$ for a suitable linear space $Y$, e.g., traces of functions
in $W^{1,p}(\O;\R^3)$ restricted to a subset $\GD$ of $\p\O$. 
We assume that the energy density
$W\in C^2(\R^{3\times 3})$ satisfies the following standard requirements: 
\begin{itemize}
\item $W$ is frame-indifferent, i.e., for all $F\in \R^{3\times 3}$
and $Q\in SO(3)$ we have
\[
W(Q F) = W(F),
\]
\item $W$ vanishes at the identity $I_3 \in \R^{3\times 3}$ 
and grows at least quadratically away from $SO(3)$, i.e., for all $F \in \R^{3\times 3}$
we have 
\[
W(I_3) = 0, \quad W(F) \ge c \, {\rm dist}^2(F,SO(3)),
\]
\item $W$ is isotropic, i.e., for all $F\in \R^{3\times 3}$
and $R\in SO(3)$ we have
\[
W(F R) = W(F).
\]
\end{itemize}

From the first two conditions we have that $W'(I_3) = 0$ and a Taylor expansion yields
\[
W(I_3+G) = \frac12 Q_3(G) + o(|G|^2),
\]
where $Q_3(G) = W''(I_3)[G,G]$ is the quadratic form defined by the second variation
of $W$ at the identity matrix. 
Incorporating the implicitly assumed homogeneity of the underlying material we have 
\[
Q_3(G) =  \C G: G,
\]
where the linear operator $\C:\R^{3\times 3}\to \R^{3\times 3}$ is given by 
\[
\C A = 2\mu \sym(A) + \lambda \trace(A) I_3,
\]
with the Lam\'e parameters $\lambda,\mu>0$, the symmetric part
$\sym(A) = (A+A^\transp)/2$, and the trace $\trace(A) = A:I_3$. 

\subsection{Elastic rods}\label{sec:dim_red_rods}
We assume that the deformation $y:(0,L) \to \R^3$ of an elastic rod of
vanishing thickness and 
length $L$ preserves distances, i.e., satisfies $|y'(x_1)| = 1$, and complement
the vector field $y'$ by normal vector fields $b,d:(0,L)\to \R^3$ 
to an orthonormal frame 
\[
[y',b,d]: (0,L)\to SO(3).
\]
We then consider a three-dimensional deformation $y_\d$ obtained from extending
the deformation of the centerline $(0,L)$ to the three-dimensional body 
$\O_\d = (0,L)\times \d S$ with scaled cross section $S\subset \R^2$ containing
zero, i.e., 
\[
y_\d (x_1,x_2,x_3) = y(x_1) + x_2 b(x_1) + x_3 d (x_1) + \d^2 \b(x),
\]
with a correction function $\b:\O_\d \to \R^3$. Inserting this deformation
into the three-dimensional energy functional and considering the limit as $\d\to 0$,
we expect to identify a dimensionally reduced functional minimized by $y$ and the
normal fields $b$ and $d$. We note that $x_2,x_3 = \cO(\d)$ and that we
expect $\p_2 \b, \p_3 \b = \cO(\d^{-1})$. We therefore set  
\[\begin{split}
\nabla y_\d &= \big[y',b,d\big] 
+ \big[x_2b' + x_3 d',\d^2 \p_2 \b, \d^2 \p_3 \b \big]+ \d^2 \big[\p_1 \b,0,0\big] \\
&= R + \d B + \d^2 C.
\end{split}\]
with the matrix $R=[y',b,d] \in SO(3)$. 
The matrix $R^\transp \nabla y_\d$ is thus a perturbation of the identity
matrix $I_3$ and a Taylor expansion of the energy density yields with 
\[
R^\transp \nabla y_\d = I_3 + \d R^\transp B + \d^2 R^\transp C
\]
that we have
\[\begin{split}
W(\nabla y_\d) & = W(R^\transp \nabla y_\d)  \\
& = \frac12 Q_3 \big(R^\transp \big[x_2b'+x_3d',\d^2 \p_2 \b, \d^2 \p_3\b\big]\big) + o(\d^2).
\end{split}\]
Letting $\a = \d^2 R^\transp \b$ and noting that $R$ does not depend on $x_2$ and $x_3$ we have
\[
R^\transp \big[x_2b'+x_3d',\d^2 \p_2 \b, \d^2 \p_3\b \big]
= R^\transp R' \begin{bmatrix} 0 \\ x_2 \\ x_3 \end{bmatrix} + \big[0,\p_2 \a,\p_3 \a\big].
\]
We note that since $R^\transp R = I_3$ we have $(R^\transp)'R = -R^\transp R'$ so that
$R^\transp R'$ is skew-symmetric. A minimization of the integral of the energy density
over the cross section $\d S$ motivates defining for skew-symmetric matrices 
$A = [a_1,a_2,a_3]\in \R^{3\times 3}$ the reduced quadratic form $Q_\rod$ via 
\[
Q_\rod(A) = \min_{\a \in H^1(S;\R^3)}
 \int_{\d S} Q_3 \big(\big[x_2 a_2 + x_3 a_3,\p_2 \a,\p_3 \a\big]\big) \dv{x_2}\dv{x_3}.
\]
With the particular representation of $Q_3$ by the Lam\'e parameters one finds with
the entries $a_{ij} = -a_{ji}$ of $A$ for a circular cross section $S= B_{1/\pi}(0)$ 
that 
\[
Q_\rod(A) = \frac{1}{2\pi} \frac{\mu(3\lambda+2\mu)}{\lambda+\mu} (a_{12}^2 + a_{13}^2) 
+ \frac{\mu}{2\pi} a_{23}^2.
\]
The constant factors on the right-hand side define the bending and torsion
rigidities and are abbreviated by
\[
\cb = \frac{1}{2\pi} \frac{\mu(3\lambda+2\mu)}{\lambda+\mu}, \quad
\ct = \frac{\mu}{2\pi}.
\]
We always assume $\lambda,\mu>0$ so that $\cb \ge 2 \ct$. 
For the particular matrix $A=R^\transp R'$ and $R=[y',b,d] \in SO(3)$ we 
have that
\[
a_{12} = y'' \cdot b, \quad a_{13} = y''\cdot d, \quad a_{23} = b' \cdot d.
\]
Noting that $y''\cdot y'=0$ we have that
\[ 
a_{12}^2 + a_{13}^2 = |y''|^2 
\]
is the squared curvature of the deformed rod and that
\[
a_{23}^2 = (b' \cdot d)^2 = (d' \cdot b)^2
\]
is its squared torsion. We eliminate the variable $d$ via the identity
$d= y'\times b$ in what follows. We thus expect that the deformation $y:(0,L)\to \R^3$
of the centerline of a thin rod and the unit normal vector field $b:(0,L)\to \R^3$ 
solve the following dimensionally reduced problem:
\[\label{eq:min_rods}
\tag{${\rm P}_\rod$} \left\{ 
\begin{array}{l}
\text{Minimize} \\
\quad \displaystyle{I_\rod[y,b] = \frac{\cb}{2} \int_0^L |y''|^2 \dv{x_1} 
+ \frac {\ct}{2} \int_0^L (b'\cdot (y'\times b))^2 \dv{x_1}} \\[2mm]
\text{in the set } \\[1mm]
 \cA = \big\{ (y,b) \in V_\rod \!:\!
L_\bc^\rod[y,b] = \ell_\bc^\rod, \, |y'| = |b| =1, y'\cdot b = 0 \big\}.
\end{array}\right.
\]
Here, we abbreviate
\[
V_\rod = H^2(0,L;\R^3) \times H^1(0,L;\R^3).
\]
The second part for justifying the dimensionally reduced model consists
in showing that for any sequence $(y_\d)_{\d>0}$ of three-dimensional deformations
with $I_\fulld[y_\d]\le c \d^3$ there exists an appropriate limit $(y,b)\in \cA$
such that 
\[
\liminf_{\d \to 0} I_\fulld[y_\d] \ge I_\rod[y,b].
\]
This so-called compactness property is proved in~\cite{MorMul03} which provides
the complete rigorous dimension reduction in a more general setting. 
We refer the reader to~\cite{LanSin96} for further aspects of the
description of elastic rods. 

\begin{remark} 
To illustrate that the inextensibility or isometry condition $|y'|=1$ arises 
naturally in the dimension reduction we consider the planar deformation of 
a two-dimensional thin beam $\O=(0,L)\times (-\d/2,\d/2)$ with the simple energy
density 
\[
W(F) = \dist^2(F,SO(2)) \approx (1/4) |F^\transp F - I|^2.
\]
We assume that the deformation is given by  
\[
y_\d(x_1,x_2) = y(x_1) + x_2 b(x_1)
\]
for a deformation $y:(0,L)\to \R^2$ of the centerline and a corresponding
normal field $b:(0,L)\to \R^2$, i.e., we have $y'(x_1)\cdot b(x_1) = 0$. 
Noting that
\[
\nabla y_\d  = \big[y' + x_2 b',b\big]
\]
we find that
\[\begin{split}
(\nabla y_\d)^\transp \nabla y_\d -I_2  
& = \begin{bmatrix} |y'|^2-1  &  0 \\ 0   & |b|^2 -1  \end{bmatrix}
+ x_2 \begin{bmatrix} 2 y'\cdot b' & b \cdot b' \\ b \cdot b' & 0 \end{bmatrix}
+ x_2^2 \begin{bmatrix} |b'|^2 & 0 \\ 0 & 0 \end{bmatrix} \\
&= A + x_2 B + x_2^2 C. 
\end{split}\]
We insert this expression into the energy functional and carry out the
integration in vertical direction, i.e. 
\[\begin{split}
I_\d[y_\d] 
& \approx \frac14\int_0^L \int_{-\d/2}^{\d/2} |A+x_2 B+ x_2^2 C|^2 \dv{x_2} \dv{x_1} \\
& =  \frac14 \int_0^L  \d |A|^2 + \frac{\d^3}{12} |B|^2 + \frac{\d^5}{80} |C|^2
+ \frac{\d^3}{12} 2  A:C  \dv{x_1}.
\end{split}\]
For a cubic scaling of the elastic energy we need $A=0$, i.e., $|y'|^2 =1$
and $|b|^2=1$. This implies $b'\cdot b = 0$, hence $|B|^2 = 4 |y'\cdot b'|^2$,
and shows that up to terms of order $\d^5$ we have 
\[
I_\d[y_\d] \approx \frac{\d^3}{12}  \int_0^L |y''|^2 \dv{x_1},
\]
where we used the identities $y'\cdot b = - y'' \cdot b$ and $y''\cdot y' = 0$ 
in combination with the fact that $(y',b)$ is an orthonormal basis in $\R^2$ so
that $(y'\cdot b)'=0$.
\end{remark}

\subsection{Elastic plates}\label{sec:dim_red_plates}
To explain the derivation of the bending model for elastic plates we consider
an isometry $y:\o\to \R^3$ and a corresponding unit normal field $b:\o \to \R^3$,
i.e., we have 
\[
\p_i y(x') \cdot \p_jy(x') = \d_{ij}, 
\]
for $1\le i,j\le 2$ and
\[ 
|b(x')|^2 = 1, \quad \p_j y(x') \cdot b(x') = 0
\]
for almost every $x'\in \o$ and $j=1,2$. We define a deformation $y_\d$ of the three-dimensional
body $\O_\d = \o \times (-\d/2,\d/2)$ by extending $y$ in normal direction, i.e.,
\[
y_\d(x',x_3) = y(x') + x_3 b(x') + (x_3^2/2) \b(x'),
\]
with a quadratic correction term $\b:\o \to \R^3$.
Using the planar gradient $\nabla' = [\p_1,\p_2]$ we have that 
\[
R = \big[\nabla' y, b \big] \in SO(3)
\]
and
\[\begin{split}
\nabla y_\d &= \big[\nabla' y_\d,\p_3 y_\d\big]
= \big[\nabla'y + x_3 \nabla' b + (x_3^2/2) \nabla'\b, b + x_3 \b\big] \\
&= R + x_3 \big[\nabla'b,\b\big] + (x_3^2/2) \big[\nabla' \b,0\big].
\end{split}\]
This implies that 
\[
R^\transp \nabla y_\d - I 
= x_3 R^\transp \big[\nabla' b, \b\big] + (x_3^2/2) \big[\nabla'\b,0\big].
\]
We insert the deformation $y_\d$ into the hyperelastic energy functional,
use the Taylor expansion $W(I+x_3 G) = Q_3(x_3 G) + o(x_3^2 |G|^2)$, 
with $G = R^\transp \nabla y_\d$, and carry out the integration in vertical direction. 
This leads to 
\[\begin{split}
\int_\o & \int_{-\d/2}^{\d/2} W  (\nabla y_\d) \dv{x_3}\dv{x'} 
= \int_\o \int_{-\d/2}^{\d/2} W(R^\transp \nabla y_\d) \dv{x_3}\dv{x'} \\
&= \frac12 \int_\o \int_{-\d/2}^{\d/2} Q_3\big(x_3 R^\transp [\nabla' b,\b] 
 + (x_3^2/2) [\nabla'\b,0]\big) \dv{x_3}\dv{x'} + o(\d^3) \\
&= \frac{\d^3}{24} \int_\o Q_3(R^\transp[\nabla' b,\b]) \dv{x'} + o(\d^3).
\end{split}\]
The correction field $\b:\o\to \R^3$ is eliminated via a pointwise minimization,
i.e., for $M \in \R^{2 \times 2}$ extended by a vanishing third row 
to a matrix $\hM \in \R^{3\times 2}$, we define
\[
Q_\plate(M) = \min_{c\in \R^3} Q_3([\hM,c]).
\]
Since we assume a homogeneous and isotropic material one obtains 
for a symmetric matrix $M \in \R^{2\times 2}$ that
\[
Q_\plate(M) = 2 \mu |M|^2 + \frac{\lambda \mu}{\mu + \lambda/2} \trace(M)^2.
\]
For the matrix $\hM = R^\transp \nabla' b \in \R^{3\times 2}$ the third
row vanishes and its uppper $2\times 2$ submatrix coincides with the
second fundamental form $II$ of the surface parametrized by $y$, i.e., 
\[
II_{ij}(x') = \p_i y(x') \cdot \p_j b(x') = - \p_i\p_j y(x') \cdot b(x'),
\]
for $i,j=1,2$, where in fact $b = \pm \p_1 y \times \p_2 y$. 
Using that $y$ is an isometry we have that the squared mean curvature
is up to a fixed factor given by the identical expressions
\[
|II|^2 = \trace(II)^2 = |D^2 y|^2 = |\Delta y|^2.
\]
Hence, the dimensionally reduced problem seeks an isometric deformation
that minimizes the integral of the squared Hessian: 
\[
\tag{${\rm P}_\plate$} \left\{ \, 
\begin{array}{l}
\text{Minimize} \quad 
\displaystyle{I_\plate[y] = \frac{\cb}{2} \int_\o |D^2 y|^2 \dv{x'}} \quad \text{in the set}\\[2.5mm]
\cA = \big\{y\in H^2(\o;\R^3): (\nabla y)^\transp \nabla y = I_2, \
L_\bc^\plate[y] = \ell_\bc^\plate\big\}.
\end{array}\right.
\]
The bending rigidity is defined by $\cb = 2 \mu + \lambda 2 \mu/(2\mu + \lambda)$.
As in the case of rods, a rigorous derivation additionally requires showing
that the functional defines a general lower bound, i.e.,
establishing a lim-inf inequality, and we refer the reader to~\cite{FrJaMu02a,FrJaMu02b}
for details. Analogously to the functional, also the boundary conditions 
change their nature in the dimension reduction. A fixed part of the
lateral boundary leads to a clamped boundary condition in the 
reduced model which imposes a condition on the deformation and
its gradient. 

\section{Convergent finite element discretizations}\label{sec:fe_discr}
We discuss in this section the discretization of the dimensionally reduced
nonlinear bending models using appropriate finite element methods. Challenges
are the treatment of higher order derivatives and a nonlinear pointwise
constraint. We establish the correctness of the discretizations by showing
that the discrete functionals~$I^h$ converge in the sense of $\G$-convergence
with respect to weak convergence on a space $X$,
cf., e.g.,~\cite{Dalm93-book}, to the continuous, dimensionally reduced 
functional~$I$. We use the terminology
{\em almost-minimizing} for a sequence of objects that are minimizers of a
sequence of functions up to tolerances that converge to zero with~$h$. 
This follows from verifying the following three conditions: 
\begin{itemize}
\item[(a)] {\em Well posedness or equicoercivity:} The discrete functionals 
are uniformly coercive, i.e., if $I^h[y_h]\le c$ then it follows that
$\|y_h\|_X \le c'$ with $h$-independent constants 
$c,c'\ge 0$, and admit discrete minimizers.
\item[(b)] {\em Stability or lim-inf inequality:} If $(y_h)_{h>0} \subset X$ is a 
bounded sequence of discrete almost-minimizers then every weak accumulation
point $y$ belongs to the set of admissible deformations $\cA$ and we have
\[
I[y] \le \liminf_{h\to 0} I^h[y_h].
\]
\item[(c)] {\em Consistency or lim-sup inequality:} For every $y\in \cA$
there exists a sequence $(y_h)_{h>0}$ of admissible discrete deformations
such that $y_h \wto y$ in $X$ and 
\[
I[y] \ge \limsup_{h\to 0} I^h[y_h].
\]
\end{itemize}
It is an immediate consequence of (a)--(c) that sequences of discrete almost-minimizers
accumulate at minimizers of the continuous problem. Well posedness typically
follows from coercivity properties of the functional $I$ while the stability
is established with the help of lower semicontinuity properties of $I$. 
If the union of discrete sets of admissible deformations is dense in the
set of admissible deformations then consistency is obtained via 
continuity properties of $I$. We specify these concepts for the finite 
element approximation of elastic deformations of rods and plates in 
what follows. We always use a regular triangulation $\cT_h$ of the domain 
$A=(0,L)$ or $A=\o$ into intervals or triangles, respectively, with a set of
nodes (vertices of elements) denoted $\cN_h$, i.e., 
\[ 
\cN_h = \{z_1,z_2,\dots,z_N\}, \quad \cT_h = \{T_1,T_2,\dots,T_M\},
\]
We often use numerical integration or quadrature, defined with the 
element\-wise applied nodal interpolation via 
\[
(v,w)_h = \int_A \hcI_h [(v\cdot w)] \dv{x}
\]
for elementwise continuous functions $v,w:A\to \R^\ell$ with $A\subset \R^d$ and 
\[
\|v\|_{L^p_h}^p = \sum_{T\in \cT_h} \frac{|T|}{d+1} \sum_{z\in \cN_h\cap T} 
|v(z)|^p.
\]
If $p=2$ we write $\|v\|_h$ instead of $\|v\|_{L^2_h}$. 
We note that these expressions define equivalent scalar products and
norms on 
function spaces containing elementwise polynomials of bounded degree,
cf., e.g.,~\cite{Bart15-book}. 

\subsection{Elastic rods}\label{sec:fem_rods}
For a discretization of the bending-torsion model~\eqref{eq:min_rods}
for elastic rods we first
derive a suitable reformulation of the minimization problem. Recalling that
for an admissible pair $(y,b) \in \cA$ and the vector $d=y'\times b$ 
we have that $[y',b,d] \in SO(3)$ almost everywhere in the interval $(0,L)$, 
we deduce that
\[
|b'|^2 = (b'\cdot y')^2 + (b'\cdot d)^2 + (b'\cdot b)^2.
\]
Since $|b|^2 =1$ the last term on the right-hand side vanishes while the
orthogonality $b\cdot y'=0$ implies that $b'\cdot y' = - b\cdot y''$. We thus
have that 
\[
(b'\cdot d)^2 = |b'|^2 - (b\cdot y'')^2.
\]
This identity leads to the equivalent representation
\[
I_\rod[y,b] = \frac{\cb}{2} \int_0^L |y''|^2 \dv{x_1} + \frac{\ct}{2} \int_0^L |b'|^2 \dv{x_1}
- \frac{\ct}{2} \int_0^L (b\cdot y'')^2 \dv{x_1}.
\]
The dimension reduction of Section~\ref{sec:dim_red_rods} shows that we have
$\cb\ge 2 \ct$ so that the last term is controlled by the first one
and the coercivity of $I_\rod^h$ becomes explicit. Another advantage of this representation
is that the last term is separately concave which allows for an effective 
iterative treatment. 
To define the discrete funtional $I_\rod^h$ 
we consider a partitioning of the reference interval $(0,L)$ defined by
sets of nodes~$\cN_h$ and elements~$\cT_h$. For this partitioning
we define the linear and cubic finite element spaces with different differentiability
requirements via
\[\begin{split}
\cS^{1,0}(\cT_h)
&= \big\{\phi_h \in C^0([0,L]): \phi_h|_T \in P_1(T) \text{ for all } T\in \cT_h\big\},\\
\cS^{3,1}(\cT_h) 
&= \big\{v_h\in C^1([0,L]): v_h|_T \in P_3(T) \text{ for all }T\in \cT_h\big\},
\end{split}\]
with sets of polynomials of degree at most $k$ on $T$ given by $P_k(T)$. 
The degrees of freedom of the finite element spaces are are depicted in 
Figure~\ref{fig:one_dim_fe}.

\begin{figure}
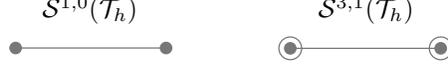
\caption{\label{fig:one_dim_fe} 
Degrees of freedom of piecewise linear, continuous and piecewise cubic, continuously
differentiable finite element functions. Filled dots indicate function values and
circles evaluations of derivatives.} 
\end{figure}

It is straightforward to verify that there exist nodal bases $(\varphi_z)_{z\in \cN_h}$ 
and $(\psi_{z,j})_{z\in \cN_h,j=1,2}$
such that for all $\phi_h\in \cS^{1,0}(\cT_h)$ and $v_h\in \cS^{3,1}(\cT_h)$ we have
\[\begin{split}
\phi_h &= \sum_{z\in \cN_h} \phi_h(z) \varphi_z, \\
v_h &= \sum_{z\in \cN_h} v_h(z) \psi_{z,0} + \sum_{z\in \cN_h} v_h'(z) \psi_{z,1}.
\end{split}\]
The right-hand sides define nodal interpolation operators $\cI_h^{1,0}$ and $\cI_h^{3,1}$
on $C^0([0,L])$ and $C^1([0,L])$, respectively. For ease of notation we use the
abbreviation 
\[
V_\rod^h = \cS^{3,1}(\cT_h)^3 \times \cS^{1,0}(\cT_h)^3 .
\]
For efficient numerical quadrature
we introduce the elementwise averaging operator $Q_h$ defined for 
a vector field $v\in L^1(0,L;\R^3)$ and every element $T\in \cT_h$ via 
\[
Q_h v|_T = |T|^{-1} \int_T v \dv{x_1}.
\]
With the product finite element space $V_\rod^h$ and the operator $Q_h$ we consider
the following discretization of the minimization problem~\eqref{eq:min_rods} in which
the pointwise orthogonality relation $y'\cdot b =0$ is approximated via a penalty term: 
\[\tag{${\rm P}_\rod^{h,\veps}$} 
\left\{ 
\begin{array}{l}
\text{Minimize}\quad \displaystyle{I_\rod^{h,\veps}[y_h,b_h] =  \frac{\cb}{2} \int_0^L |y_h''|^2 \dv{x_1} 
+ \frac{\ct}{2} \int_0^L |b_h'|^2 \dv{x_1}} \\[2mm]
\qquad\qquad \displaystyle{- \frac{\ct}{2} \int_0^L (Q_h b_h \cdot y_h'')^2 \dv{x_1} 
+ \frac{1}{2\veps} \int_0^L \cI_h^{1,0}[(y_h'\cdot b_h)^2]\dv{x_1}}  \\[2.5mm]
\text{in the set} \ \cA_h = \{(y_h,b_h)\in V_\rod^h: 
L_\bc^\rod[y_h,b_h] = \ell_\bc^\rod, \\[2.5mm]
\qquad\qquad \qquad\qquad\qquad\qquad |y_h'(z)| = |b_h(z)| = 1 \text{ f.a. }z\in \cN_h \big\}.
\end{array}\right.
\]
Note that the constraints are imposed on particular degrees of freedom
which makes the method practical. We have the following existence and 
convergence result. 

\begin{proposition}[Convergent approximation]
For every pair $(h,\veps)>0$ there exists a minimizer $(y_h,b_h)\in \cA_h$ 
for $I_\rod^{h,\veps}$ satisfying 
\[
\|y_h\|_{H^2} + \|b_h\|_{H^1}\le c.
\]
As $(h,\veps)\to 0$ we have that every accumulation point of a sequence of 
discrete almost-minimizers is a minimizer for $I_\rod$ in $\cA$. 
\end{proposition}

\begin{proof}[Proof (sketched)] 
We outline the main arguments of the proof and refer the reader 
to~\cite{Bart13c,BarRei19-in-prep-pre} for details.  \\
(a) Let $(y_h,b_h)\in \cA_h$ with $I_\rod^{h,\veps}[y_h,b_h]\le c$. The coercivity
of the discretized functional follows from the fact that $\cb\ge 2 \ct$ 
and the identity
\[\begin{split}
I_\rod^{h,\veps}[y_h,b_h] = &  \frac{\cb-\ct}{2} \int_0^L |y_h''|^2 \dv{x_1}
+ \frac{\ct}{2} \int_0^L |b_h'|^2 \dv{x_1} \\
& + \frac{\ct}{2} \int_0^L [y_h'']^\transp P_{b_h} [y_h''] \dv{x_1}
+ \frac{1}{2\veps} \int_0^L \cI_h[(y_h'\cdot b_h)^2]\dv{x_1},
\end{split}\]
with the positive semi-definite matrix
\[
P_{b_h} = I_3 - (Q_h b_h) \otimes (Q_h b_h).
\]
The discrete coercivity and the continuity properties of the functional
$I_\rod^{h,\veps}$ imply the existence of discrete minimizers. \\
(b) Given a sequence of discrete almost-minimizers $(y_h,b_h)_{h,\veps>0}$ one first checks
that accumulation points $(y,b)$ as $(h,\veps)\to 0$ belong to the continuous admissible set $\cA$.
Noting that $b_h \to b$ strongly in $L^\infty$ we find that  
\[
I_\rod[y,b] \le \liminf_{(h,\veps) \to 0} I_\rod^{h,\veps} [y_h,b_h].
\]
(c) It remains to show that $I_\rod[y,b]$ is minimal. For this, we choose 
a smooth almost-minimizing pair $(\ty,\tb)\in \cA$ obtained from an appropriate
regularization of a minimizing pair and verify that the sequence of interpolants
$(\ty_h,\tb_h)$ satisfies $\lim_{(h,\veps)\to 0} I_\rod^{h,\veps}[y_h,b_h] = I_\rod[\ty,\tb]$. 
\end{proof}

\begin{remark}
We note that the result of the proposition can be also be established if the orthogonality
relation $y_h'\cdot b_h=0$ is imposed exactly in the nodes of the triangulation.
For an efficient numerical solution of the minimization problem the 
approximation via a separately convex term is advantageous as this allows
for a decoupled treatment of the variables.
\end{remark}

\subsection{Elastic plates}\label{sec:fem_plates}
Constructing finite element spaces that provide convergent second order 
derivatives and which are efficiently implementable is
significantly more challenging in two space dimensions. Among the various 
possibilities is the {\em discrete Kirchhoff triangle}, cf., e.g.,~\cite{Brae07-book},
which defines a 
nonconforming finite element method in the sense that its elements do not
belong to $H^2$. The space can be seen as a natural generalization
of the space of one-dimensional cubic $C^1$ functions since the degrees of 
freedom are the deformations and deformation gradients at the nodes of a 
triangulation which are appropriately interpolated on the individual elements. 
To define this finite element space
we choose a triangulation $\cT_h$ of $\o$ into triangles and set
\[\begin{split}
\cS^\dkt(\cT_h) &= \{w_h\in C(\overline{\o}): w_h|_T \in P_{3-}(T)
\text{ for all } T\in\cT_h, \\ & \qquad \qquad \qquad  \qquad \qquad
 \nabla w_h \text { continuous at all }
z\in \cN_h\}, \\
\cS^{2,0}(\cT_h) &= \{ q_h \in C(\overline{\o}): q_h|_T \in P_2(T)
\text{ for all } T\in \cT_h\}.
\end{split}\]
Here, $P_{3-}$ denotes the subset of cubic polynomials
on $T$ obtained by eliminating the degree of freedom associated with
the midpoint $z_T$ of $T$, i.e., we have 
\[
P_{3-}(T) = \Big\{ p \in P_3(T): 
p(z_T) = \frac13 \sum_{z\in \cN_h\cap T} 
\big[p(z) + \nabla p(z) \cdot (z_T-z)\big] \Big\}.
\]
The degrees of freedom in $\cS^\dkt(\cT_h)$ are the function values
and the derivatives at the vertices of the elements. It is interesting
to note that a particular basis
for $\cS^\dkt(\cT_h)$ will not be needed. A canonical interpolation operator 
$\cI_h^\dkt: C^1(\overline{\o}) \to \cS^\dkt(\cT_h)$
is defined by requiring that the identities
\[
\cI_h^\dkt w(z) = w(z), \quad \nabla \cI_h^\dkt w(z) = \nabla w(z)
\]
hold at all nodes $z\in \cN_h$. The employed finite element spaces are 
depicted in Figure~\ref{fig:two_dim_fe}.

\begin{figure}
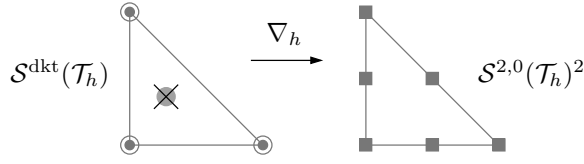
\caption{\label{fig:two_dim_fe} 
Degrees of freedom of finite element spaces using reduced cubic polynomials and
quadratic vector fields. Filled dots indicate function values,
circles evaluations of derivatives, and squares vectorial function values.
One degree of freedom is eliminated from the set of cubic polynomials.} 
\end{figure}

Crucial for the finite element discretization
of the bending problem is the definition of a discrete gradient operator
\[
\nabla_h : \cS^\dkt(\cT_h) \to \cS^{2,0}(\cT_h)^2
\]
which allows us to define discrete second order derivatives of
functions $w_h\in \cS^\dkt$ via
\[
D_h^2 w_h = \nabla \nabla_h w_h.
\]
Here we make use of the fact that $\nabla_h w_h\in H^1(\O;\R^2)$.
The discrete gradient operator $\nabla_h$ is for given $w_h\in \cS^\dkt(\cT_h)$ 
defined as the unique piecewise quadratic, continuous 
vector field $q_h\in \cS^{2,0}(\cT_h)^2$ that satisfies the condition
\[
q_h(z) = \nabla w_h(z)  
\]
for all $z\in \cN_h$ while the degrees of freedom associated with the sides
of elements are defined by the two conditions 
\[\begin{split}
q_h(z_S) \cdot n_S &=  \frac12 \big(\nabla w_h(z_S^1) + \nabla w_h(z_S^2)\big)\cdot n_S, \\
q_h(z_S) \cdot t_S &= \nabla w_h(z_S) \cdot t_S,
\end{split}\]
for all sides $S=[z_S^1,z_S^2] \in \cS_h$ with normals $n_S$, tangent vectors $t_S$,
and midpoints $z_S = (z_S^1+z_S^2)/2$. For $w\in C^1(\overline{\o})$, we set 
\[
\nabla_h w = \nabla_h \cI_h^\dkt w.
\]
With the discrete second derivatives we are in a position to state 
the finite element discretization of the plate bending model:
\[
\tag{${\rm P}_\plate^h$} 
\left\{ \, \begin{array}{l}
 \text{Minimize}\quad 
\displaystyle{I_\plate^h[y_h] =\frac{\cb}{2} \int_\o |D_h^2 y_h|^2 \dv{x'}} \quad 
\text{in the set}\\[2.5mm] 
 \cA_h = \big\{y_h \in \cS^\dkt(\cT_h)^3: \, L_\bc^\plate[y_h] = \ell_\bc^\plate, \\[2mm]
\qquad\qquad\qquad\qquad
[\nabla y_h(z)]^\transp \nabla y_h(z) = \Id_2 \text{ f.a. } z\in \cN_h \big\}.
\end{array}\right.
\]
To prove the correctness of this discretization we show that existing
finite element minimizers accumulate at admissible isometries, incorporate that the bending
energy is weakly lower semicontinuous, and use that isometries can be
approximated by smooth isometries which is a result proved in~\cite{Pakz04,Horn11}.

\begin{theorem}[Convergent approximation]
For every $h>0$ there exists a minimizer $y_h\in \cA_h$. 
If $(y_h)_{h>0}$ is a sequence of 
almost-minimizers, then $\|\nabla y_h\| \le c$, for all $h>0$, and every
accumulation point $y\in H^1(\o;\R^3)$ of the sequence is a 
strong accumulation point, belongs to the continuous admissible 
set $\cA$ and is minimal for $I_\plate$. 
\end{theorem}

\begin{proof}[Proof (sketched)]
We follow the typical steps for establishing a $\Gamma$-convergence result 
provided in~\cite{Bart13a,Bart15-book}. \\
(a) Using the boundary conditions included in the set $\cA_h$ it follows
that the mapping $z_h\mapsto \|D_h^2 z_h\|$ is a norm on the subset of 
$\cS^\dkt(\cT_h)^3$ with functions satisfying corresponding homogeneous
boundary conditions. This leads to a coercivity property and the existence
of discrete solutions. \\
(b) The uniform discrete coercivity property implies that the sequences $(D_h^2 y_h)_{h>0}$
and $(\nabla y_h)_{h>0}$ have weak accumulation points $\xi$ and $\nabla y$
in $L^2$ 
which are compatible in the sense that $\xi = D^2 y$. Moreover, we have
that $y\in \cA$ and weak lower semicontinuity of the $L^2$ norm shows that
\[
I_\plate[y] \le \liminf_{h\to 0} I_\plate^h [y_h].
\]
(c) Let $y\in \cA$ be a minimizer for $I_\plate$. The continuity of the 
functional $I_\plate$ with respect to the strong topology in $H^2$ 
in combination with the density results for smooth isometries established
in~\cite{Horn11} allow us to assume that $y$~is smooth. We may thus define
an approximating sequence of finite element functions by setting
$\ty_h = \cI_h^\dkt[y]$. Approximation properties of the interpolation
operator lead to the inequality
\[
I_\plate[y] \ge \limsup_{h\to 0} I_\plate^h [\ty_h],
\]
which proves the statement. 
\end{proof}

\section{Iterative solution via constrained gradient flows}\label{sec:grad_flow}
The practical solution of the finite element discretizations of the nonlinear bending
problems is nontrivial due to the presence of nonlinear pointwise constraints and
the corresponding lack of higher regularity properties. To provide a reliable strategy
that decreases the energy we adopt gradient flow strategies. Our estimates show that
these converge to stationary, low energy configurations. We will always use a
linearized treatment of the constraints which is then discretized semi-implicitly. This 
makes the iterative scheme practical. To illustrate the main idea, consider the
following abstract minimization problem in a Hilbert space $X \subset L^2(\O;\R^\ell)$:
\begin{equation}\label{eq:abstract_min}
\tag{${\rm M}$} \left\{
\begin{array}{l}
\text{Minimize}\quad  I[y] \\[2mm]
\text{in $X$ subject to } G[y] = 0.
\end{array}\right.
\end{equation}
Here, we assume that the constraint is understood pointwise with a 
function $G:\R^\ell\to \R$. The Euler--Lagrange equations for 
the problem are then formally given by the identity 
\[
I'[y;w] + (\l,G'[y;w]) = 0
\]
for all $w\in X$ with a Lagrange multiplier $\l\in L^1(\O)$. Note that the  
term involving~$\l$ disappears if~$w$ satisfies $G'[y;w]=0$ and that this is sufficient to 
characterize a stationary point subject to the constraint. 
The corresponding gradient flow is formally defined via
\[
\p_t y = - \nabla_X I[y] -\l G'[y,\cdot] \quad \text{subject to} \quad G'[y;\p_t y] = 0. 
\]
Our corresponding time-stepping scheme uses the backward difference quotient operator
$d_t a^k = (a^k-a^{k-1})/\tau$ for a step-size $\tau>0$ and determines iterates via
the linearly constrained problems
\[
d_t y^k = -\nabla_X I[y^k]-\l^k G'[y^{k-1},\cdot] \quad \text{subject to} \quad G'[y^{k-1};d_t y^k] = 0.
\]
We specify the meaning of the iterative scheme in the following algorithm. 

\begin{algorithm}[Abstract constrained gradient descent]\label{alg:abstr_constr_gradflow}
Let $y^0\in X$ be such that $G[y^0]=0$ and $I[y^0]<\infty$ and choose $\tau>0$,
set $k=1$.  \\
(1) Compute $y^k \in X$ such that 
\[
(d_ty^k, w)_X +  I'[y^k;w] = 0 
\]
for all $w \in X$  under the constraints $d_ty^k,\, w \in \ker G'[y^{k-1}]$, i.e., 
\[
G'[y^{k-1};d_t y^k] = 0, \quad G'[y^{k-1};w] = 0.
\]
(2) Stop the iteration if $\|d_t y^k\|_X \le \veps_{\rm stop}$; otherwise
increase $k\to k+1$ and continue with~(1). 
\end{algorithm}

We remark that it is useful to regard $d_t y^k$ rather than $y^k$ 
as the unknown in the iteration steps. In particular, we may eliminate 
$y^k$ via the identity $y^k = y^{k-1}+\tau d_t y^k$ with the known function $y^{k-1}$. 
A geometric interpretation of the iteration is that given an iterate $y^{k-1}$ 
the correction $d_t y^k$ is computed in the tangent space of the level set $\cM^{k-1}$
of~$G$ defined by the value $G[y^{k-1}]$. Note that we do not use a 
projection step onto the zero level set of $G$ since this will in general
not be energy stable. Figure~\ref{fig:abstr_constr_gradflow} illustrates 
the conceptual idea of Algorithm~\ref{alg:abstr_constr_gradflow}. 

\begin{figure}[h]
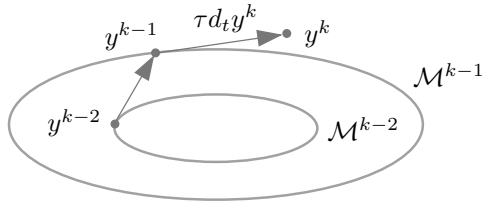
\caption{\label{fig:abstr_constr_gradflow} Illustration of the 
iteration of Algorithm~\ref{alg:abstr_constr_gradflow}: corrections are 
computed in tangent spaces of level sets $\cM^\ell$ of $G$ defined
by the values $G[y^\ell]$.}
\end{figure}

The following theorem states the main features of 
Algorithm~\ref{alg:abstr_constr_gradflow}, i.e., its unconditional 
energy stability with the resulting convergence to a stationary 
configuration of lower energy and the control of the constraint violation
by the step size. 

\begin{theorem}[Convergent iteration]\label{thm:abstract_flow}
Assume that $I$ is convex, coercive, continuous, and Fr\'echet differentiable
on $X$ and $G:\R^\ell \to \R$ is twice differentiable 
with uniformly bounded second derivative, i.e., we have 
$G''[r;s,s]\le 2 c_G |s|^2$ for all $r,s\in \R^\ell$. 
Then, the iterates of Algorithm~\ref{alg:abstr_constr_gradflow}
are uniquely defined and satisfy for $L=0,1,2,\dots$ 
the energy estimate  
\[
I[y^L] + \tau \sum_{k=1}^L \|d_t y^k\|_X^2 \le I[y^0].
\]
Moreover, if $\|\cdot\|_\Hom$ is a norm on $X$ with the property
$\||z|^2\|_\Hom \le c_\Hom \|z\|_X^2$ for all $z\in X$ then we have
the constraint violation bound
\[
\max_{k=1,\dots,L} \|G(y^k)\|_\Hom \le \tau c_G c_\Hom e_0,
\]
where $e_0 = I[y^0]$. 
In particular, we have that $\|d_t y^k\|_X\to 0$ and 
Algorithm~\ref{alg:abstr_constr_gradflow}
terminates within a finite number of iterations. The output $y^L \in X$
satisfies the residual estimate 
\[
\sup_{\substack{w\in X\setminus \{0\}\\ G'[y^{L-1};w] = 0}}
\frac{|I'[y^L;w]|}{\|w\|_X} \le \veps_{\rm stop}.
\]
\end{theorem}

\begin{proof}
(a) The existence of the iterates follows by applying the direct method 
of the calculus of variations to the minimization problems: 
\[\begin{split}
&\text{Minimize}\quad z\mapsto  \frac{1}{2\tau} \|z-y^{k-1}\|_X^2  + I[z] \\
&\text{subject to} \quad G'[y^{k-1};z]=0. 
\end{split}\]
The solution is unique and the corresponding Euler--Lagrange equation
coincides with the equation that defines the iterates in 
Algorithm~\ref{alg:abstr_constr_gradflow}. \\
(b) Choosing the admissible test function $w = d_t y^k$ and using the 
convexity of $I$ leads to 
\[
\|d_t y^k\|_X^2 + \frac{1}{\tau} \big(I[y^k] - I[y^{k-1}]\big) \le 0.
\]
A multiplication by $\tau$ and summation over $k=1,2,\dots,L$ prove
the energy stability. \\
(c) With a Taylor expansion of $G$ about the iterate $y^{k-1}$ and
the imposed identity $\d G[y^{k-1};d_t y^k]=0$ we find that 
\[\begin{split}
G[y^k] &= G[y^{k-1}] + \frac12  \tau^2 G''[\xi;d_ty^k,d_ty^k].
\end{split}\]
Repeating this argument and noting that $G[y^0]=0$ leads to the estimate 
\[
|G[y^\ell]| \le c_G \tau^2 \sum_{k=1}^\ell |d_t y^k|^2. 
\]
Applying the norm $\|\cdot\|_\Hom$ to the estimate, using the triangle
inequality, and incorporating the assumed 
bound $\||z|^2\|_\Hom \le c_\Hom \|z\|_X^2$ as well as the energy bound
proves the estimate for the constraint violation error. \\
(d) The estimate for $I'[y^L;w]$ is an immediate consequence of
the bound $\|d_t y^L\|_X \le \veps_{\rm stop}$. 
\end{proof}

Examples of pairs of norms $\|\cdot\|_\Hom$ and $\|\cdot\|_X$ that satisfy
the assumed estimate are the $L^1$ norm in combination with the $L^2$
norm or the $L^\infty$ norm together with a Sobolev norm in $H^s$ with $s$
sufficiently large. It is remarkable that the violation of the constraint
is independent of the number of iterations. An explanation for this
is that the updates $d_t y^k$ converge quickly to zero in the gradient
flow iteration. 

\subsection{Elastic rods}\label{sec:iter_rods}
We apply the abstract framework for constrained minimization problems
to the energy functional describing the bending-torsion behavior of 
elastic rods. For a vector field $y_h\in \cS^{3,1}(\cT_h)^3$ we set 
\[
\cF_h[y_h] = \{ w_h \in \cS^{3,1}(\cT_h)^3: \, L_{\bc,y}^\rod[w_h] = 0,
\, y_h'(z)\cdot w_h'(z) = 0 \text{ f.a. } z\in \cN_h\}
\]
while for a vector field $b_h\in \cS^{1,0}(\cT_h)^3$ we define
\[\begin{split}
\cE_h[b_h] = \{v_h\in  \cS^{1,0}&(\cT_h)^3: L_{\bc,b}^\rod[v_h]=0, \, v_h(z)\cdot b_h(z) = 0
\text{ f.a. } z\in \cN_h \}.
\end{split}\]
The functionals $L_{\bc,y}^\rod$ and $L_{\bc,b}^\rod$
are the components of $L_{\bc}^\rod$
corresponding to the variables $y$ and $b$, respectively. 
We generate a sequence $(y_h^k,b_h^k)_{k=0,1,\dots}$ that
approximates a stationary configuration for $I_\rod^{h,\veps}$ 
with the following algorithm.

\begin{algorithm}[Gradient descent for elastic rods]\label{alg:descent_rods}
Choose an initial pair $(y_h^0,b_h^0) \in \cA_h$ and a step size $\tau>0$,
set $k=1$. \\
(1) Compute $d_t y_h^k\in \cF_h[y_h^{k-1}]$ such that for all $w_h \in \cF_h[y_h^{k-1}]$ we have
\[\begin{split}
(d_t y_h^k,w_h)_\star + \cb ([y_h^k]'', w_h'') + & \veps^{-1}([y_h^k]'\cdot b_h^{k-1},w_h' \cdot b_h^{k-1})_h \\
&= \ct \big([Q_h b_h^{k-1}]\cdot [y_h^{k-1}]'',[Q_h b_h^{k-1}] \cdot [w_h]''\big). 
\end{split}\]
(2) Compute $d_t b_h^k\in \cE_h[b_h^{k-1}]$ such that 
for all $r_h\in \cE_h[b_h^{k-1}]$ we have
\[\begin{split}
(d_t b_h^k, r_h)_\dagger + \ct ([b_h^k]',r_h')+ & \veps^{-1}([y_h^k]'\cdot b_h^k,[y_h^k]' \cdot r_h)_h \\
&= \ct \big([Q_h b_h^{k-1}]\cdot [y_h^k]'',Q_h r_h \cdot [y_h^k]''\big). 
\end{split}\]
(3) Stop the iteration if 
\[
\|d_t y_h^k\|_\star + \|d_t b_h^k\|_\dagger  \le \veps_{\rm stop};
\]
otherwise, increase $k\to k+1$ and continue with~(1).
\end{algorithm}

Again, it is useful to view $d_t y_h^k$ and $d_t b_h^k$ as the unknowns
in Steps~(1) and~(2) instead of $y_h^k=y_h^{k-1}+\tau d_t y_h^k$
and $b_h^k = b_h^{k-1}+\tau d_t b_h^k$. 
The algorithm exploits the fact that the penalty term is separately
convex while the nonquadratic contribution to the torsion term is 
separately concave. Therefore, the decoupled semi-implicit 
treatment of these terms is natural and unconditionally energy stable. 

\begin{proposition}[Convergent iteration]
Assume that we have
\[
\|w_h'\|_h \le c_\star \|w_h\|_\star, \qquad
\|r_h\|_h \le c_\dagger \|r_h\|_\dagger 
\]
for all $(w_h,r_h)\in V^h_\rod$ with $L_\bc^\rod[w_h,r_h] = 0$.
Algorithm~\ref{alg:descent_rods} is well defined and produces a sequence
$(y_h^k,b_h^k)_{k=0,1,\dots}$ such that for all $L\ge 0$ we have
\[
I_\rod^{h,\veps}[y_h^L,b_h^L] + \tau \sum_{k=1}^L \big(\|d_t y_h^k\|_\star^2
+ \|d_t b_h^k\|_\dagger^2 \big) \le I_\rod^{h,\veps}[y_h^0,b_h^0].
\]
The iteration controls the unit-length violation via 
\[
\max_{k=0,\dots,L} \||[y_h^k]'|^2-1\|_{L^1_h} + \||b_h^k|^2-1\|_{L^1_h} 
\le \tau c_{\star,\dagger} e_{0,h},
\]
where $e_{0,h} = I_\rod^{h,\veps}[y_h^0,b_h^0]$. 
In particular, the algorithm terminates within a 
finite number of iterations. 
\end{proposition}

\begin{proof}
(a) To prove the stability estimate we note that the functional
\[
G_h[y_h,b_h] =  \frac{\ct}{2} \int_0^L (Q_h b_h \cdot y_h'')^2 \dv{x_1}
\]
is separately convex, i.e., convex in $y_h$ and in $b_h$. Therefore, we have that
\[\begin{split}
\p_y G_h[y_h^{k-1},b_h^{k-1};y_h^k-y_h^{k-1}] + G_h[y_h^{k-1},b_h^{k-1}] & \le G_h[y_h^k,b_h^{k-1}], \\
\p_b G_h[y_h^k,b_h^{k-1};b_h^k-b_h^{k-1}] + G_h[y_h^k,b_h^{k-1}] &\le G_h[y_h^k,b_h^k],
\end{split}\]
which by summation leads to the inequality 
\[
\p_y G_h[y_h^{k-1},b_h^{k-1};d_t y_h^k] + \p_b G_h[y_h^k,b_h^{k-1};d_t b_h^k]
\le d_t G_h[y_h^k,b_h^k].
\]
Similarly, the functional 
\[
P_{h,\veps}[y_h,b_h] = \frac{1}{2\veps} \int_0^L \cI_h^{1,0}[(y_h' \cdot b_h)^2] \dv{x_1}
\]
is separately convex and we have
\[
\p_y P_{h,\veps}[y_h^k,b_h^{k-1};d_t y_h^k] + \p_b P_{h,\veps}[y_h^k,b_h^k;d_t b_h^k]
\ge d_t P_{h,\veps}[y_h^k,b_h^k].
\]
By choosing $w_h = d_t y_h^k$ and $r_h = d_t b_h^k$ in the equations of
Steps~(2) and~(3) of Algorithm~\ref{alg:descent_rods} we thus find that
\[\begin{split}
\|d_t y_h^k\|_\star^2& + \|d_t b_h^k\|_\dagger^2  
+  d_t \big( \frac{\cb}{2} \|[y_h^k]''\|^2 + \frac{\ct}{2} \|[b_h^k]'\|^2 \big) 
+ d_t P_{h,\veps}[y_h^k,b_h^k] \\
&\qquad \qquad\qquad \qquad  + \tau \big(\frac{\cb}{2} \|[d_t y_h^k]''\|^2 + \frac{\ct}{2} \|[d_t b_h^k]'\|^2 \big) \\
& \le \p_y G_h[y_h^{k-1},b_h^{k-1};d_t y_h^k] + \p_b G_h[y_h^k,b_h^{k-1};d_t b_h^k] \le d_t G_h[y_h^k,b_h^k].
\end{split}\]
Since 
\[
I_\rod^{h,\veps}[y_h^k,b_h^k] = \frac{\cb}{2} \|[y_h^k]''\|^2 + \frac{\ct}{2} \|[b_h^k]'\|^2 
- G_h [y_h^k,b_h^k] + P_{h,\veps}[y_h^k,b_h^k] 
\]
we deduce the asserted estimate. \\
(b) The nodal orthogonality conditions encoded in the spaces 
$\cF_h[y_h^{k-1}]$ and $\cE_h[b_h^{k-1}]$ lead to the relations
\[\begin{split}
|[y_h^k]'(z)|^2 & = |[y_h^{k-1}]'(z)|^2 + \tau^2 |[d_t y_h^k]'(z)|^2,  \\
|b_h^k(z)|^2 &= |b_h^{k-1}(z)|^2 + \tau^2 |d_t b_h^k(z)|^2.
\end{split}\]
By induction and incorporation of the stability estimate we deduce the asserted
estimates for the constraint violation. 
\end{proof}

We illustrate the performance of Algorithm~\ref{alg:descent_rods} via a 
numerical experiment showing the relaxation of a twisted initially flat 
curve which is clamped at both ends. 
We plotted in the bottom of Figure~\ref{fig:twist_exp} the total
energy and the torsion contribution defined by 
\[
T_\tor^h[y_h,b_h] = \frac{\ct}{2} \int_0^L |b_h'|^2 \dv{x_1}
- \frac{\ct}{2} \int_0^L (Q_h b_h \cdot y_h'')^2 \dv{x_1}.
\]
We observe that the curve quickly releases its large energy and becomes
a spatial curve attaining a stationary configuration with equilibrated
curvature after approximately~2000 iterations. The model parameters 
used in the simulation were set to $\cb = 2$ and $\ct = 1$. We used
a partition into~1006 subintervals corresponding to a mesh-size
$h=1/80$. The step-size $\tau$ and the penalty parameter $\veps$ were
chosen proportional to $h$. 

\newcommand{\bild}[1]{\fbox{\includegraphics[scale=.16,trim=120 80 95 50,clip]{twist_k_#1.jpg}}\makebox[0ex][r]{\tiny$k=#1$\ }\,\ignorespaces}
\begin{figure}[p]
\bild{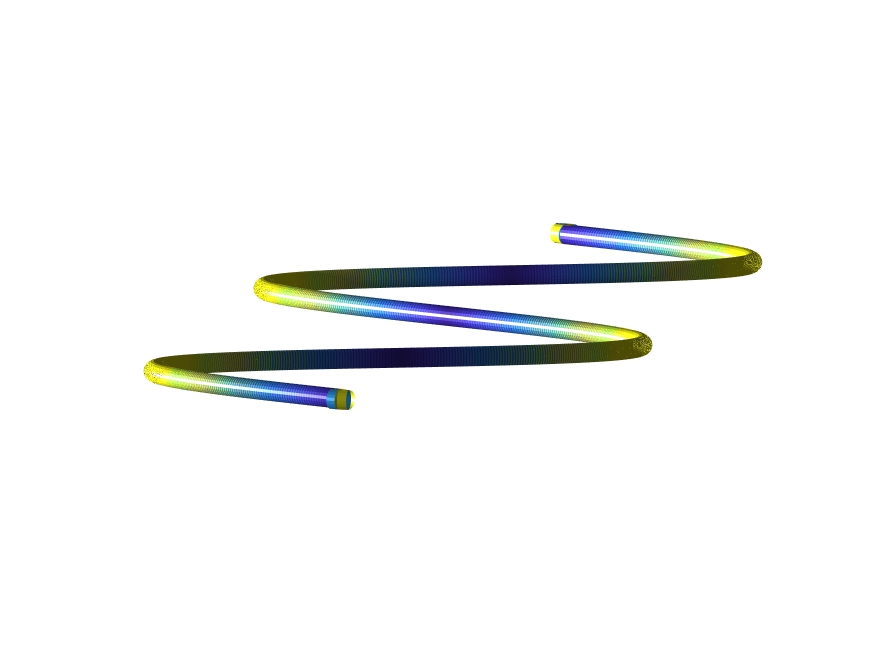} \bild{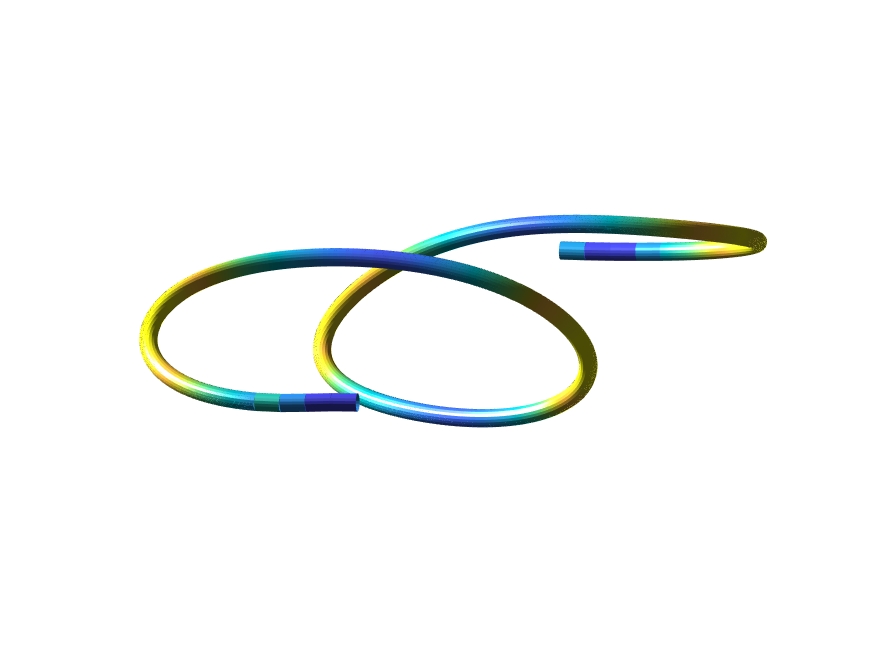} \bild{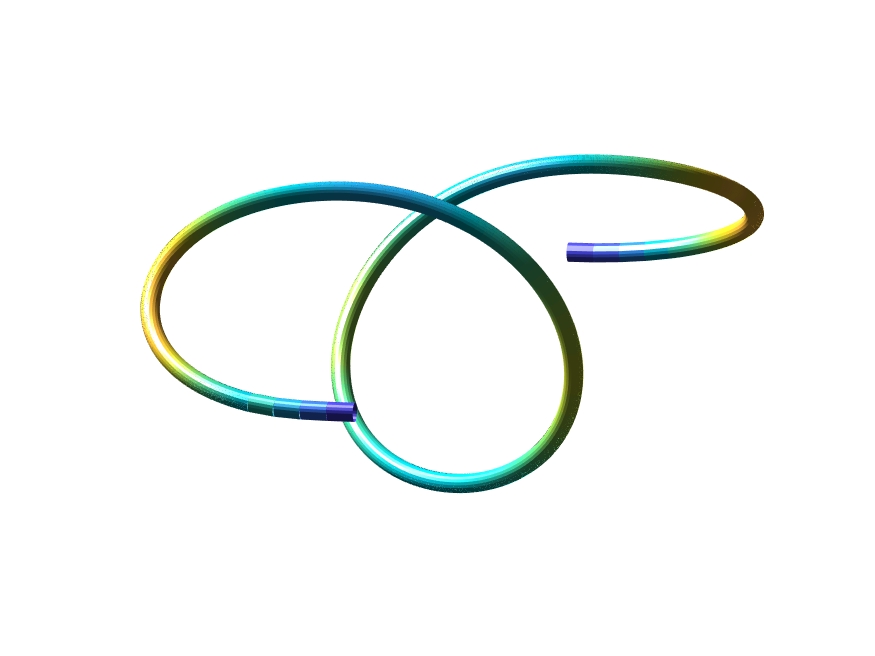} \\
\bild{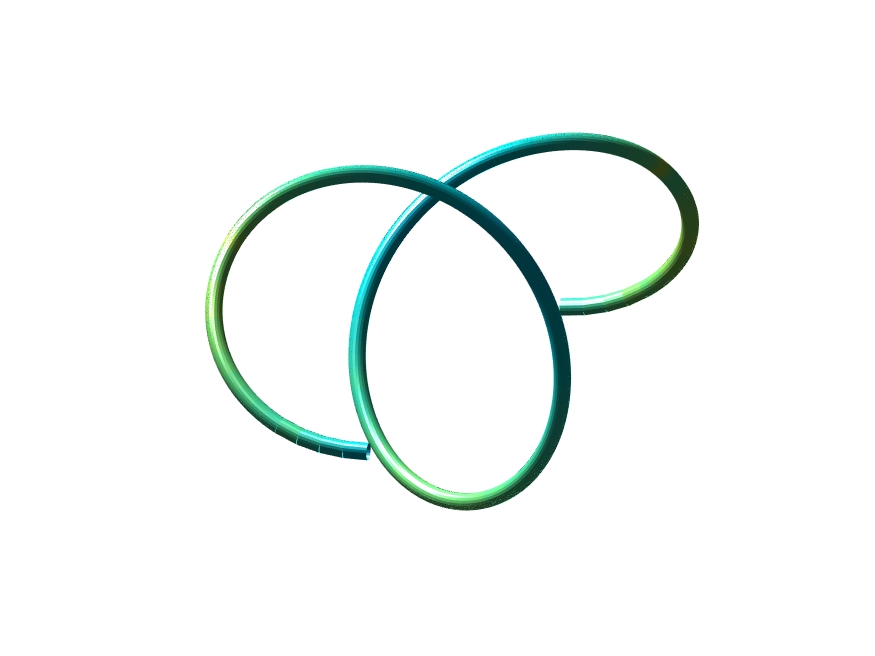}  \bild{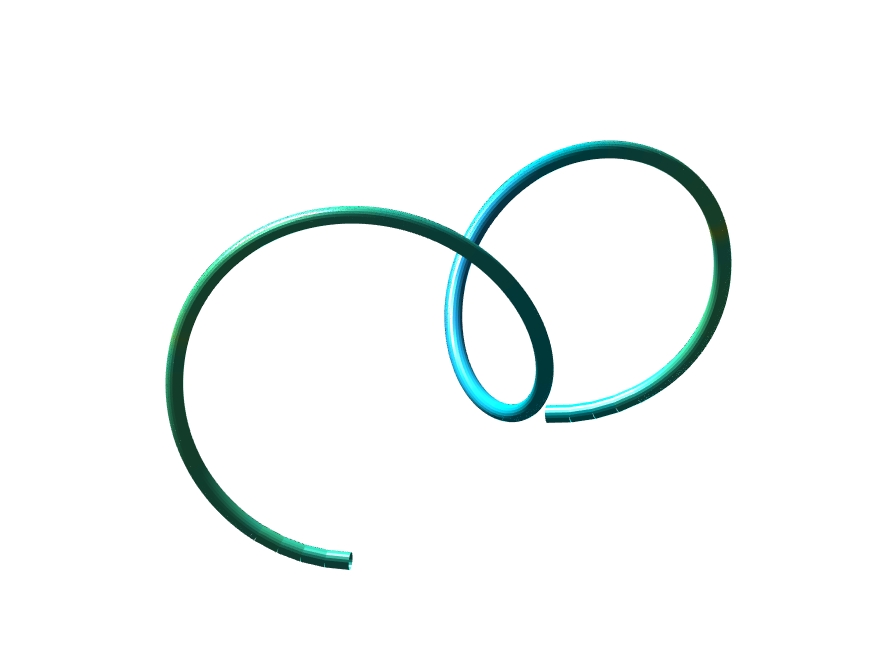} \bild{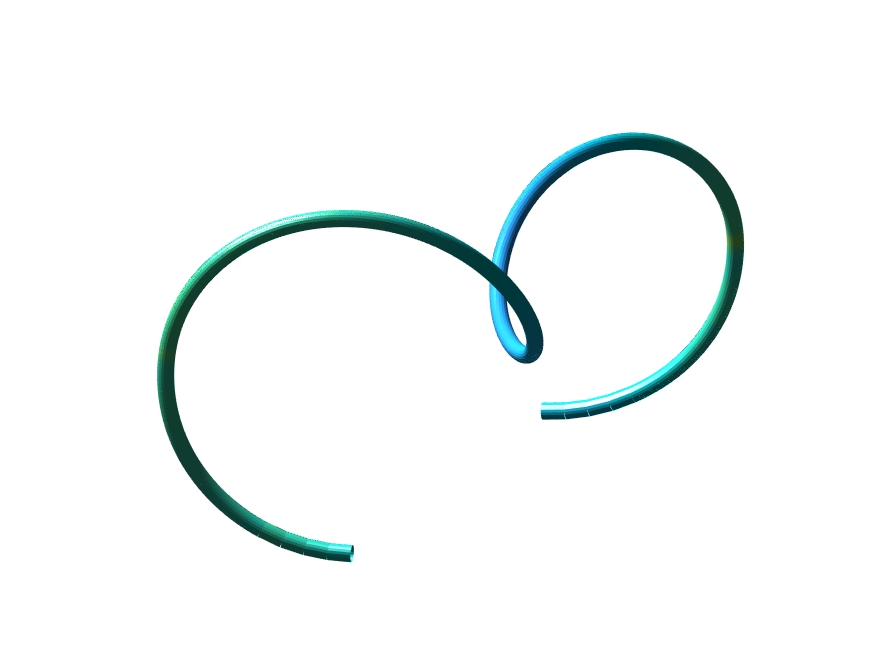} \\
\bild{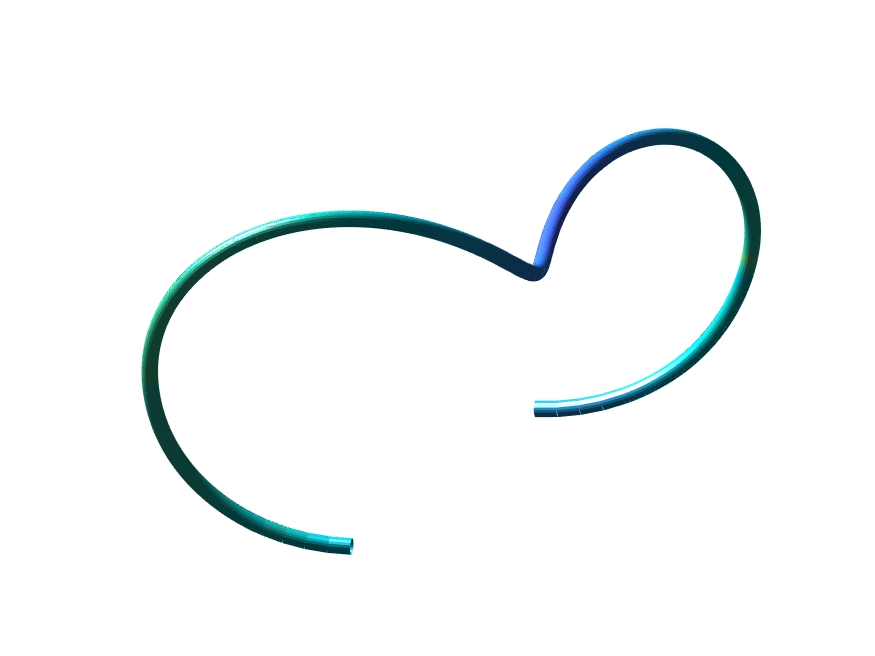} \bild{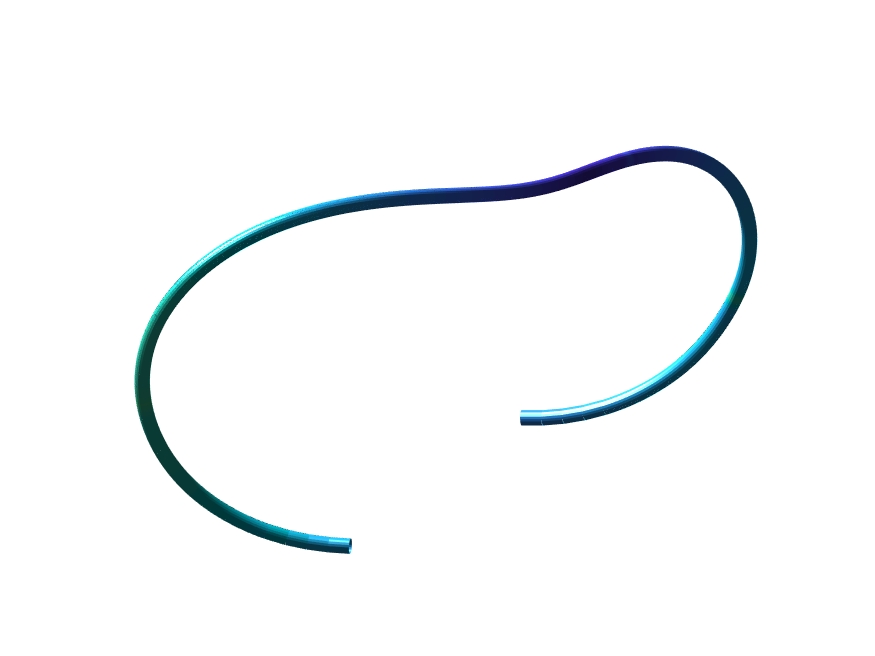} \bild{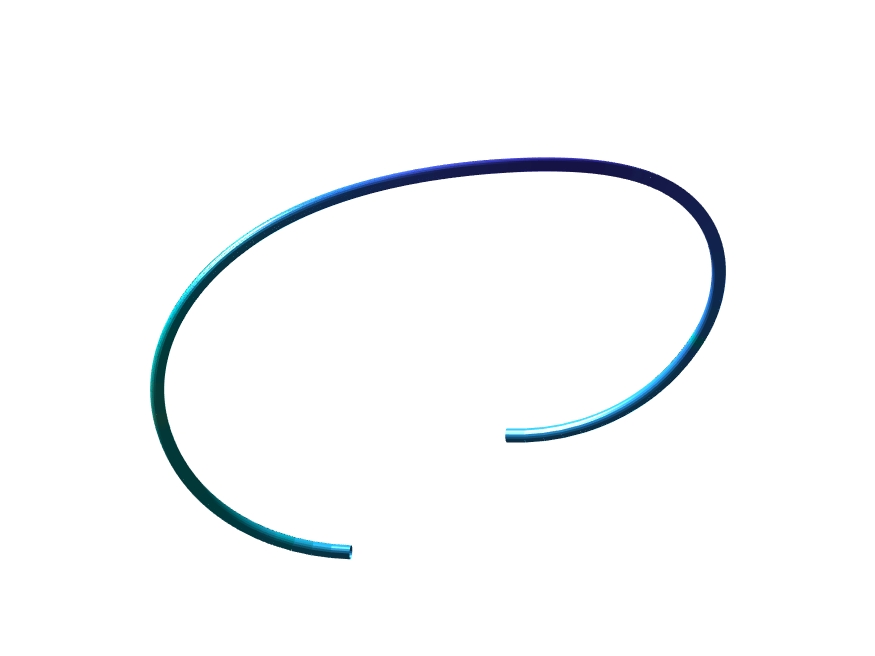}
\vspace*{5mm}
\includegraphics[scale=.58]{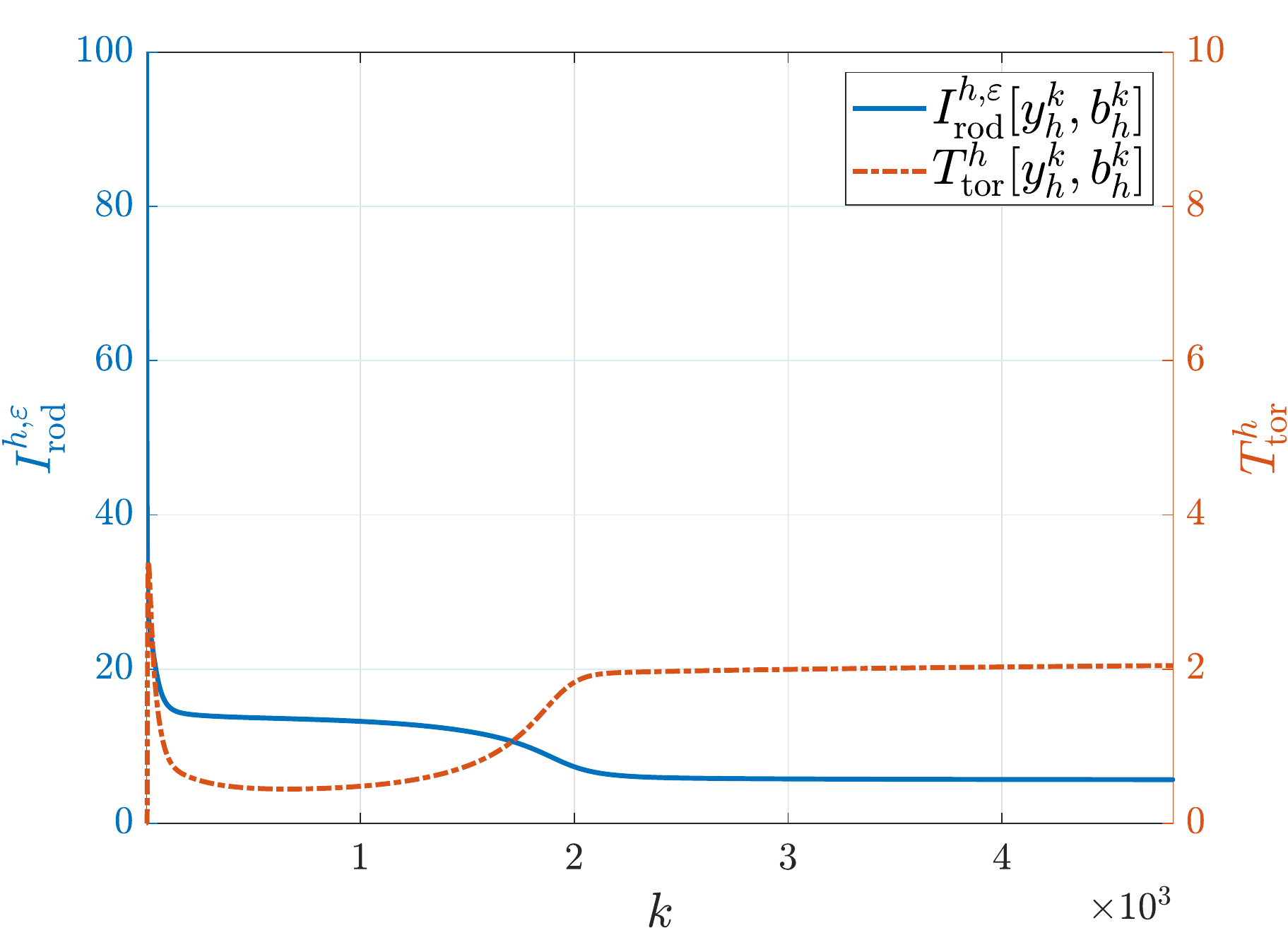}
\caption{\label{fig:twist_exp}
Snapshots of an evolution from an initially flat but twisted curve 
colored by curvature 
after different numbers of iterations (top): the curve equilibrates its
curvature to relax the initially dominant bending energy, afterwards
deforms into a spatial helix, and finally attains a large stationary 
configuration. The total energy decreases monotonically while the
contribution due to torsion increases (bottom).}
\end{figure}

\subsection{Elastic plates}\label{sec:iter_plates}
We next apply the conceptual approach for solving nonlinearly constrained
minimization problems outlined above to the case of approximating bending
isometries. For this we recall that the discrete bending energy is given by 
\[
I_\plate^h [y_h] =\frac{\cb}{2} \int_\o |D_h^2 y_h|^2 \dv{x'}
\]
with the set of admissible discrete deformations defined as
\[\begin{split}
\cA_h = \big\{  y_h\in \cS^\dkt(\cT_h)^3: \, L_\bc^\plate & [y_h] = \ell_\bc^\plate, \\
& [\nabla y_h(z)]^\transp \nabla y_h(z) = \Id_2 \text{ f.a. } z\in \cN_h \big\}.
\end{split}\]
Using the linearized isometry operator
\[
L_{A}^\iso[B] = A^\transp B + B^\transp A 
\]
we define the (shifted) tangent space of the set of discrete isometric deformations $\cA_h$
at a deformation $y_h$ via
\[\begin{split}
\cF_h[y_h] = \big\{ w_h\in \cS^\dkt(\cT_h)^3:   \, & L_\bc^\plate[w_h] = 0, \\
& L^\iso_{\nabla y_h}[\nabla w_h](z) = 0 \text{ f.a. }z\in \cN_h  \big\}.
\end{split}\]
We then decrease the bending energy for given boundary
conditions by iterating the steps of the following algorithm.

\begin{algorithm}[Gradient descent for elastic plates]\label{alg:plates_descent}
Choose an initial $y_h^0 \in \cA_h$ and a step size $\tau>0$, set $k=1$. \\
(1) Compute $d_t y_h^k \in \cF_h[y_h^{k-1}]$ such that
for all $w_h \in \cF_h[y_h^{k-1}]$ we have 
\[
(d_t y_h^k, w_h)_\star + (D_h^2 y_h^k, D_h^2 w_h) = 0.
\]
(2) Stop if $\|d_t y_h^k\|_\star \le \veps_\stop$; otherwise, increase $k\to k+1$
and continue with~(1). 
\end{algorithm}

The iterates $(y_h^k)_{k=0,1,\dots}$ will in general not satisfy the nodal isometry
constraint exactly, but the violation is again independent of the number of iterations
and controlled by the step size~$\tau$.

\begin{proposition}[Convergent iteration]\label{thm:iteration}
The iterates $(y_h^k)_{k=0,1,\dots}$ of Algorithm~\ref{alg:plates_descent}
are well defined and satisfy for every $L\ge 0$ the energy estimate
\[
I_\plate^h [y_h^L] + \tau \sum_{k=1}^L  \|d_t y_h^k\|_\star^2
\le I_\plate^h [y_h^0].
\]
Moreover, if $\|\nabla w_h\|_h \le c_\star \|w_h\|_\star$ for all
$w_h\in \cS^\dkt(\cT_h)^3$ with $L_\bc^\plate[w_h]= 0$, then
we have the constraint violation bound
\[
\max_{k=0,\dots,L} \|(\nabla y_h^k)^\transp \nabla y_h^k - \Id_2 \|_{L^1_h} \le 
c \tau e_{0,h},
\]
where $e_{0,h} = I_\plate^h [y_h^0]$.
\end{proposition}

\begin{proof}[Proof (sketched)]
(a) Since for any $y_h^{k-1}\in \cS^\dkt(\cT_h)^3$ we have that $\cF_h[y_h^{k-1}]$
is a nonempty linear space the Lax-Milgram lemma implies that
there exists a unique solution $y_h^k\in \cF_h[y_h^{k-1}]$. \\
(b) The energy decay property is an immediate consequence of choosing
$w_h =d_t y_h^k$ and using the binomial formula
\[
2 (D_h^2 y_h^k,D_h^2 d_t y_h^k) = d_t \|D_h^2y_h^k\|^2 + \tau \|D_h^2 d_t y_h^k\|^2. 
\]
(c) The error bound for the isometry violation follows from the orthogonality
defined by the linearized isometry condition and the energy decay property
as in the proof of Proposition~\ref{thm:abstract_flow}.
\end{proof}

Figure~\ref{fig:moebius} illustrates the discrete 
evolution defined by Algorithm~\ref{alg:plates_descent} via snapshots of 
different iterates. The clamped boundary conditions imposed at the ends 
$\g_\DD = \{0,L\}\times [0,w]$ of the strip $\o = (0,L)\times (0,w)$ 
with length $L=10$ and width $w=1$
are defined via the operator
\[
L_\bc^\plate[y] = \big[y|_\gD, \nabla y|_\gD\big]
\]
and functions $y_\DD \in L^2(\gD;\R^3)$ and $\hy_\DD \in L^2(\gD;\R^{3\times 2})$.
These functions are constructed in such a way that the segment $\{L\}\times [0,w]$
is rotated and mapped onto the fixed opposite segment $\{0\}\times [0,w]$. In this
way the formation of a M\"obius strip is enforced, we included a small linear forcing
term to avoid certain nonuniqueness effects. We observe that the nonsmooth
choice of the starting value with large bending energy does not influence 
the robustness of the iteration
and that within less than $10.000$ iterations a stationary configuration for 
the stopping parameter $\veps_\stop = 5 \cdot 10^{-3}$ and evolution metric with $\|\cdot\|_\star
= \|D_h^2 \cdot\|$ is 
attained. We also observe that we obtain a satisfactory stationary shape for coarse triangulations
and that the energy decreases monotonically as predicted, with a small violation
of the isometry constraint indicated by the quantity
\[
\d_{\rm iso}^\infty[y_h^k] = \|(\nabla y_h^k)^\transp \nabla y_h^k - I_2 \|_{L^\infty_h}
\]
which appears to be nearly independent of the iteration. We finally remark that
we observe concentrations of curvature at boundary points corresponding to
certain singularities discussed in~\cite{BarHor15}. 

\renewcommand{\bild}[1]{\fbox{\includegraphics[scale=.095,clip]{moebius-#1.png}}\makebox[0ex][r]{\tiny$k=#1$\ }\,\ignorespaces}
\begin{figure}[p]
\bild{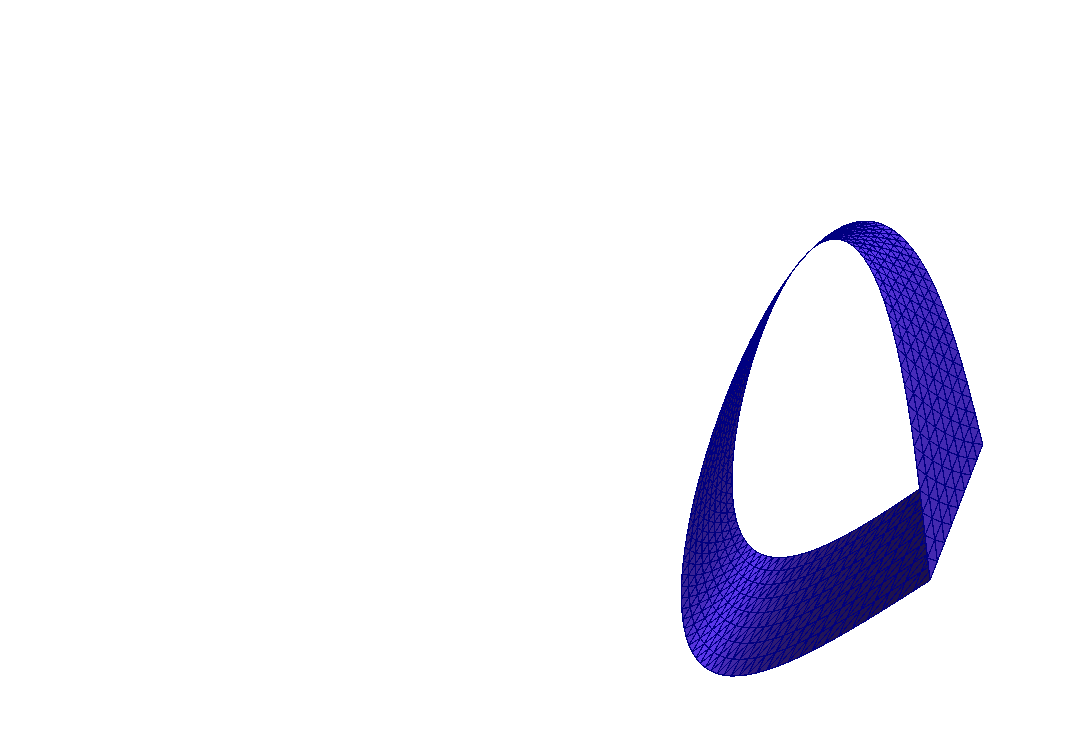} \bild{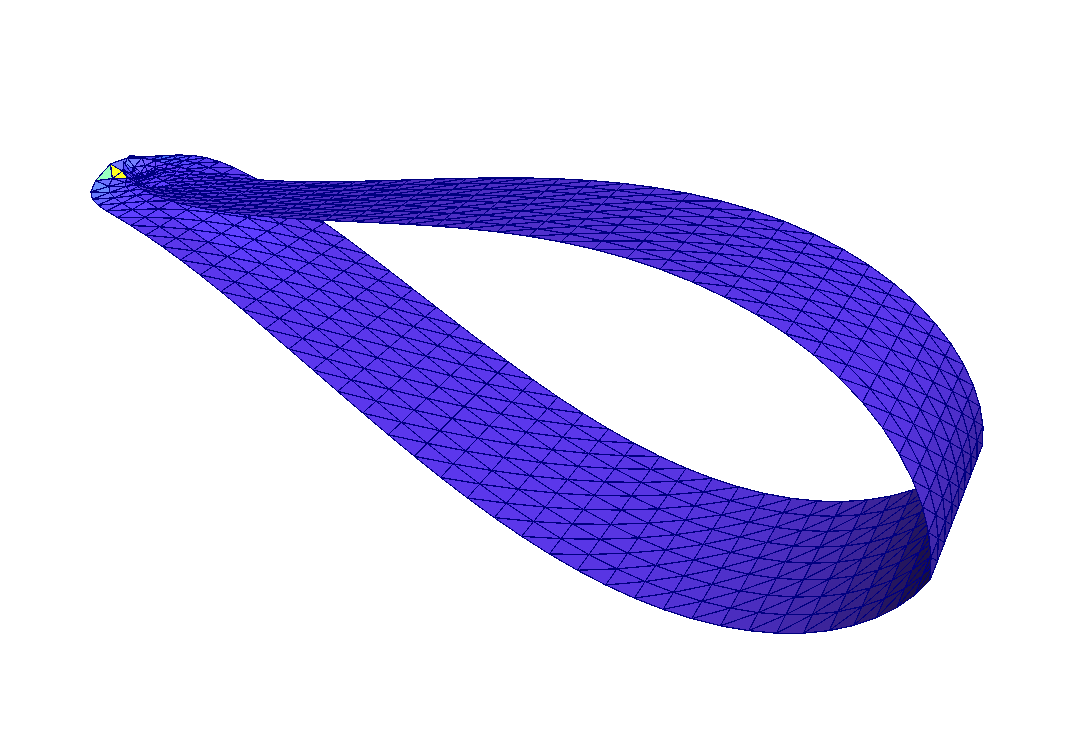} \bild{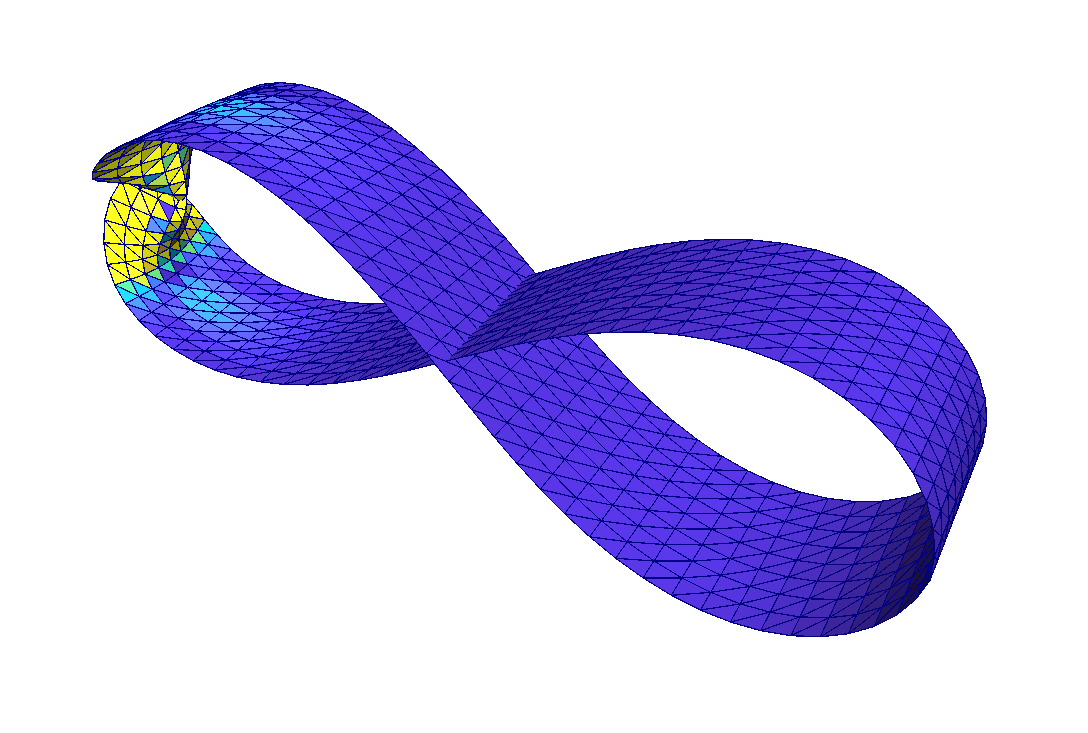} \\ 
\bild{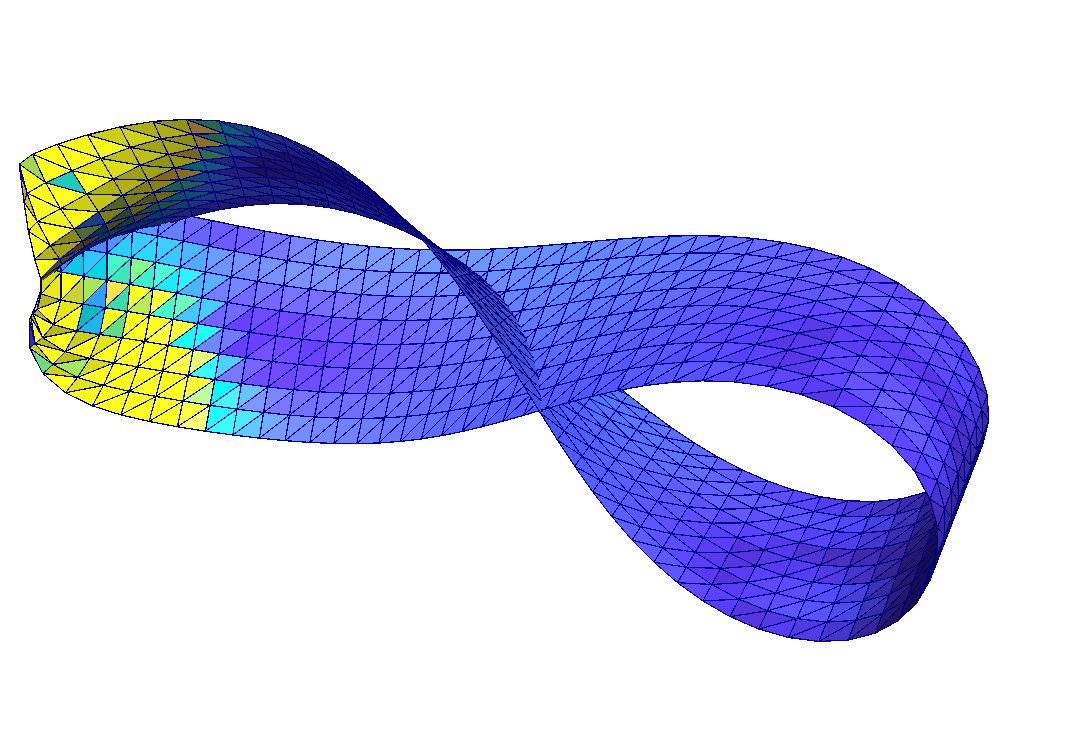} \bild{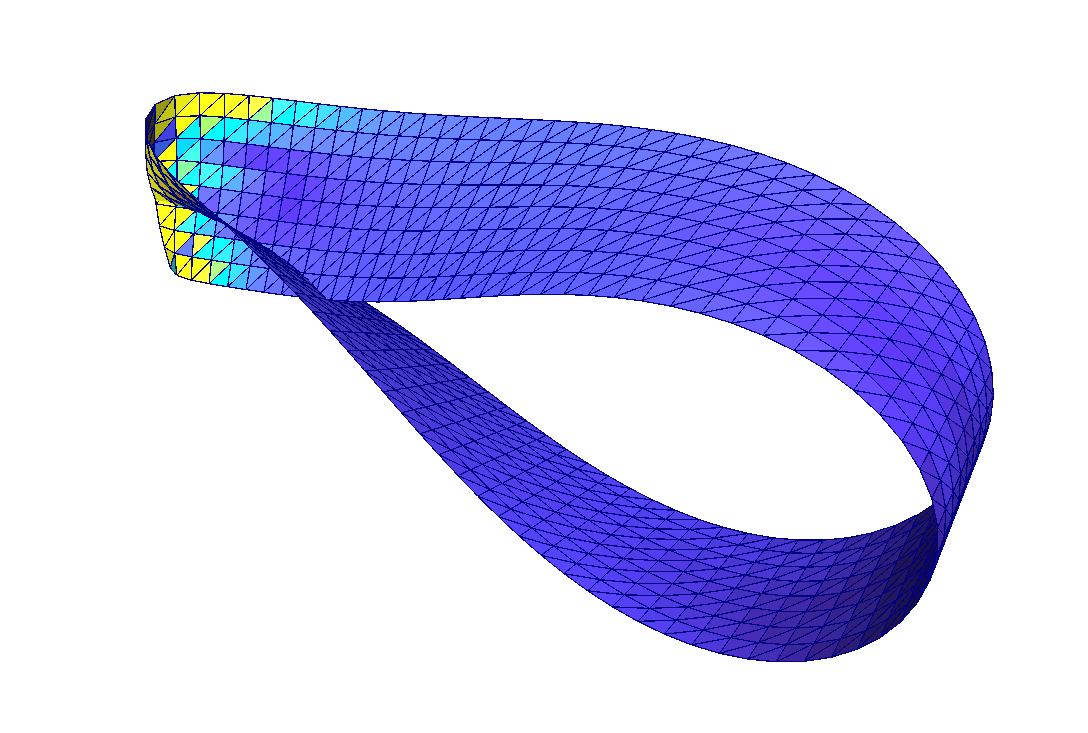} \bild{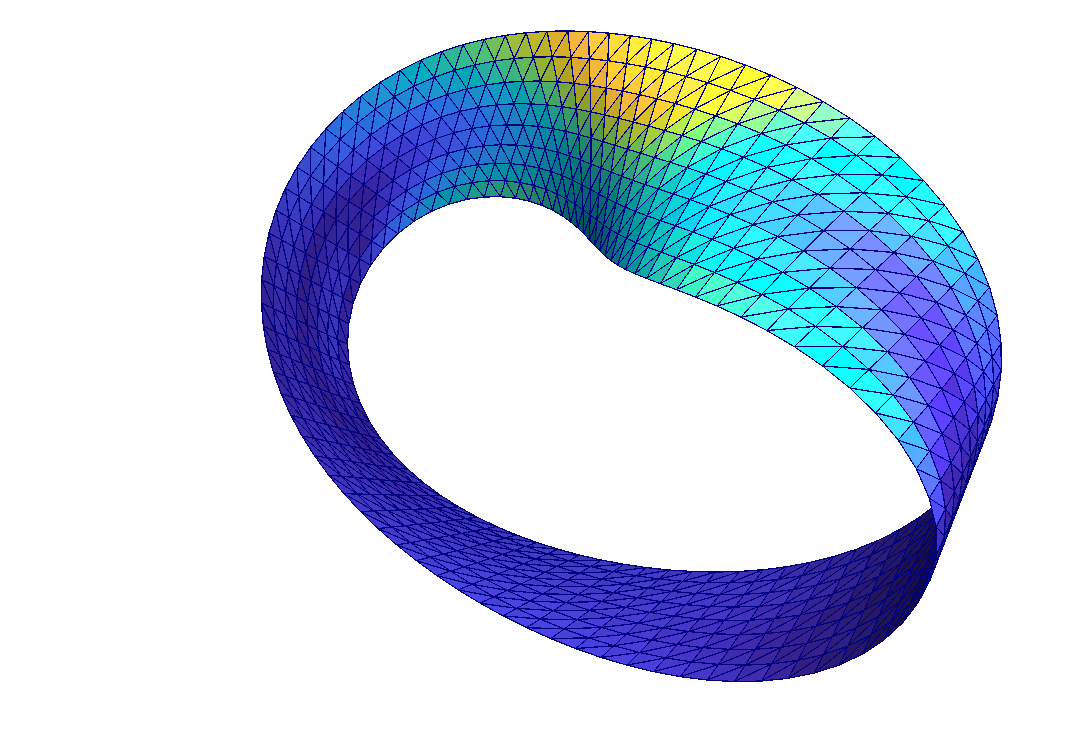}
\vspace*{4mm}
\fbox{\includegraphics[scale=.119,trim=100 -40 95 0,clip]{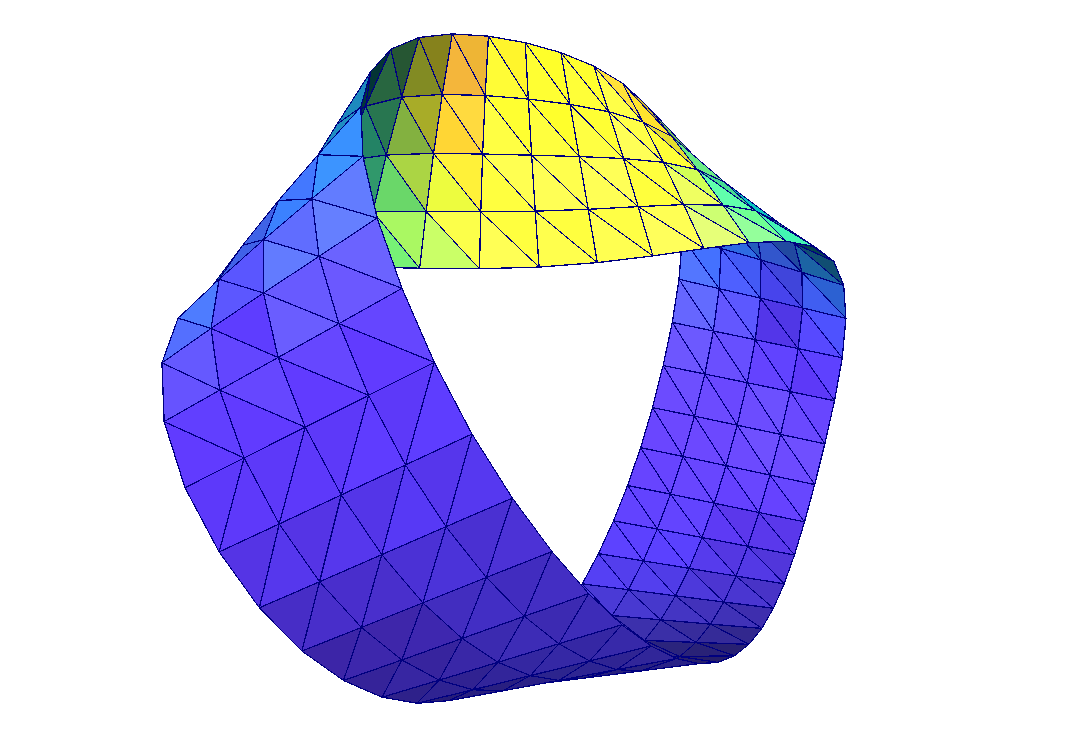}}\makebox[-14ex][r]{\tiny $\#\cT_h=320$\ }\makebox[14ex][r]{\tiny $k=1073$\ }\,\ignorespaces
\fbox{\includegraphics[scale=.11985,trim=130 -35 95 0,clip]{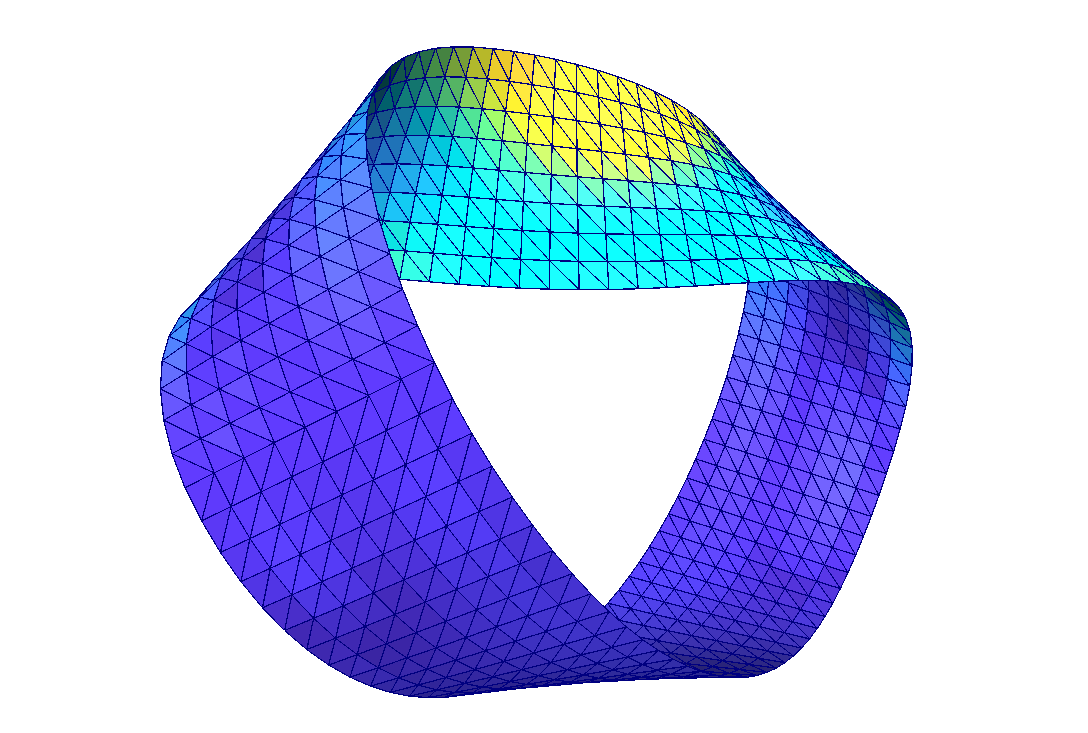}}\makebox[-13ex][r]{\tiny $\#\cT_h=1280$\ }\makebox[13ex][r]{\tiny $k=6606$\ }\,\ignorespaces
\fbox{\includegraphics[scale=.119,trim=100 -40 95 0,clip]{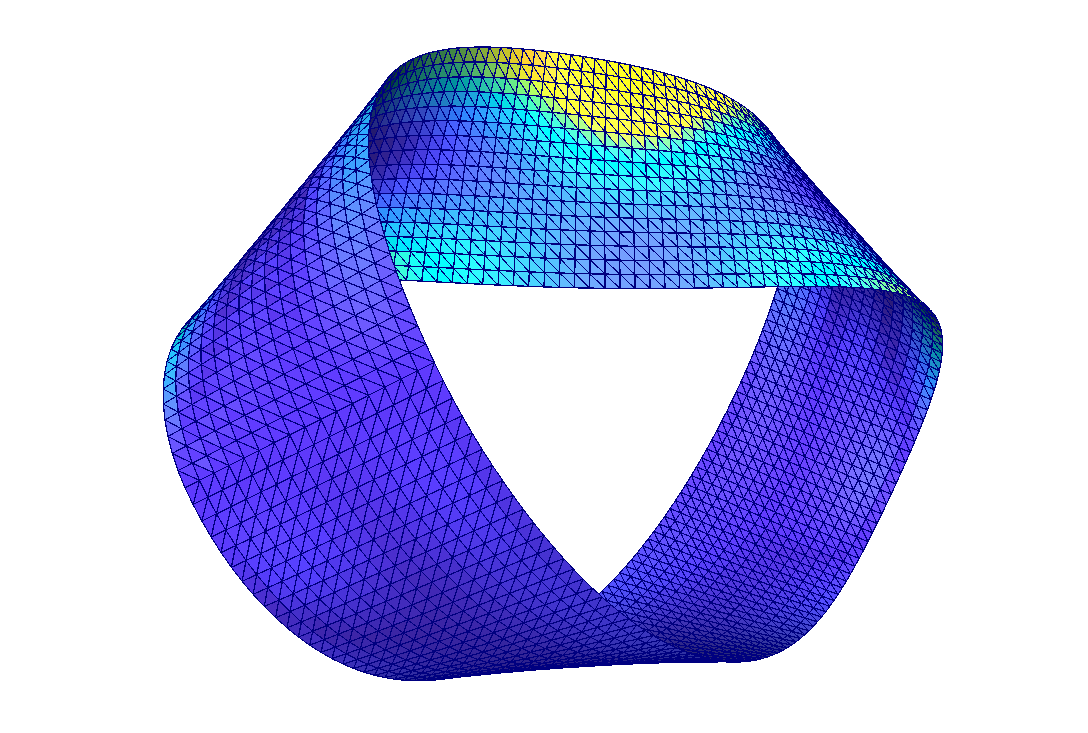}}\makebox[-13ex][r]{\tiny $\#\cT_h=5120$\ }\makebox[13ex][r]{\tiny $k=33276$\ }\,\ignorespaces
\vspace*{4mm}
\includegraphics[scale=.75]{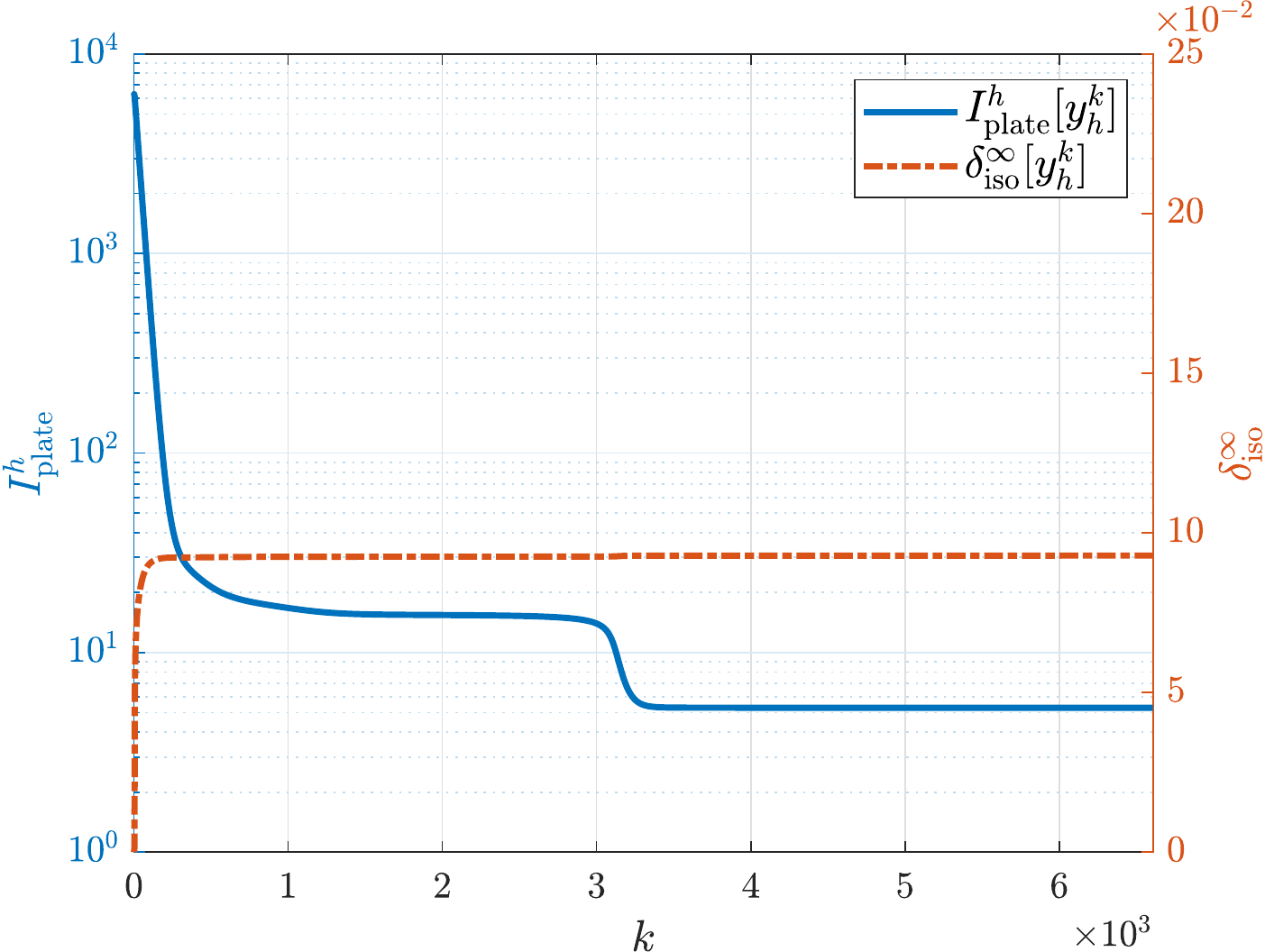}
\caption{\label{fig:moebius}
Snapshots of the iteration to minimize bending energy of an elastic strip with boundary 
conditions leading to the formation of a M\"obius strip colored by its mean curvature (top); stationary
configurations for different triangulations (middle); energy decay
and constraint violation 
throughout the iteration (bottom).}
\end{figure}

\section{Linear finite element systems with nodal constraints}\label{sec:saddle_solve}
The discrete gradient flows devised in the previous sections lead to
linear systems of equations in the time steps that have a special saddle
point structure. In particular, they are given in the form
\begin{equation}\label{eq:saddle_system}
\tag{${\rm S}$} 
\begin{bmatrix} 
A & B^\transp \\ B & 0 
\end{bmatrix}
\begin{bmatrix}
x \\ \lambda 
\end{bmatrix}
= 
\begin{bmatrix}
f \\ 0
\end{bmatrix}
\end{equation}
with a fixed positive definite symmetric matrix $A\in \R^{n\times n}$ and a matrix
$B \in \R^{p\times n}$ that changes in the iteration steps. The matrix $B$ is
of full rank and block diagonal, i.e.,
\[
B = 
\begin{bmatrix}
b_1^\transp   &  &  &   \\
   & b_2^\transp &  &   \\
   &       & \ddots &   \\
   &  &    &        & b_p^\transp
\end{bmatrix}
\]
with vectors $b_i \in \R^\ell\setminus\{0\}$ for $i=1,2,\dots,p$ and $n = p\ell$.
We note that the solution $x\in \R^n$ of the linear system of equations above is 
equivalently characterized by the system 
\[
z^\transp A x = z^\transp f
\]
for all $z\in \R^n$ belonging to the kernel of $B$, i.e., subject to the conditions
\[
x, z\in \ker B.
\]
To obtain a simpler, unconstrained system of equations we choose for each $i=1,2,\dots,p$ 
a set of orthonormal vectors
$(c_2^i,c_3^i,\dots,c_\ell^i) \subset \R^\ell$ that are orthogonal to $b_i$,
i.e., we have
\[
\{b_i\}^\perp = \hull\big\{c_2^i,c_3^i,\dots,c_\ell^i\big\}.
\]
We may then represent the kernel of $B$ by the image of the matrix 
$C\in \R^{p \ell \times p (\ell-1)} $ defined by 
\[
C = 
\begin{bmatrix}
c_2^1 \ \dots \ c_\ell^1  & & & \\
& c_2^2 \ \dots \ c_\ell^2  & & \\
& & \ddots \ \ddots \ \ddots &     & \\
& & & c_2^p \ \dots \ c_\ell^p  
\end{bmatrix}.
\]
The matrix $C$ defines an isomorphism $C:\R^{p (\ell-1)} \to \ker B$ 
and we may thus reformulate the linear system of equations as
\begin{equation}\label{eq:reduced_system}
\tag{${\rm R}$} 
C^\transp A C \hx = C^\transp f.
\end{equation}
With the solution $\hx$ we obtain 
the solution $x$ of the orginal system~\eqref{eq:saddle_system} via $x=C \hx$.
Since the columns of $C$ are linearly independent the matrix $C$ has full rank. 
Hence, the symmetric matrix $\hA = C^\transp A C$ is positive definite 
and the linear system of equations $\hA \hx = \hf$ can be solved efficiently with
a preconditioned conjugate gradient scheme. The construction of suitable
preconditioners has been discussed in~\cite{NasSof96} and the main issue in 
their justification is to understand how accurate the approximation 
\[
(C^\transp A C)^{-1} \approx C^\transp A^{-1} C 
\]
is. Straightforward manipulations show that the solution $x$ of the
saddle-point system~\eqref{eq:saddle_system} satisfies
\[
x = A^{-1} f - A^{-1}B^\transp (B A^{-1} B^\transp)^{-1} B A^{-1} f.
\]
With the equvialent formulation $(C^\transp A C)\hx = C^\transp f$ and
$x= C\hx$ we find that
\[
x = C (C^\transp A C)^{-1} C^\transp f.
\]
Since $C$ is orthogonal in the sense that $C^\transp C = I_{p(\ell-1)}$ it follows that
\[
(C^\transp A C)^{-1} = C^\transp A^{-1} C 
- C^\transp (A^{-1}B^\transp B A^{-1} B^\transp)^{-1} B A^{-1} C.
\]
This justifies the approximation $(C^\transp A C)^{-1} \approx C^\transp A^{-1} C$
if the second term on the right-hand side is small. 
Another relation is obtained by formally applying the Neumann series
\[
T^{-1} = \sum_{j=0}^\infty (I-T)^j
\]
to the product $T = \vrho (C^\transp A C)(C^\transp A^{-1} C)$ with a suitable
parameter $0<\vrho \le 1$ to deduce that 
\[
(C^\transp A C)^{-1} = \vrho (C^\transp A^{-1} C) 
\big(I + (I-T) + (I-T)^2 + \dots\big).
\]
Accepting the approximation $(C^\transp A C)^{-1} \approx C^\transp A^{-1} C$
the next step is to choose a preconditioner $P$ for $A$, i.e., $P$ approximates
$A^{-1}$ and the multiplication by $P$ is inexpensive, and
use the matrix $\hP = C^\transp P C$ as a preconditioner 
for $\hA$. Noting that $C$ is orthogonal this is justified by the
spectral norm estimate from~\cite{NasSof96},
\[
\|C^\transp A^{-1} C - C^\transp P C \| \le \|A^{-1}-P\|.
\]
A further discussion of related preconditioners in the context of 
micromagnetics can be found in~\cite{KPPRS18-pre}.
We illustrate the construction of the bases of the null space and the
performance of different solution strategies in the context of harmonic
maps into spheres. 

\subsection{Application to harmonic maps}
Harmonic maps are vector fields with values in a given manifold which are 
stationary for the Dirichlet energy. In the case of the unit sphere,
the vector field satisfies a pointwise unit-length constraint. Already this simple case
illustrates analytical difficulties when dealing with constrained partial 
differential equations. In particular, harmonic maps are nonunique even for
fixed boundary values and may be discontinuous everywhere, cf.~\cite{Rivi95}. 
Partial regularity results can be proved if a harmonic map is energy 
minimizing, cf.~\cite{SchUhl82}. These observations underline the importance of
computing harmonic maps with low energy. We aim at applying the concepts
of the previous sections to the approximation of harmonic maps and consider
the following model problem defined with a function $u_\DD\in C(\p\O;\R^d)$ with
$|u_\DD(x)|=1$ for all $x\in \p\O$ which is the trace of a function 
$\tu_\DD \in H^1(\O;\R^d)$:
\[
\tag{${\rm P}_\hm$} \left\{ \, 
\begin{array}{l}
\text{Minimize} \quad \displaystyle{\frac12 \int_\O |\nabla u|^2\dv{x}} \quad \text{in the set}\\[2.5mm]
 \cA = \big\{u\in H^1(\O;\R^d): |u(x)|^2 =1 \text{ f.a.e. } x\in \O, \, u|_\pO = u_\DD\big\}.
\end{array}\right.
\]
Following~\cite{Alou97,Bart05} we discretize 
the admissible set $\cA$ by piecewise affine vector fields 
and impose the unit-length constraint in the nodes of the underlying triangulation.
This leads to the following discrete minimization problem:
\[
\tag{${\rm P}_\hm^h$} \left\{  
\begin{array}{l}
\text{Minimize} \quad \displaystyle{\frac12 \int_\O |\nabla u_h|^2\dv{x}} \quad \text{in the set} \\[2.5mm]
\cA_h = \big\{u_h\in \cS^{1,0}(\cT_h)^d: |u_h(z)|^2 =1 \text{ f.a.}\, z\in \cN_h, u_h|_\pO = u_{\DD,h}\big\}.
\end{array}\right.
\]
For a function $u_h\in \cS^{1,0}(\cT_h)^d$ with nonvanishing nodal values we define
\[
\cF_h[u_h] = \big\{w_h \in \cS^{1,0}(\cT_h)^d: \, 
w_h(z)\cdot u_h(z) = 0 \text{ f.a. }z\in \cN_h, \ w_h|_{\p\O} = 0 \big\}.
\]
With a gradient flow approach and a linearized treatment of the nodal constraint
we are led to the following algorithm.

\begin{algorithm}[Gradient descent for harmonic maps]\label{alg:hms}
Choose $u_h^0\in \cA_h$ and $\tau>0$, set $k=1$. \\
(1) Compute $d_t u_h^k \in \cF_h[u_h^{k-1}]$ such that for all $w_h\in \cF_h[u_h^{k-1}]$
we have 
\[
(d_t u_h^k,w_h)_\star + (\nabla u_h^k,\nabla w_h) = 0.
\]
(2) Stop if $\|d_t u_h^k\|_\star \le \veps_{\rm stop}$; otherwise, increase 
$k\to k+1$ and continue with~(1).
\end{algorithm}

The Lax--Milgram lemma shows that the iteration is well defined. Because of the 
orthogonality condition we have that 
\[\begin{split}
|u_h^k(z)|^2 = |u_h^{k-1}(z) + \tau d_t u_h^k(z)|^2 
\ge |u_h^{k-1}(z)|^2
\end{split}\]
which by induction gives $|u_h^k(z)| \ge \dots \ge |u_h^0(z)|=1$ for $k=0,1,\dots,K$. 
If the restriction of the inner product $(\cdot,\cdot)_\star$ to $\cS^{1,0}(\cT_h)^d$ is
represented by the matrix $M$ then the linear systems in the iteration can be written as
\[
\begin{bmatrix} 
M+\tau S & B^\transp \\
B & 0 
\end{bmatrix}
\begin{bmatrix} 
V^k \\ \Lambda
\end{bmatrix}
=
\begin{bmatrix} 
- S U^{k-1} \\ 0
\end{bmatrix}
\]
with the vectorial finite element element stiffness matrix $S$ and the
constraint matrix 
\[
B = 
\begin{bmatrix}
u_h^{k-1}(z_1)^\transp &             &    &  \\
               & u_h^{k-1}(z_2)^\transp &    &  \\
               &                       & \ddots &  \\
               &                         &      & u_h^{k-1}(z_p)^\transp 
\end{bmatrix}.
\]
To define the matrix $C$ that provides an isomorphism from $\R^{p(d-1)}$ onto the 
kernel of $B$ we need to compute for a given vector $b\in \R^d\setminus \{0\}$
an orthonormal basis $(c_2,\dots,c_d)$ of its orthogonal complement. We proceed as follows:
if $b$ is parallel to $e_1$ then we choose the remaining canonical basis vectors 
$(e_2,\dots,e_d)$; otherwise, we set 
\[
(\tc_2,\dots,\tc_d) = \begin{cases}
b^\perp & \text{if } d=2, \\
\big(b\times e_1,b\times (b\times e_1)\big) & \text{if } d=3,
\end{cases}
\]
where $b^\perp$ denotes the rotation of $b$ by $\pi/2$, followed by the
normalization $c_j = \tc_j/|\tc_j|$ for $j=2,\dots,d$.

\begin{example}
Let $\O = (-1/2,1/2)^3$ and define $u_\DD$ on $\GD = \p\O$ for $x\in \p\O$ via
\[
u_\DD(x) = \frac{x}{|x|}.
\]
The employed triangulations $\cT_\ell$ result from $\ell$ uniform 
refinements of a reference triangulation of $\O$ into six tetrahedra. The 
starting value $u_h^0 \in \cS^{1,0}(\cT_\ell)^3$ is defined by generating
random nodal values of unit length. 
\end{example}

Table~\ref{tab:exp_hms_iter} compares the following general strategies
to solve the linear problems in the iterative solution of harmonic maps:
\begin{itemize}
\item[(1)] Solve the original indefinite saddle-point system~\eqref{eq:saddle_system}
with a direct solution method for sparse systems. 
\item[(2)] Solve the reduced positive definite system~\eqref{eq:reduced_system} with a 
direct solution method for sparse linear systems. 
\item[(3)] Solve the reduced positive definite system~\eqref{eq:reduced_system} with the 
preconditioned conjugate 
gradient method using a diagonal preconditioning of the full system matrix $C^\transp A C$.
\item[(4)] Solve the reduced system~\eqref{eq:reduced_system} with the preconditioned conjugate 
gradient method using the preconditioner 
$C^\transp (L_\ic L_\ic^\transp)^{-1} C$ with the incomplete 
Cholesky factorization $A \approx L_\ic L_\ic^\transp$ computed once. 
\end{itemize}
We measured the average time needed for different discretizations to
solve the linear systems of equations using strategies~(1)-(4). We always 
chose the step 
size $\tau = h$, the stopping parameter $\veps_{\rm pcg} = 10^{-8}$ 
for the relative residuals in the preconditioned conjugate gradient
scheme, and the stopping criterion $\veps_{\rm stop} = h_\ell/10$ in
Algorithm~\ref{alg:hms}. We employed the $H^1$ inner product to define
 $(\cdot,\cdot)_\star$ and used {\sc Matlab}'s backslash operator as a model 
direct solver for sparse linear systems. From 
the obtained numbers we see that the preconditioner particularly
designed for the structure of the constrained problems with changing 
constraints clearly outperforms all other approaches. However, it
does not lead to mesh-size independent iteration numbers for this example
with a nonsmooth solution. We also remark that
we did not observe a further improvement when additional terms in the
Neumann series were used with a damping factor $\vrho = 1$. In a two-dimensional
setting with a smooth solution the complete Cholesky factorization led 
to mesh-independent iteration numbers. 

\sisetup{round-precision=1,round-mode=places,scientific-notation=true}
\begin{table}
{\small \begin{tabular}{|r|c|c|c|c|} \hline
\mbox{} & saddle, direct & reduced, direct & reduced, pcg                  & reduced, pcg    \\
$\#\cN_\ell$  & (backslash)    & (backslash)     & (diagonal $C^\transp A C$) & (incompl. Chol.) \\\hline\hline  
$125$    &  \num{0.0007388} & \num{0.0003246} & \num{0.0000835} \hspace*{1.1mm}(11.9) & \num{0.0000300} \hspace*{1.1mm}(6.0)  \\\hline
$729$    &  \num{0.0150466} & \num{0.0064836} & \num{0.0001906} \hspace*{1.1mm}(25.4)  & \num{0.0000743} (10.2) \\\hline
$4913$   &  \num{0.5961178} & \num{0.2436273} & \num{0.0009204} \hspace*{1.1mm} (52.5) & \num{0.0003884} (19.9) \\\hline 
$35937$  &  \num{34.472256} & \num{8.2637542} & \num{0.0081063} (105.9)                & \num{0.0032082} (38.8) \\\hline
$274625$ &    --  & --                        & \num{0.1094539} (212.1)                & \num{0.0585564} (76.1) \\\hline
\end{tabular}}  \vspace*{2mm}
\caption{\label{tab:exp_hms_iter} Average runtimes in seconds and average iteration 
numbers of the preconditioned conjugate gradient scheme (pcg) per iteration in parentheses 
for solution strategies~(1)-(4) for the constrained linear systems of equations
arising in the iterative approximation of harmonic maps on triangulations 
$\cT_\ell$ of the unit cube with mesh-sizes $h_\ell \approx 2^{-\ell}$, 
$\ell =2,3,\dots,6$.}
\end{table}

\section{Applications, modifications, and extensions}\label{sec:extensions}
We address in this section the application of the developed methods to the
simulation of extended problems. In particular, we devise a numerical method
for simulating bilayer plates, discuss how injectivity
can be enforced in the case of elastic rods, and investigate a model that involves the
thickness parameter and thereby allows for describing situations in which low energy
membrane and bending effects occur simultaneously. 

\subsection{Bilayer plates}
An important class of applications of nonlinear plate bending is related 
to composite materials, i.e., structures such as bilayer plates which are made of two
sheets with slightly different mechanical features that are glued on top of each other.
The difference in the material properties, e.g., concerning the response to 
environmental changes such as temperature, allows for externally controlled large 
bending effects. In bimetal strips one of the metals contracts while the other
one expands leading in combination to the formation of rolls with prescribed curvature. 

\begin{figure}[h]
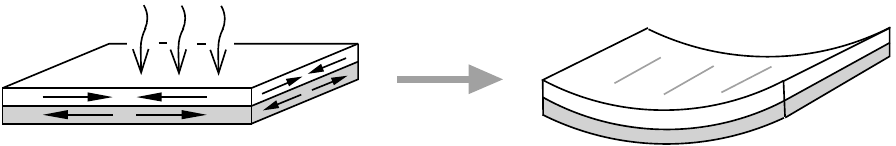
\caption{\label{fig:bilayer_sketch} Schematical description of the bilayer
bending effect: the upper sheet contracts while the lower one expands leading
in combination to a large isometric deformation.}
\end{figure}

The preferred curvature is defined by the difference in elastic properties of
the involed materials and is explicitly visible in the dimensionally reduced model
identified in~\cite{Schm07a}. Following that work we consider a thin plate 
$\O_\d = \o \times (-\d/2,\d/2)$ and the 
inhomogeneous energy density 
\[
W_\bil(x_3,F) = \begin{cases}
W \big((1+\d \a)^{-1} F\big) & \text{ for }x_3>0, \\
W \big((1-\d \a)^{-1} F\big) & \text{ for }x_3<0,
\end{cases}
\]
with a material parameter $\a\in \R$. 
The energy density $W$ is assumed to satisfy the standard requirements stated
in Section~\ref{sec:dim_red} so that $W_\bil$  
is minimal for deformation gradients $F$ which are multiples of rotations
with $\det F = (1+\d \a)^3>1 $ if $x_3>0$
and $\det F = (1-\d \a)^3 <1$ for $x_3<0$ corresponding to expansive and
contractive behavior in the upper and lower sheets, respectively. 
It has been shown in~\cite{Schm07a}
that for deformations with energy proportional to $\d^3$
a limiting, dimensionally reduced problem is given by the following 
minimization problem:
\[
\tag{${\rm P}_\bil$} 
\left\{ \, 
\begin{array}{l}
\text{Minimize} \quad \displaystyle{I_\bil[y] = \int_\o Q_\plate(II-\a I_2) \dv{x'}} \quad \text{in the set}\\[2.5mm]
\cA = \big\{y\in H^2(\o;\R^3): (\nabla y)^\transp \nabla y = I_2, \, L_\bc^\bil[y] = \ell_\bc^\bil \big\}.
\end{array}\right.
\]
Here, $Q_\plate$ is the reduced quadratic form introduced in Section~\ref{sec:dim_red_plates}
which is applied
to the difference of the second fundamental form $II$ and a multiple of the 
identity matrix with a factor given by half of the jump in the material 
difference divided by $\d$.
This material mismatch introduces a term that is similar to so-called 
spontaneous curvature terms in the context of biomembranes described via 
the Helfrich-Willmore model. 
Note that in the definition of $W_\bil$ the difference is $2\d \a$ and
its proportionality to the thickness $\d$ of the plates is important 
to avoid delamination effects. The reduced model prefers deformations $y$ 
that define surfaces with mean curvature $\a$ 
and vanishing Gaussian curvature due to the isometry constraint. 
Hence, rolling effects in one direction occur. In~\cite{Schm07b} global 
minimizers for $I_\bil$ have been shown to be given by cylinders with radius
$r=1/\a$. Recalling that the quadratic form $Q_\plate$ is given by the scalar 
product
\[
\langle M, N \rangle_\plate = c_1 M:N  + c_2 \trace(M) \trace(N),
\]
we do not have the property that the integrand is proportional to $|D^2y|^2$ as in the
case of a single layer or two identical materials 
corresponding to the case $\a=0$. Instead we infer
\[
Q_\plate(II-\a I_2) = Q_\plate(II) -2 \a \langle II, I_2 \rangle_\plate + \a^2 Q_\plate(I_2).
\]
Because of the isometry condition the first term is proportional to $|D^2y|^2$ 
while the third term is a constant $c_\a$ which is irrelevant for the minimization 
process. For the second term on the right-hand side we have 
\[
-2 \a \langle II, I_2 \rangle_\plate = -2\a (c_1 + 2 c_2) \trace(II).
\]
Therefore, using $II_{ij} = -\p_i \p_j y \cdot (\p_1 y\times \p_2 y)$ we have 
\[
Q_\plate(II-\a I_2) = \cb |D^2 y|^2 + 2\a \cspont \Delta y \cdot (\p_1 y \times \p_2 y)
+ c_\a.
\]
Convergent finite element discretizations and iterative strategies for the 
numerical solution of the bilayer plate bending
problem have been devised in the articles~\cite{BaBoNo17,BBMN18} extending the ideas
described in the Sections~\ref{sec:fem_plates} and ~\ref{sec:iter_plates}. The iterative 
scheme of~\cite{BaBoNo17} is
unconditionally stable but requires a subiteration, i.e., the solution of 
a nonlinear system of equations in every time step, which limits its practical
applicability, while the scheme used in~\cite{BBMN18} is explicit and efficient
but can only be expected to be conditionally stable. We devise here 
a new semi-implicit scheme with improved stability properties. The set of admissible
deformations $\cA_h$ and the tangent spaces $\cF_h[y_h]$ are defined
as in Section~\ref{sec:fem_plates} and we set 
\[
\Delta_h v_h = \trace D_h^2 v_h
\]
for a function $v_h\in \cS^\dkt(\cT_h)$ with a componentwise application in
the case of a vector field. The discretized energy functional describing
actuated bilayer plates is therefore given by
\[
I_\bil^h[y_h] = \frac{\cb}{2} \int_\o |D_h^2 y_h|^2 \dv{x'}
+ \a \cspont\int_\o \hcI_h^{1,0} \big[ \Delta_h y_h \cdot (\p_1 y_h \times \p_2 y_h)\big] \dv{x'}.
\]
A uniform discrete coercivity property of this discretization has been 
established in~\cite{BaBoNo17}.
The new iterative minimization from~\cite{BarPal19-in-prep-pre}
uses an explicit treatment of the spontaneous
curvature term. 

\begin{algorithm}[Gradient descent for bilayer plates]\label{alg:bilayer_plates}
Choose $y_h^0\in \cA_h$ and $\tau>0$, set $k=1$. \\
(1) Compute $d_t y_h^k \in \cF_h[y_h^{k-1}]$ such that for all $w_h\in  \cF_h[y_h^{k-1}]$ we have
\begin{multline*}
(d_t y_h^k, w_h)_\star + \cb (D_h^2 y_h^k,D_h^2 w_h)  
= - \a \cspont\int_\o \hcI_h^{1,0} \big[ \Delta_h w_h \cdot (\p_1 y_h^{k-1} \times \p_2 y_h^{k-1})\big] \dv{x'} \\
 \quad - \a \cspont \int_\o \hcI_h^{1,0}\big[\Delta_h y_h^{k-1} \cdot (\p_1 y_h^{k-1} \times \p_2 w_h 
+ \p_1 w_h \times \p_2 y_h^{k-1}) \big]\dv{x'}.
\end{multline*}
(2) Stop if $\|d_t y_h^k\|_\star \le \veps_{\rm stop}$; otherwise 
increase $k\to k+1$ and continue with~(1).
\end{algorithm}

We show that the iteration of Algorithm~\ref{alg:bilayer_plates} is energy stable 
on finite time intervals under a moderate
condition on the step-size resulting from the use of an inverse estimate.
To formulate it, we first note that we assume that the boundary conditions
imply a Poincar\'e type estimate
\[
\|\nabla w_h\|_h \le c_{\rm P} \| D^2_h w_h\|
\]
for all $w_h \in \cS^\dkt(\cT_h)^3$ with $L_\bc^\bil[w_h] =0$. With this 
estimate one verifies that the mesh-dependent estimate
\[
\|\nabla w_h \|_{L^\infty_h} \le c_{\rm inv} |\log \hh | \|D_h^2 w_h \|_h,
\]
holds for all $w_h\in \cS^\dkt(\cT_h)^3$ with $L_\bc^\bil[w_h] =0$
and with the minimal mesh-size~$\hh$. 

\begin{proposition}[Convergent iteration]
The iterates of Algorithm~\ref{alg:bilayer_plates} are well defined. 
Assume that we have
\[
\|D_h^2 w_h\|_h \le c_\star \|w_h\|_\star
\]
for all $w_h \in \cS^\dkt(\cT_h)^3$ with $L_\bc^\bil[w_h] = 0$. 
Then there exist constants $c_\bil,c_\bil'>0$ such that 
if $\tau|\log \hh|^2 \le c_\bil'$ for $L=0,1,\dots,K\le T/\tau$ we have 
\[
I_\bil^h[y_h^L] + (1-c_\bil\tau|\log \hh|)\tau  \sum_{k=1}^L \|d_t y_h^k\|_\star^2 
\le I_\bil^h [y_h^0],
\]
and 
\[
\max_{k=0,\dots,L} \|[\nabla y_h^k]^\transp \nabla y_h^k - I_2 \|_{L^\infty_h} 
\le 2 \tau  c_{\rm inv}^2 |\log \hh|^2 e_{0,h},
\]
where $e_{0,h} = I_\bil^h [y_h^0]$.
\end{proposition}

\begin{proof}
We argue by induction and assume that the estimate has been established
for $L-1$. Choosing $w_h = d_t y_h^k$ and using H\"older's inequality 
shows that for $k\le L$ we have
\[\begin{split}
\| & d_t y_h^k \|_\star^2 + \frac{\cb}{2} d_t \|D_h^2 y_h^k\|^2 \\
&\le \a \cspont \|\Delta_h d_t y_h^k\|_h \|\p_1 y_h^{k-1}\|_h \|\p_2  y_h^{k-1}\|_{L^\infty_h} \\
& \ + \a \cspont  \|\Delta_h y_h^{k-1}\|_h 
\big(\|\p_1 y_h^{k-1}\|_{L^\infty_h} \|\p_2 d_t y_h^k\|_h  + \|\p_2 y_h^{k-1}\|_{L^\infty_h} \|\p_1 d_t y_h^k\|_h\big).
\end{split}\]
Absorbing terms involving $d_t y_h^k$ on the left-hand side and 
noting that terms involving $y_h^{k-1}$ are bounded because of the 
bounds for $k\le L-1$, we find that
\[
\frac{\tau}{2} \|d_t y_h^k\|_\star^2 + \frac{\cb}{2} \|D_h^2 y_h^k\|_h^2 
\le \frac{\cb}{2} \|D_h^2 y_h^{k-1}\|_h^2 + \tau c
\le c'.
\]
Here, we also used that the discrete energy is coercive. With this
intermediate estimate and the orthogonality condition encoded in the 
definition of $\cF_h[y_h^{k-1}]$ we infer with $y_h^k = y_h^{k-1}+ \tau d_t y_h^k$
that
\[
[\nabla y_h^k(z)]^\transp \nabla y_h^k(z) 
= [\nabla y_h^{k-1}(z)]^\transp \nabla y_h^{k-1}(z) 
+ \tau^2 [\nabla d_t y_h^k(z)]^\transp \nabla d_t y_h^k(z) 
\]
and hence with the inverse estimate we obtain the suboptimal estimate that 
\[
\|[\nabla y_h^k]^\transp \nabla y_h^k - I_2 \|_{L^\infty_h}
\le c \tau |\log \hh|^2 e_{0,h} \le c''.
\]
To obtain the energy bound we note that a discrete product rule shows that 
the discrete time derivative of the spontaneous curvature term is given by 
\[\begin{split}
d_t & \int_\o \hcI_h^{1,0} \big[ \Delta_h y_h^k \cdot (\p_1 y_h^k \times \p_2 y_h^k) \big] \dv{x'} \\
& = \int_\o \hcI_h^{1,0} \big[ \Delta_h d_t y_h^k \cdot (\p_1 y_h^k \times \p_2 y_h^k) \big] \dv{x'} \\
& \qquad + \int_\o \hcI_h^{1,0} \big[ \Delta_h y_h^{k-1} \cdot (\p_1 y_h^{k-1} \times \p_2 d_t y_h^k 
+ \p_1 d_t y_h^k \times \p_2 y_h^k \big] \dv{x'}.
\end{split}\]
Comparing this expression with the right-hand side of the equation 
in Step~(1) of Algorithm~\ref{alg:bilayer_plates} with $w_h = d_t y_h^k$ leads to 
\[\begin{split}
\|d_t y_h^k & \|_\star^2 + d_t I_\bil^h[y_h^k] + \tau \frac{\cb}{2} \|d_t D_h^2 y_h^k\|^2 \\
& =  \a \cspont \int_\o \hcI_h^{1,0} \big[ \Delta_h d_t y_h^k \cdot 
\big(\p_1 y_h^k \times \p_2 y_h^k- \p_1 y_h^{k-1} \times \p_2 y_h^{k-1}\big) \big] \dv{x'} \\
& \qquad + \a \cspont \int_\o \hcI_h^{1,0} \big[ \Delta_h y_h^{k-1} \cdot (
\p_1 d_t y_h^k \times \p_2 \big(y_h^k-y_h^{k-1}\big) \big] \dv{x'} \\
&\le c \|\Delta_h d_t y_h^k\|_h 
\big( \|\p_1 y_h^k \|_{L^\infty_h} \tau \|\p_2 d_t y_h^k\|_h + 
\|\p_1 y_h^{k-1} \|_{L^\infty_h} \tau \|\p_2 d_t y_h^k\|_h \big)  \\
& \qquad + c \|\Delta_h y_h^{k-1}\|_h \tau \|\p_1 d_t y_h^k\|_h \|\p_2 d_t y_h^k\|_{L^\infty_h}.
\end{split}\]
We use the estimates $\|\p_j d_t y_h^k\|_{L^\infty_h} \le c |\log \hh| \|d_t y_h^k\|_\star $ 
and $\|\p_j y_h^{k-\ell}\|_{L^\infty_h} \le c$ for $j=1,2$ and $\ell=0,1$, 
to bound the right-hand side by $\tau |\log \hh| c''' \|d_t y_h^k\|_\star^2$. This leads to 
\[
(1-c'''\tau|\log \hh|) \|d_t y_h^k\|_\star^2 + d_t I_\bil^h[y_h^k]\le 0 
\] 
and proves the asserted energy estimate. With this bound we also obtain  
the improved bound for the constraint violation error. 
\end{proof}

The good stability properties of the newly proposed numerical scheme are
confirmed by a numerical experiment whose outcome is visualized in
Figure~\ref{fig:exp_bilayer}. The setup uses a rectangular strip of length
$L=10$ and width $w=4$ that is clamped at one end. The spontaneous
curvature parameter is $\a=-1$ and we have $\cb=1$ and $\cspont=1$. 
The figure shows snapshots of the evolution
on a grid with medium mesh-size and stationary states for different triangulations
with mesh-sizes~$h$ proportional to $2^{-\ell}$, $\ell=1,2,3$.
The step-size was always set to $\tau = h/20$. The stopping criterion
was chosen as $\veps_{\rm stop} = 10^{-3}$. Because
of the extreme geometry and the large deformation the asymmetry of the
underlying triangulations is reflected in the numerical solutions but this
effect disappears for smaller mesh-sizes. The theoeretical results about
the energy monotonicity and controlled constraint violation are confirmed
by the bottom plot of Figure~\ref{fig:exp_bilayer}.  

\renewcommand{\bild}[1]{\fbox{\includegraphics[scale=.09,trim=0 -75 0 50,clip]{bilayer-#1.png}}\makebox[0ex][r]{\tiny$k=#1$\ }\,\ignorespaces}
\begin{figure}[p]
\bild{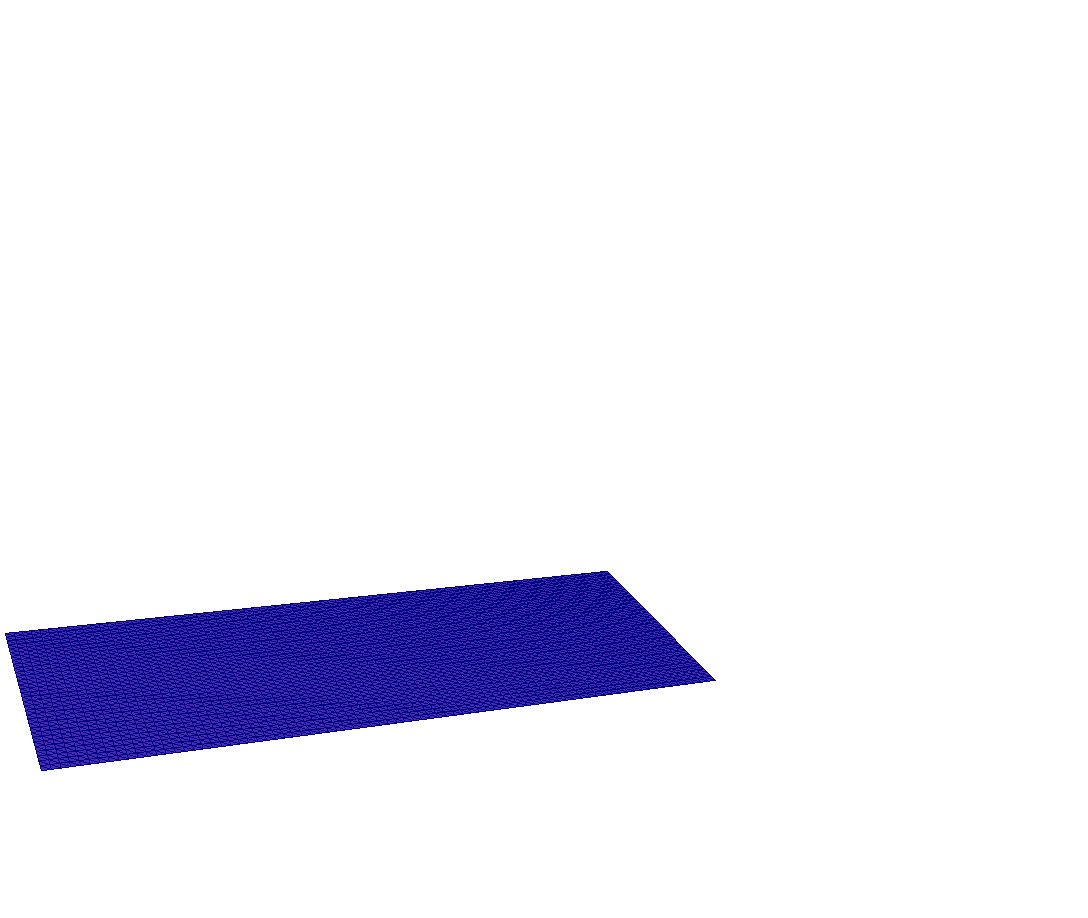} \bild{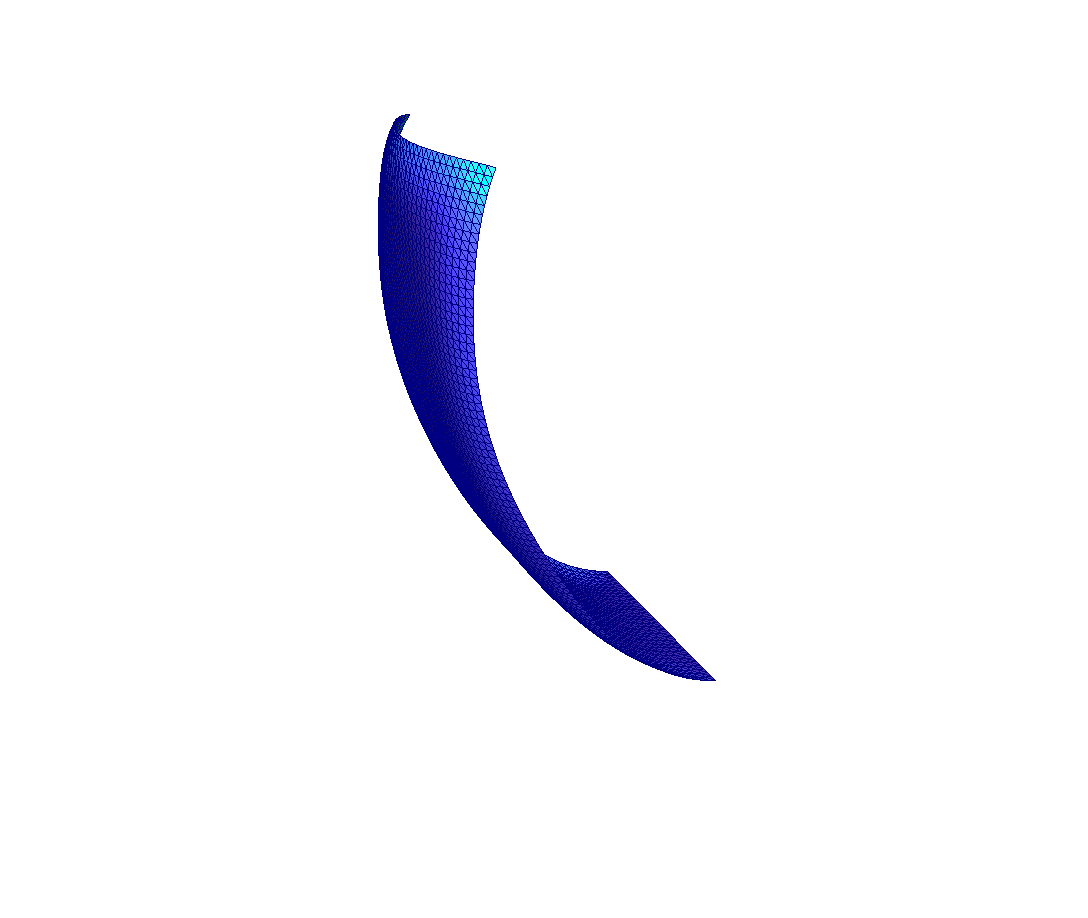} \bild{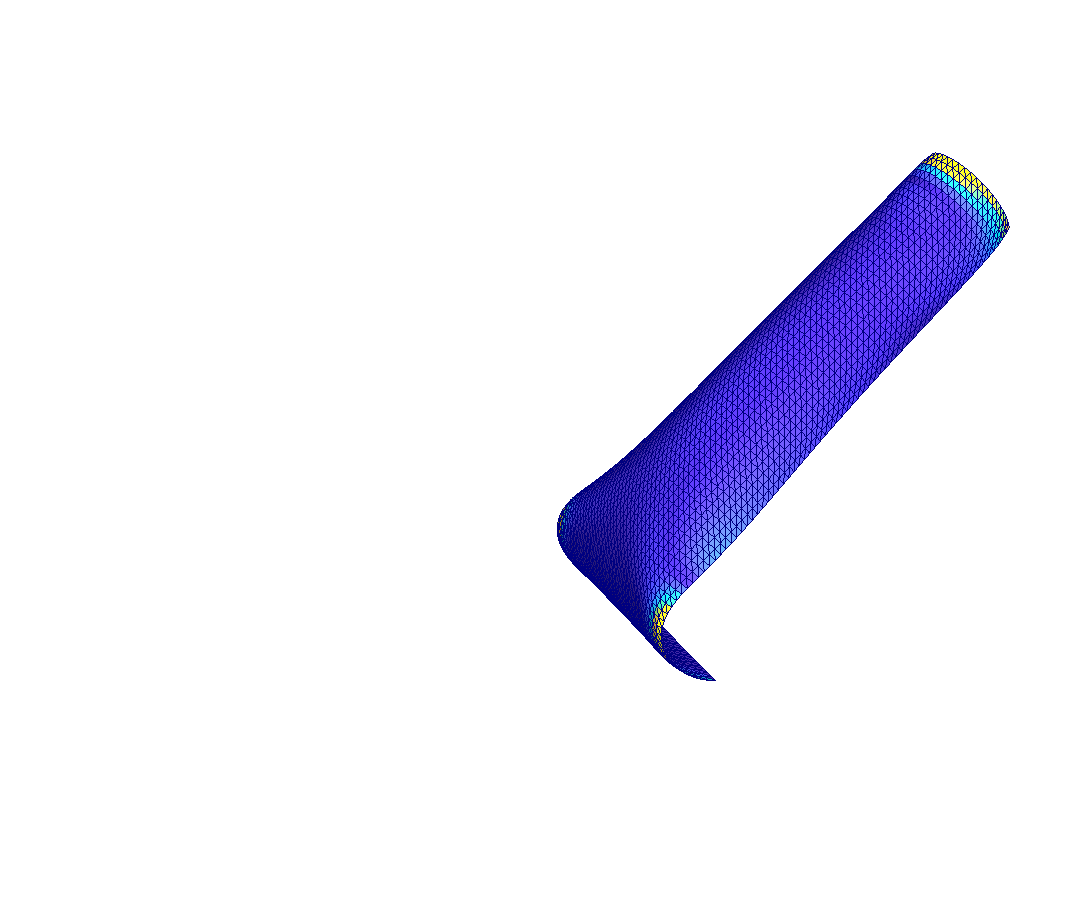} \\
\bild{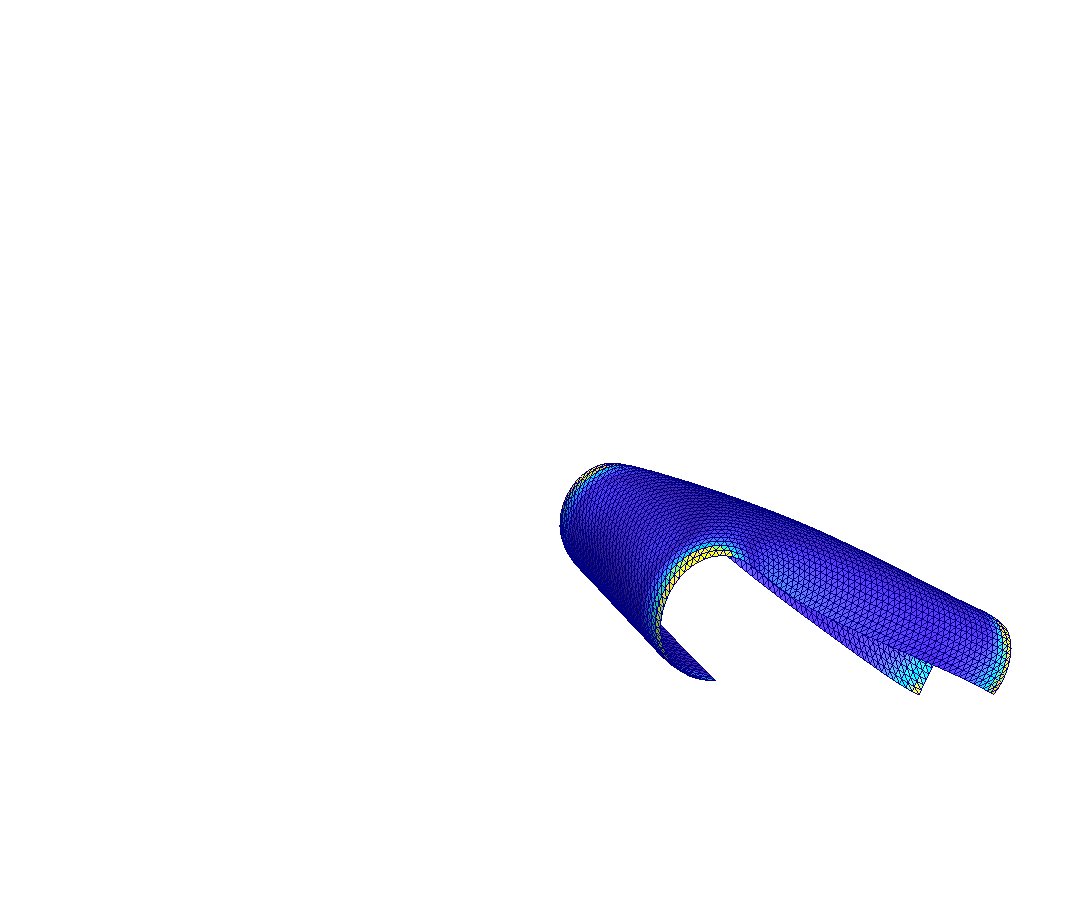} \bild{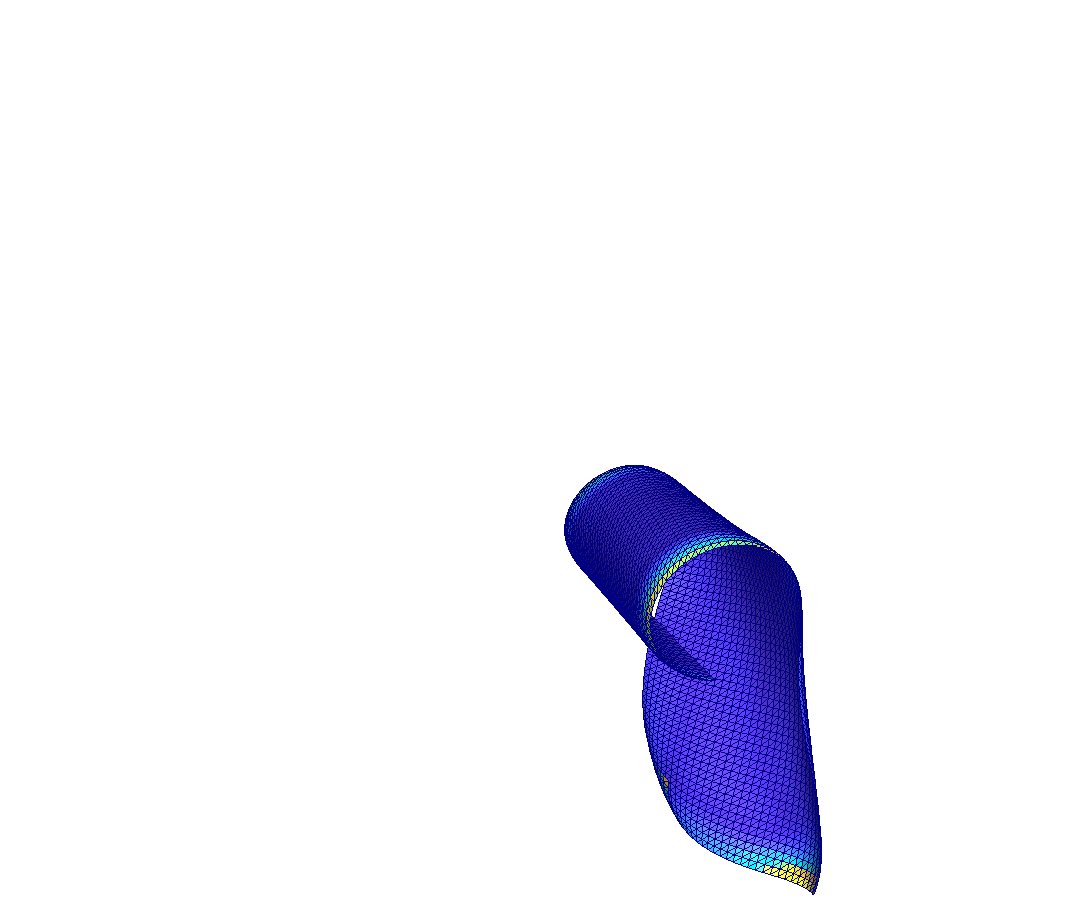} \bild{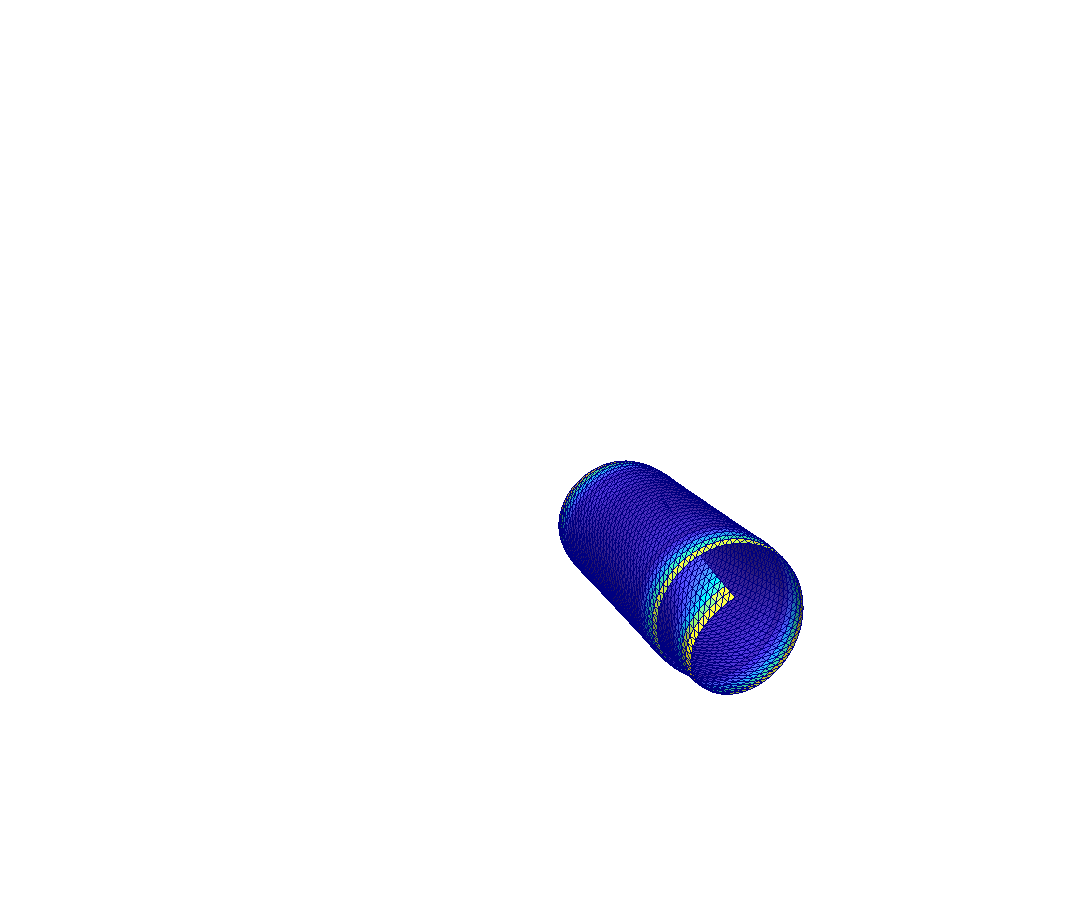} \\
 \vspace*{4mm}
 \fbox{\includegraphics[scale=.09,clip]{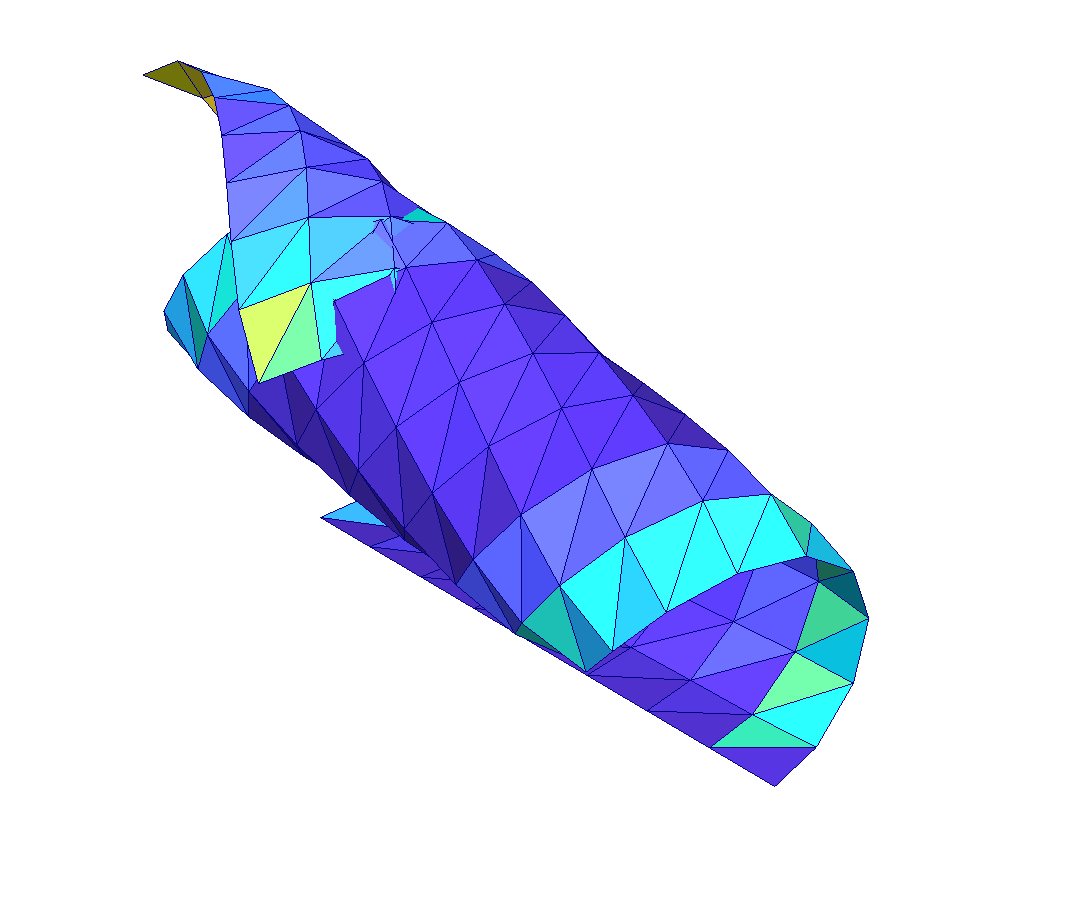}}\makebox[-13ex][r]{\tiny $\#\cT_h=320$\ }\makebox[13ex][r]{\tiny $k=27896$ }\,\ignorespaces
 \fbox{\includegraphics[scale=.09,clip]{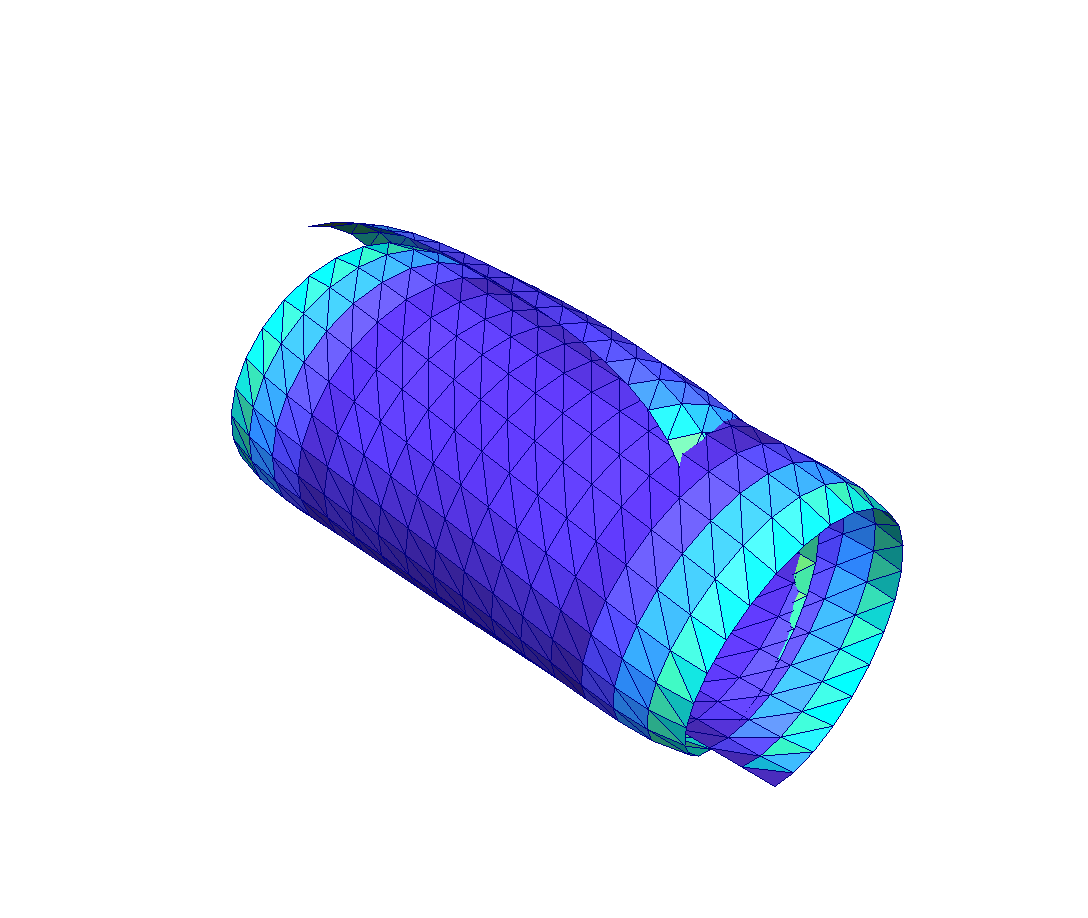}}\makebox[-12ex][r]{\tiny $\#\cT_h=1280$\ }\makebox[12ex][r]{\tiny $k=36837$ }\,\ignorespaces
 \fbox{\includegraphics[scale=.09,clip]{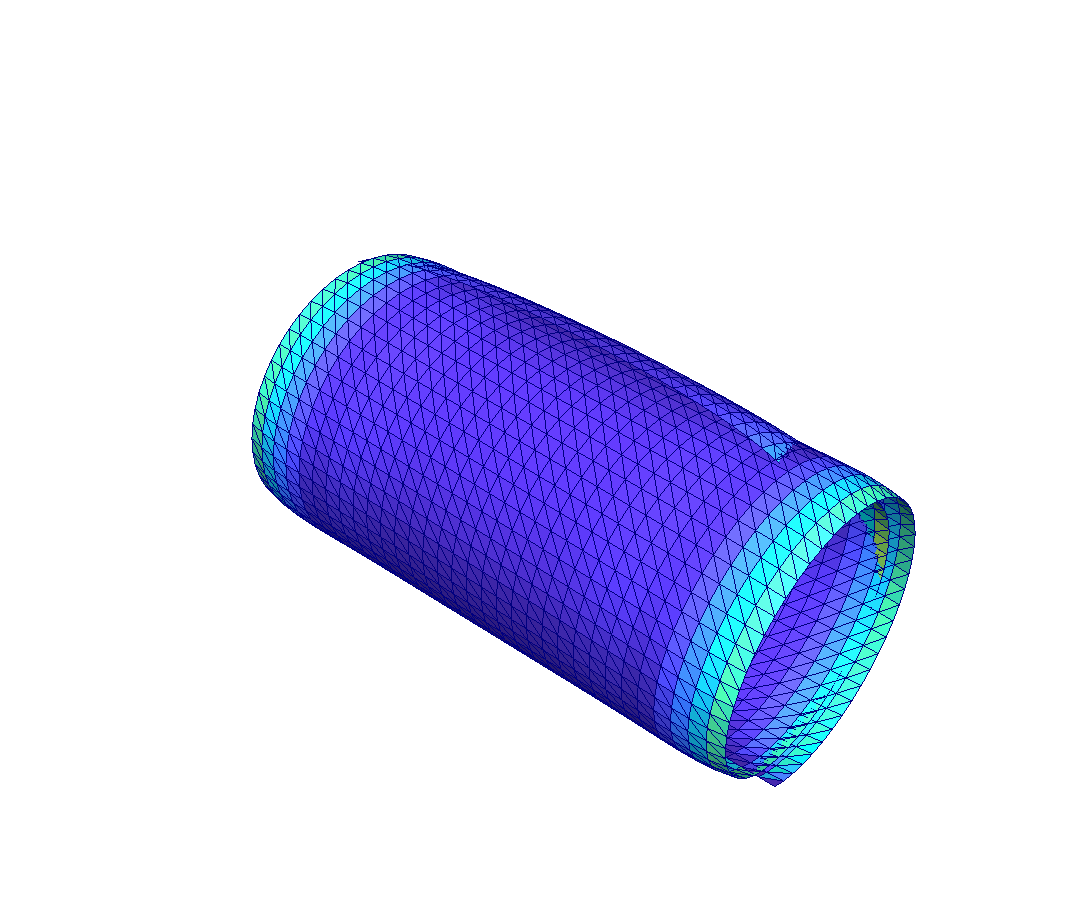}}\makebox[-12ex][r]{\tiny $\#\cT_h=5120$\ }\makebox[12ex][r]{\tiny $k=71003$ }\,\ignorespaces
 \vspace*{4mm}
\includegraphics[scale=.75]{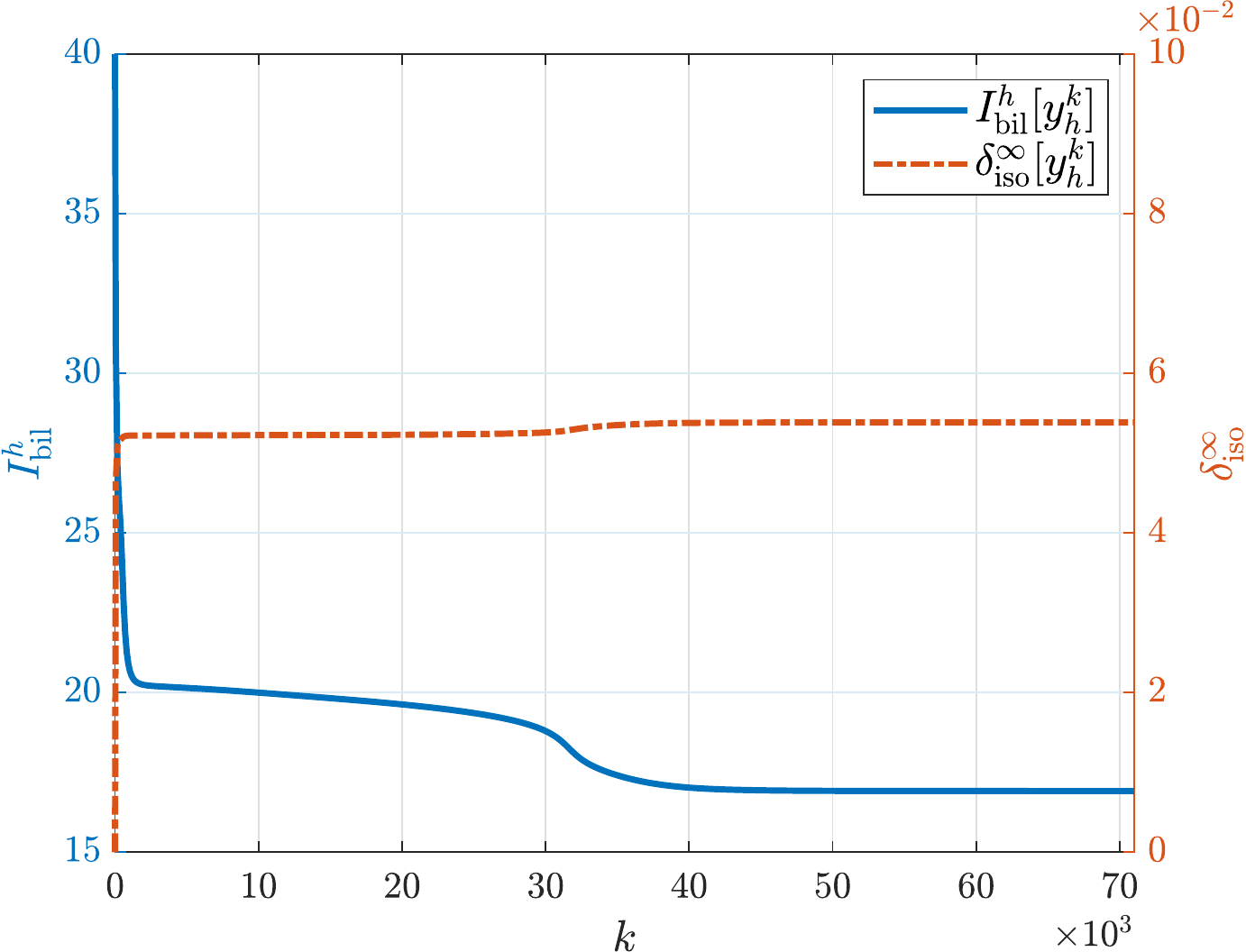}
\caption{\label{fig:exp_bilayer}
Iterates in a bilayer bending problem with 
clamped boundary condition on one end of the strip. The plate immediately
bends everywhere and evolves into a multiply covered tube (top); 
discretization effects like asymmetries disappear for
finer triangulations (middle); the energy decreases monotonically
and becomes stationary (bottom).}
\end{figure}

\subsection{Selfavoiding curves and elastic knots}
A strategy for finding useful representatives of knot classes, i.e., 
closed curves within a given isotopy class with particular features, 
is to minimize or decrease the bending energy within the given class via 
continuous evolutions defined by a gradient flow. To ensure that the
flow does not change the topological properties of the curve, appropriate
terms have to be included in the mathematical formulation. To this end
we add a self-avoidance potential to the bending energy that prevents
the curve from self-intersecting or pulling tight, i.e., we consider
flows determined by functionals
\[
I^\tot[y] = \frac{\cb}{2} \int_0^L |y''|^2 \dv{x_1} + \vrho \TP[y]
\]
in sets of curves satisfying periodic or, e.g., clamped boundary conditions.
We refer the
reader to~\cite{OHar03-book} for a general discussion of appropriate functionals. A
functional that turned out to have several advantageous features is
the tangent-point functional proposed in~\cite{GonMad99}. It is 
for a $C^1$ curve $y:[0,L]\to \R^3$ defined via the tangent-point
radius $r_y(x,z)$ which is the radius of the circle that is tangent to $y$ in
$y(x)$ and intersects the curve in the point $y(z)$. For curves that
are parametrized by arclength we have
\[
r_y(x,z) = \frac12 \frac{|y(x)-y(z)|}{|y'(x)\times (y(x)-y(z))|}
\]
which follows from investigating the geometrical configuration sketched in 
Figure~\ref{fig:tp_radius}. 
If $z\to x$ then we have that $r_y(x,z)$ converges to the inverse of the 
curvature of the curve $y$ at $x$. If otherwise $y(x)\to y(z)$ for fixed
$x\neq z$ then the radius
converges to zero as depicted in Figure~\ref{fig:tp_singular}. 
An alternative choice is the use of the Menger curvature which is discussed
in~\cite{StSzvdM13}.

\def\bfr{{\bf r}}
\begin{figure}[h]
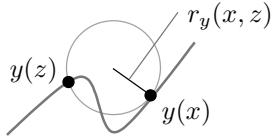
\caption{\label{fig:tp_radius} The tangent-point radius 
$r_y(x,z)$ of a curve $y$ is the radius of the circle that is tangent
to $y$ in $y(x)$ and intersects the curve in $y(z)$.}
\end{figure}

The tangent-point functional $\TP$ is for an exponent $q\ge 1$ defined
via
\[
\TP[y] = \frac{2^q}{q} \int_0^L \int_0^L \frac{1}{r_y(x,z)^q} \dv{x}\dv{z}.
\]
The exponent has to be suitably chosen so that singularities corresponding
to a vanishing tangent-point radius are sufficiently strong to lead to 
an infinite value of the functional. Thereby, an energy barrier is defined
that separates different isotopy classes. The functional~$\TP$ has important 
and remarkable features. 

\def\bfr{{\bf r}}
\begin{figure}[h]
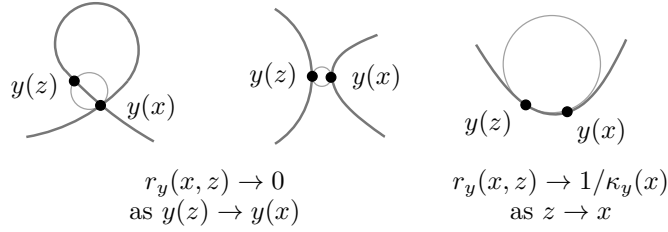
\caption{\label{fig:tp_singular} Situations in which the tangent-point
radius approaches zero leading to a singularity in the tangent-point
functional $\TP$ (left and middle); as $z\to x$ the tangent-point radius approximates
the inverse of the curvature of $y$ at $x$ (right).}
\end{figure}

\begin{remarks}
(i) If the arclength-parametrized curve $y\in C^1([0,L];\R^3)$ is injective 
then $\TP[y]$ is finite if and only if $y\in W^{2-1/q,q}(0,L;\R^3)$, 
cf.~\cite{Blat13}. \\
(ii) For all $M>0$ and $q>2$ there exists a constant $c_{M,q}>0$ such that
for all arclength-parametrized curves $y\in C^1([0,L];\R^3)$ with $\TP[y]\le M$
we have the bi-Lipschitz estimate $|x-z|\le c_{M,q} |y(x)-y(z)|$ for all $x,z\in [0,L]$,
cf.~\cite{BlaRei15}.
In particular, $\TP$ is a knot energy in the sense that for a family 
of curves $(y_k)_{k\in \N}$ converging pointwise to a curve with self-intersection, 
the values $\TP[y_k]$ blow up, cf.~\cite{VdM98,OHar03-book,StrVdM12,BlaRei15}. \\
(iii) The functional $\TP$ is twice continuously differentiable with bounded
variations and $L^1$ integrands, cf.~\cite{BlaRei15,BarRei18-pre}. 
\end{remarks}

To decrease the total energy $I^\tot$ of a given initial curve $y_0$ 
within its isotopy class we use a discretization of the evolution 
\[
(\p_t y,w)_\star + \cb (y'',w'') + \vrho \TP'[y;w] = 0
\]
subject to the linearized arclength conditions $\p_t y' \cdot y' = 0$ 
and $w' \cdot y' = 0$. The discretization uses the
ideas outlined in Sections~\ref{sec:fem_rods} and~\ref{sec:iter_rods}
with an explicit treatment of the potential.

\begin{algorithm}[Gradient descent for selfavoiding curves]\label{alg:descent_selfavoid}
Choose an initial $y_h^0 \in \cA_h$ and a step-size $\tau>0$,
set $k=1$. \\
(1) Compute $d_t y_h^k\in \cF_h[y_h^{k-1}]$ such that for all $w_h \in \cF_h[y_h^{k-1}]$
we have
\[
(d_t y_h^k,w_h)_\star + \cb  ([y_h^k]'',w_h'') = - \vrho\TP_h'[y_h^{k-1};w_h].
\]
(2) Stop the iteration if $\|d_t y_h^k\|_\star \le \veps_{\rm stop}$; 
otherwise, increase $k\to k+1$ and continue with~(1).
\end{algorithm}

The explicit treatment
of the potential has several advantages. First, we obtain linear problems
in the time steps. Second, we avoid the inversion of fully populated matrices.
Third, the assembly of the vector on the right-hand side can be fully 
parallelized without communication costs. 
Finally, we do not change the stability of the iteration compared to a fully
implicit time stepping scheme since the potential does not have obvious convexity 
properties. Using the differentiability properties of the functional $\TP$ and assuming 
that the flow metric is the $H^2$ scalar product, it has been shown in~\cite{BarRei18-pre}
that the conditional energy decay property
\[
I^\tot_h[y_h^L] + (1-c_\TP \tau) \tau \sum_{k=1}^L \|d_t y_h^k\|_\star^2 
\le I^\tot_h [y_h^0]
\]
holds for all $0\le L \le K \le T/\tau$ with a finite time horizon $T>0$. 
This inequality implies the convergence of the iteration to 
a stationary configuration. Since the total energy is bounded uniformly,
the values of the potential remain controlled so that in case of a
sufficient resolution no self-contact or topology changes can occur. 

The good stability properties and corresponding selfavoidance behavior
of the numerical scheme are illustrated 
by the snapshots of configurations and the corresponding energy decay
shown for a generic setting in Figure~\ref{fig:knot_25}. We observe
that the curve relaxes to a curve with equilibrated curvature and which
is nearly flat. While the total energy decreases the (unscaled) tangent-point
functional increases. In the experiments we used the model parameters 
$\cb = 10$, $\vrho = 10^{-3}$, $q=3.9$. The initial curve
belongs to the knot class $8_{10}$, cf.~\cite{Rolf90-book}, and has a total 
length $L\approx 45.45$ which is accurately preserved during the evolution. 
To compute the evolution we used a partitioning of the reference
interval into $551$ subintervals. The step-size $\tau$ was chosen
proportional to the mesh-size~$h$. 

\renewcommand{\bild}[1]{\fbox{\includegraphics[scale=.11,trim=150 100 130 50,clip]{#1.jpg}}\makebox[0ex][r]{\tiny$k=#1$\ }\,\ignorespaces}
\begin{figure}[p]
\bild{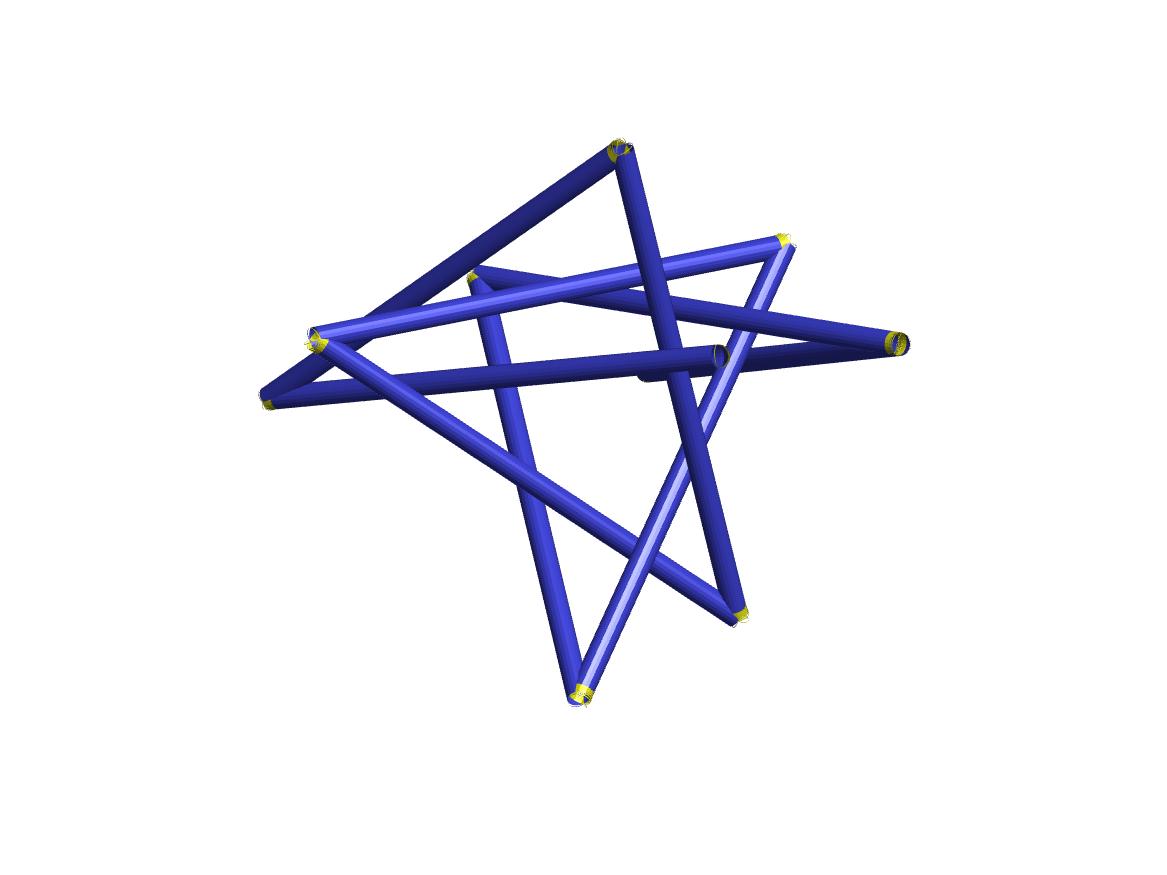} \bild{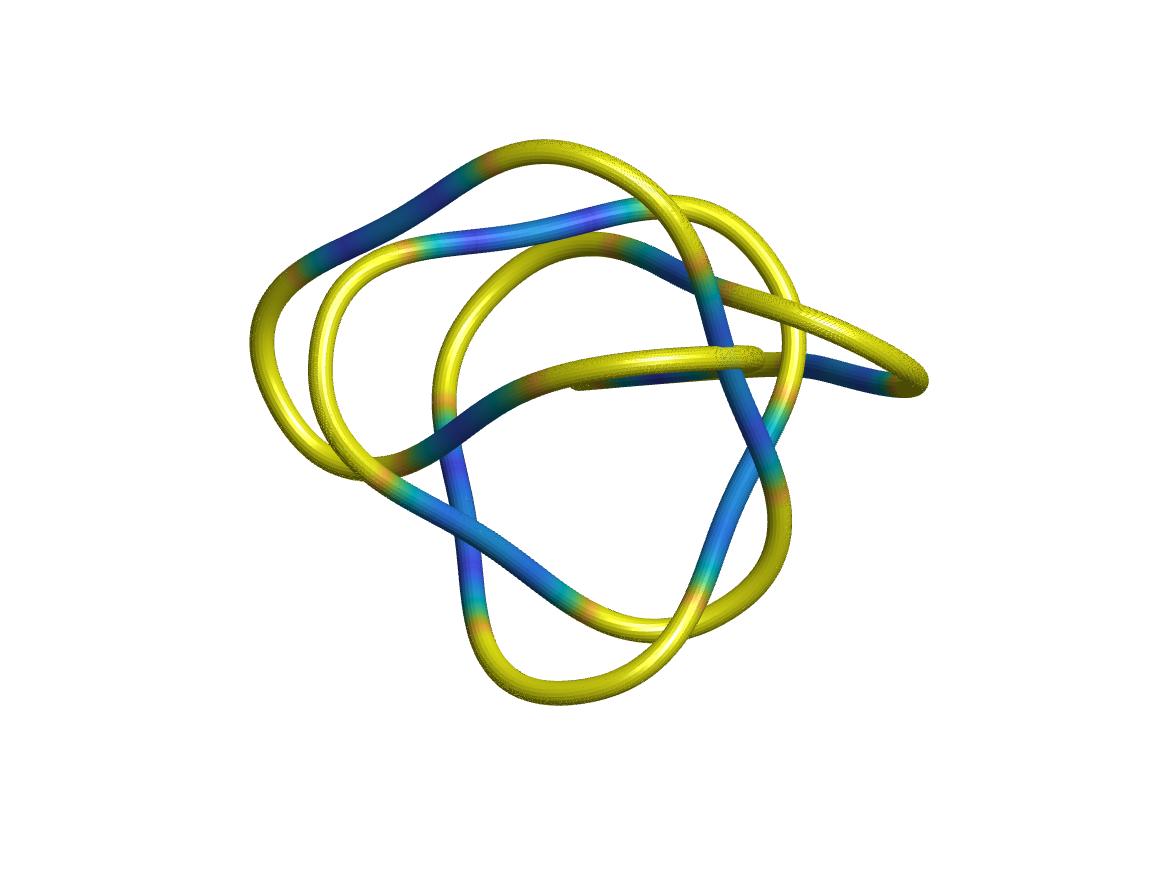} \bild{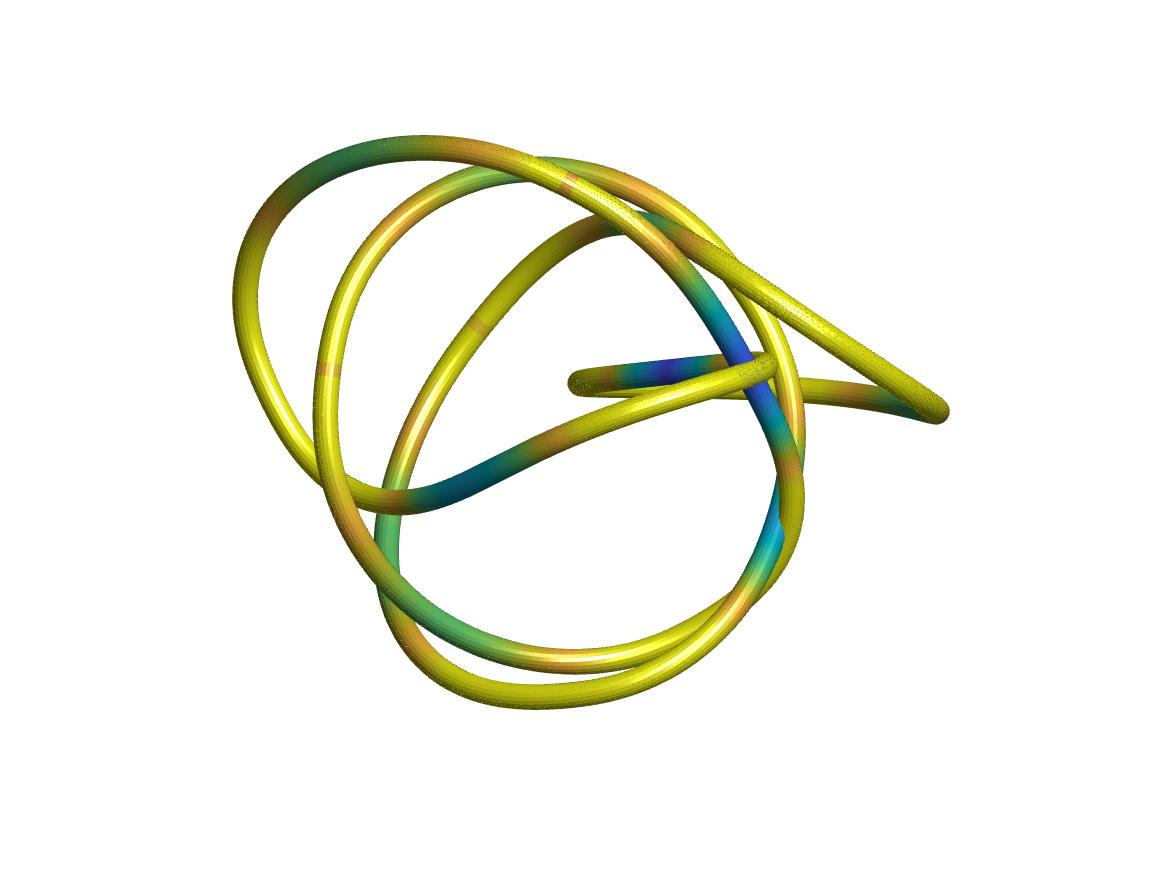} \\
\bild{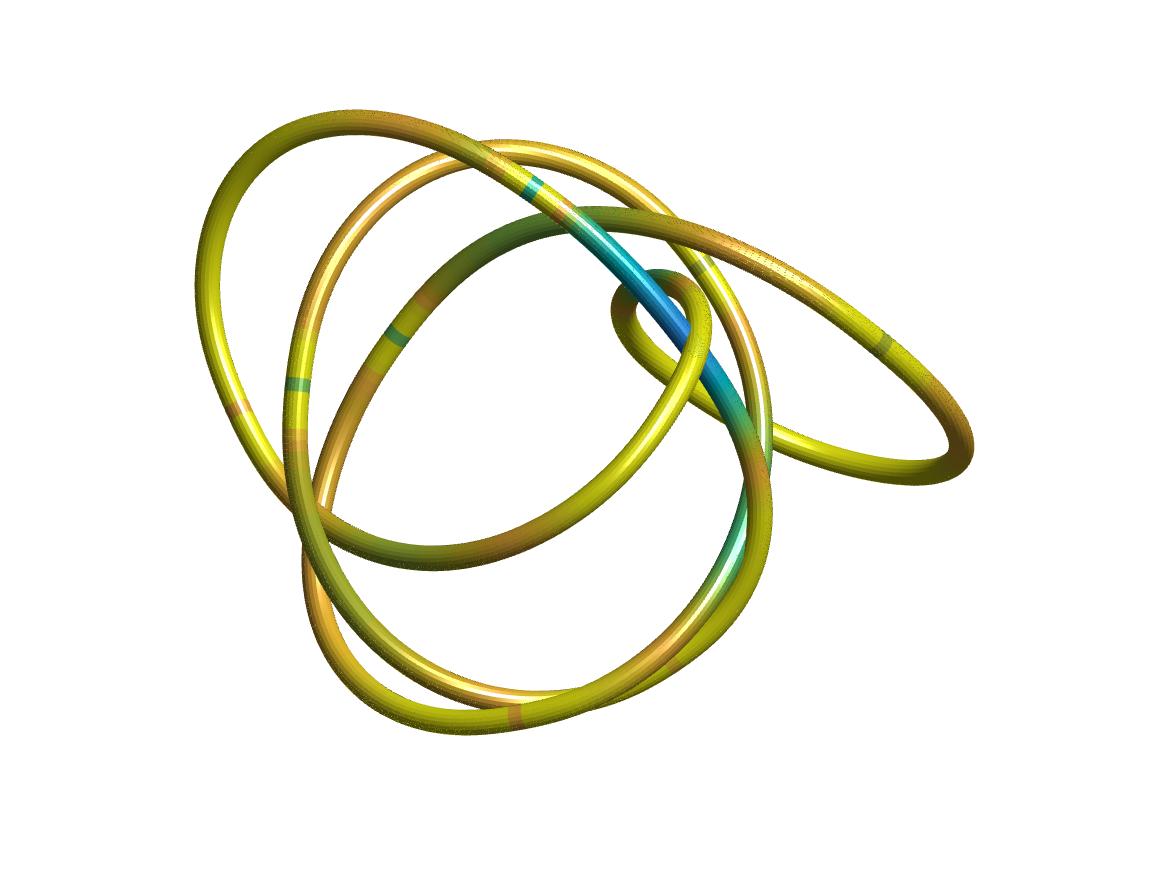}  \bild{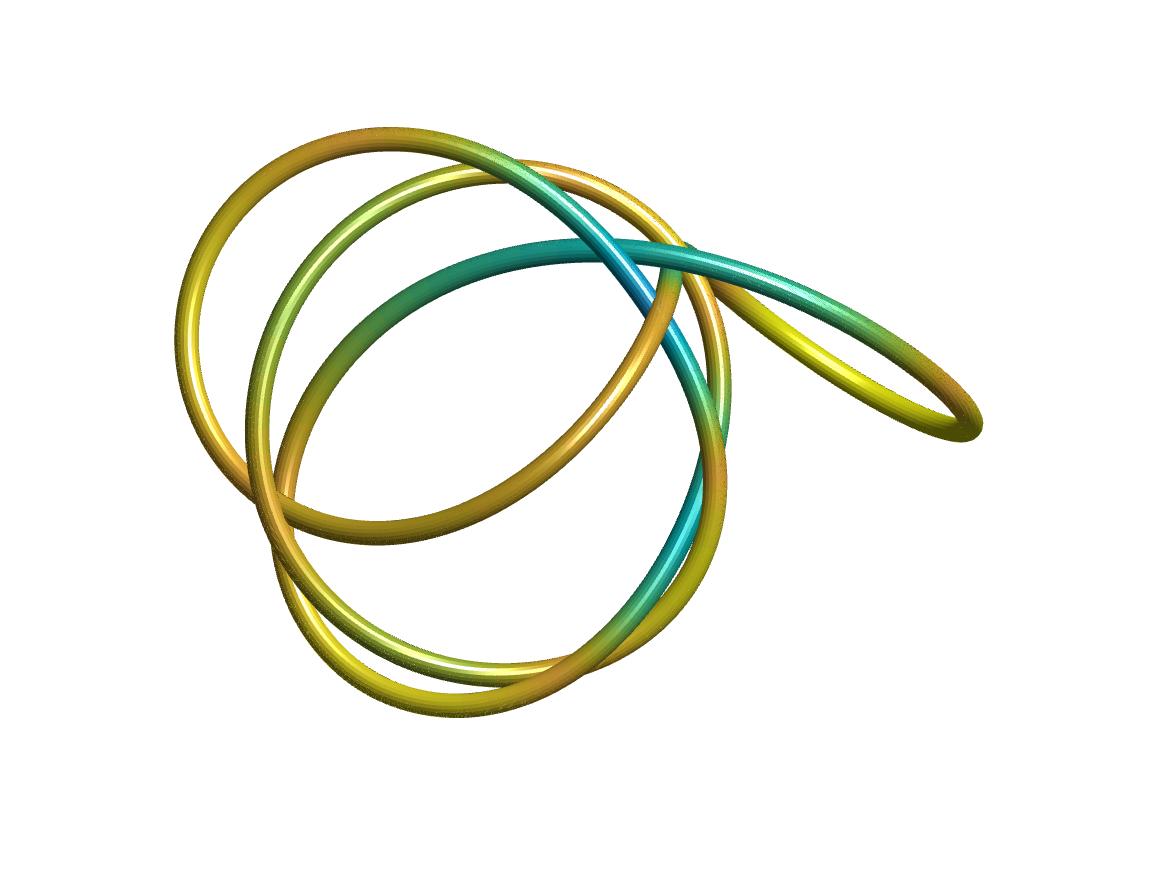} \bild{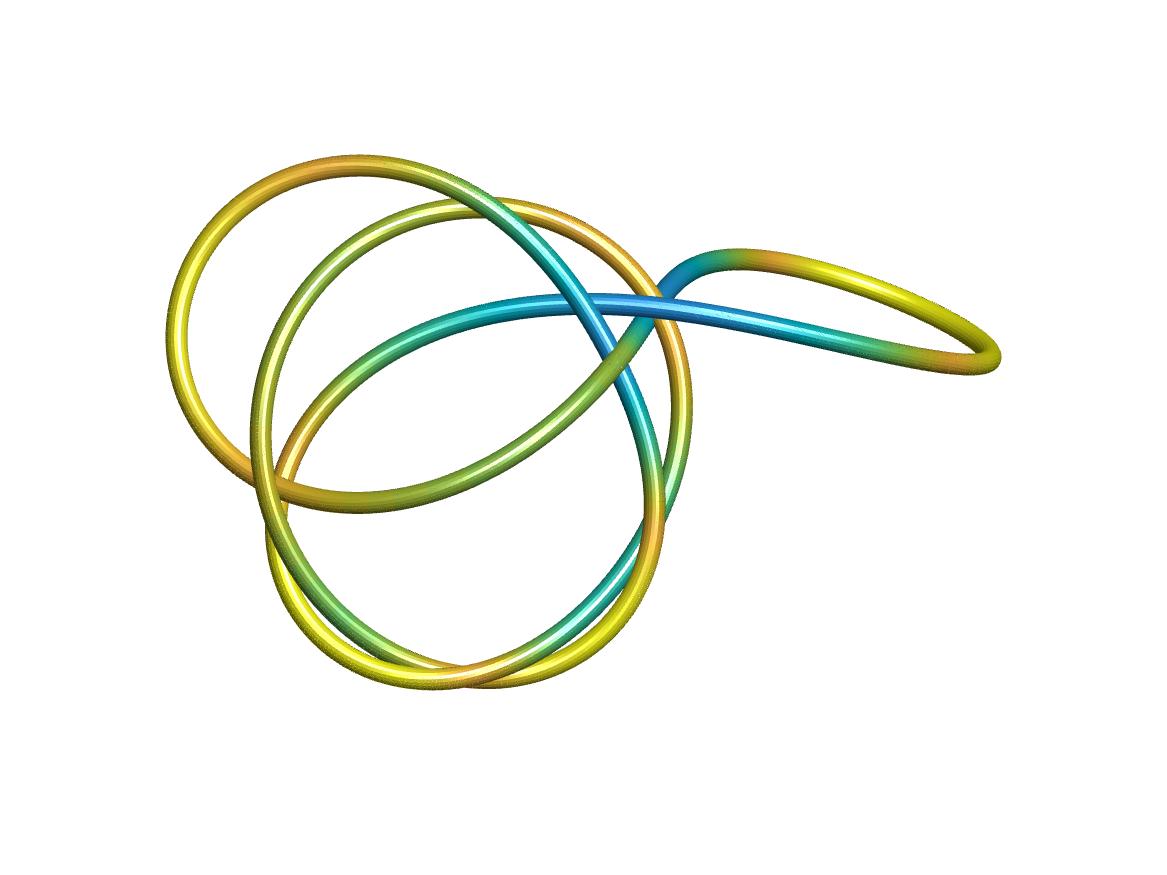} \\
\bild{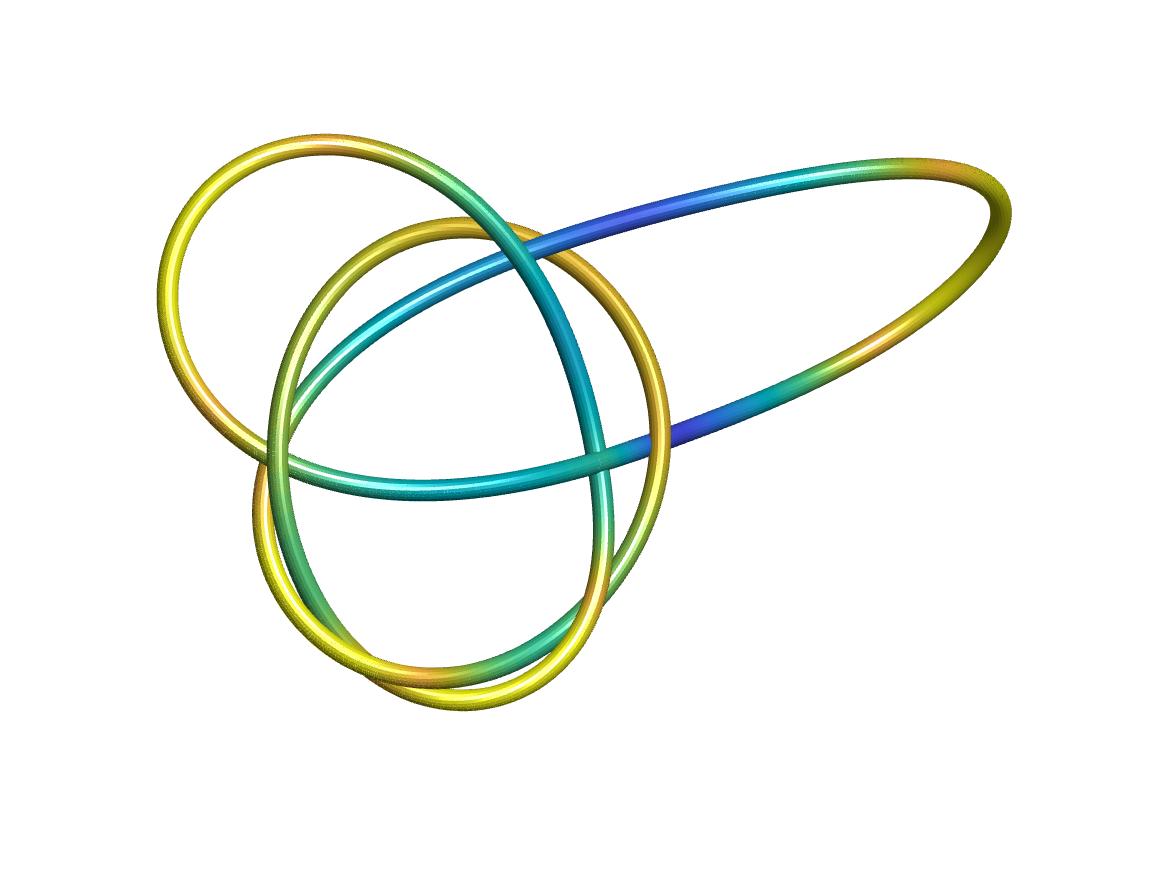} \bild{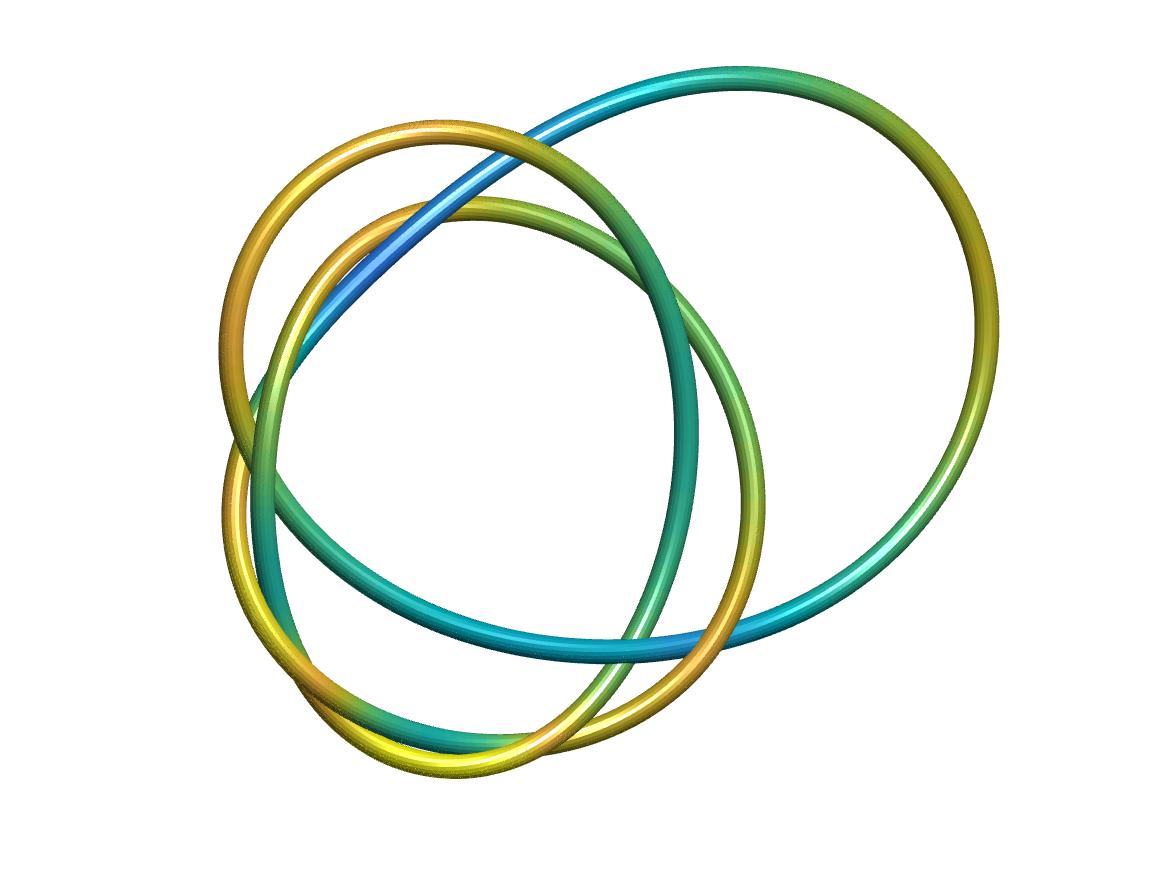} \bild{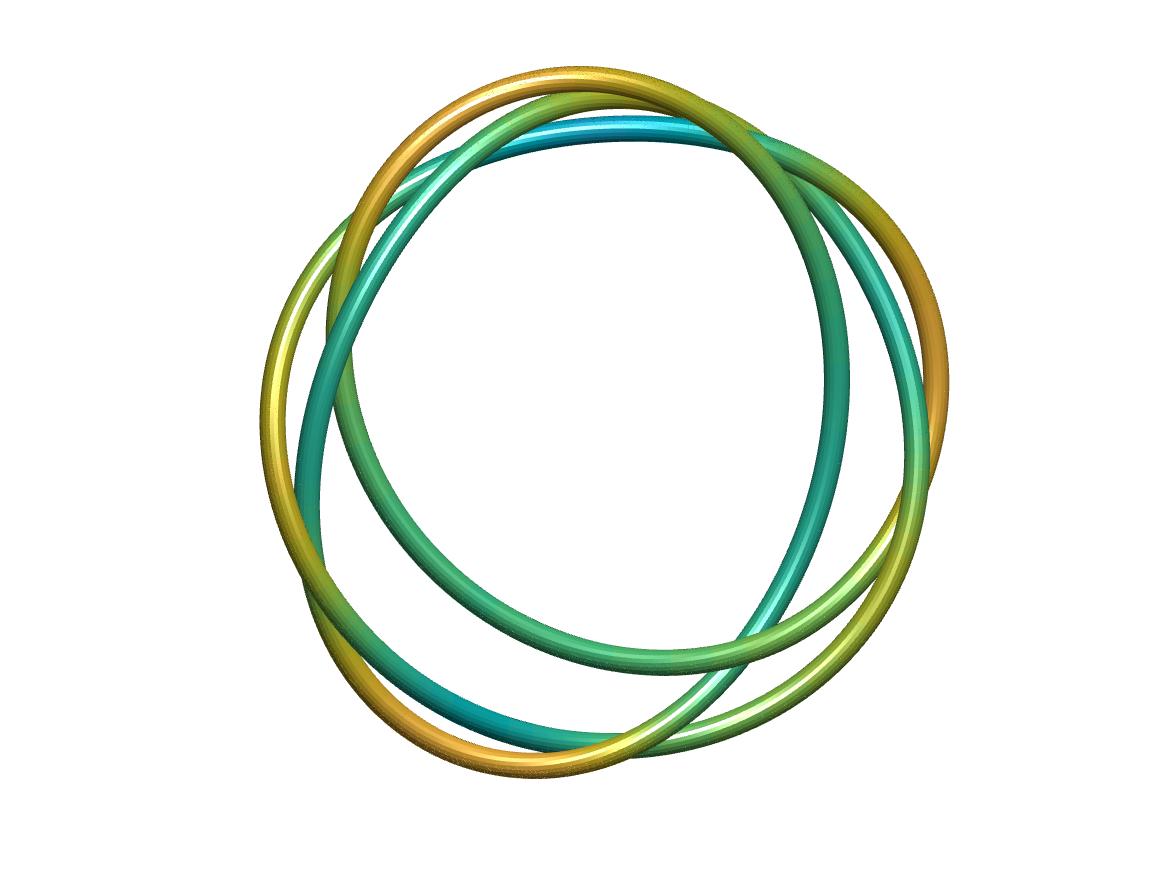}
\vspace*{5mm}
\includegraphics[width=10.5cm,height=8cm]{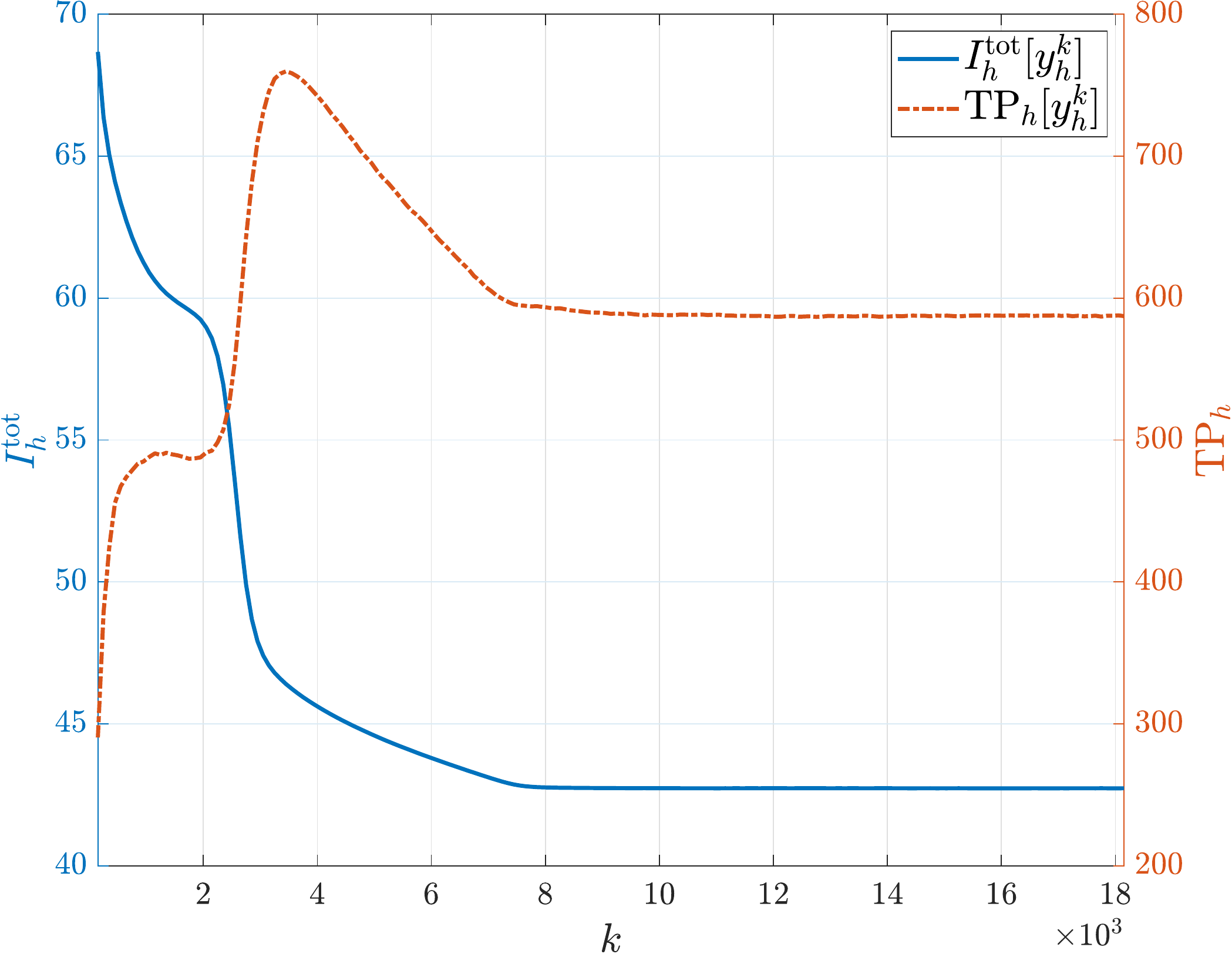}
\caption{\label{fig:knot_25}
Snapshots of an evolution from a polygonal initial curve
after different numbers of iterations. The curve relaxes to a
configuration that is close to a multiple covering of a flat circle (top).
The total energy decreases monotonically while the contribution of 
the tangent-point functional remains controlled (bottom) so that the 
curve preserves the isotopy class defined by its initial configuration.}
\end{figure}

\subsection{\fvk\ model}
The nonlinear bending model discussed in the previous sections is 
not suitable to describe certain phenomena such as the formation 
of wrinkles which occur for very small energies. In order to describe 
these effects in a 
dimensionally reduced model it is essential to involve the thickness
parameter $\d>0$ as this number determines the oscillating or 
nonsmooth behavior of deformations. A model that captures such effects
is the \fvk\ model which can be identified via different scalings
of the in-plane and out-of-plane components of a three-dimensional
deformation when~$\d$ is small. It determines a planar
displacement $u:\o\to \R^2$ and a deflection $w:\o \to \R$ 
as a minimizing pair for the energy functional 
\[
I_\fv[u,w] = \frac{\d^2}{2} \int_\o |D^2 w|^2 \dv{x} 
+ \frac12 \int_\o |\tveps(u) + \nabla w \otimes \nabla w|^2 \dv{x}
\]
in a set $\cA$ of admissible pairs contained in the product space 
$H^1(\o;\R^2)\times H^2(\o)$. Here, we use the symmetric gradient
\[
\tveps(u) = 2 \sym(\nabla u) = (\nabla u)^\transp + \nabla u
\]
and the dyadic product
\[
\nabla w \otimes \nabla w = (\nabla w) (\nabla w)^\transp,
\]
where $\nabla w$ is understood as a column vector. We refer the reader 
to~\cite{Ciar80,Ciar97-book,FrJaMu02c,FrJaMu06} for justifications
of this minimization problem as a simplification of three-dimensional
hyperelastic material models. It has been shown rigorously 
in~\cite{MulOlb14b,Venk04} that the presence of the parameter~$\d$
leads to minimizers with wrinkling patterns whose geometry is
determined by the parameter~$\d$. This is done by identifying optimal scaling 
laws of the energy in terms of $\d$, cf.~\cite{BelKoh14,CoOlTo17} for further
examples. Only a few numerical methods
for minimizing the \fvk\ energy are available, cf.~\cite{CiGrKe05}
for an abstract investigation. 

To minimize the energy for given boundary conditions we follow~\cite{Bart17}
and again adopt a gradient flow appraoch defined by the system 
\[\begin{split}
\big(\p_t w,v\big)_\verti &= -\g^2 \big(D^2 w,D^2 v\big) - 2 \big(|\nabla w|^2\nabla w 
+ \tveps(u)\nabla w,\nabla v\big), \\
\big(\p_t u,z\big)_\hor &= - \big(\tveps(u),\tveps(z)\big) 
- \big(\nabla w\otimes \nabla w,\tveps(z)\big).
\end{split}\]
Here, $(\cdot,\cdot)_\verti$ and $(\cdot,\cdot)_\hor$
are inner products on $H^2(\o)$ and $H^1(\o;\R^2)$, respectively,
and we used the identities  $|a\otimes a|^2 =|a|^4$ and 
$M:(a\otimes b) = (Ma)\cdot b=(Mb)\cdot a$
for a symmetric matrix $M\in \R^{2\times 2}$ and $a,b\in \R^2$. 
The temporal discretization of the system
decouples the equations via a semi-implicit evaluation of
different terms. In particular, given $(u^{k-1},w^{k-1})$ 
we compute $(u^k,w^k)$ such that 
\[\begin{split}
\big(d_t w^k,v\big)_\verti  & = -\g^2 \big(D^2 w^k,D^2 v\big) 
- 2 \big(|\nabla w^k|^2\nabla w^k + \tveps(u^{k-1})\nabla w^{k-1/2},\nabla v\big), \\
\big(d_t u^k,z\big)_\hor &= 
- \big(\tveps(u^k),\tveps(z)\big) - \big(\nabla w^k\otimes \nabla w^k,\tveps(z)\big),
\end{split}\]
for all $(v,z)$ satisfying appropriate homogeneous boundary conditions. The 
average $w^{k-1/2}$ is defined for subsequent approximations 
$w^k$ and $w^{k-1}$ via
\[
w^{k-1/2} = \frac12 \big(w^k + w^{k-1}\big).
\]
The unconditional stability and energy decay of the iteration 
follows from choosing $v=d_t w^k$ and $z = d_t u^k$ and exploiting
a discrete product rule, cf.~\cite{Bart17} for details. 
The problems in the time steps determine
minimizers of certain functionals and if $0<\tau \le \tau_0$ then
these functionals are strongly convex and minimizers are uniquely 
defined. Correspondingly, we expect the Newton scheme to converge
provided that the step-size $\tau$ is sufficiently small. 

For a full discretization we use the set of admissible pairs
\[
\cA_h = \big\{(u_h,w_h) \in \cS^1(\cT_h)^2 \times \cS^\dkt(\cT_h):
L_\bc[u_h,w_h] = \ell_\bc \big]
\]
and the corresponding homogeneous space
\[\begin{split}
\cF_{h,w}^0 &= \big\{ v_h \in \cS^\dkt(\cT_h): L_{\bc,w}[v_h] = 0\big\}, \\
\cF_{h,u}^0 &= \big\{ z_h \in \cS^1(\cT_h)^2: L_{\bc,u}[z_h] = 0 \big\}.
\end{split}\]
To avoid an unnecessary small uniform step-size $\tau$ we instead use
variable step-sizes $(\tau_k)_{k=0,1,\dots}$ that are adjusted
according to the performance of the Newton scheme using the following
rules:
\[\begin{split}
&\text{(i)} \quad \text{decrease $\tau_k$ until Newton scheme terminates within $N_{\rm max}$ iterations}, \\  
&\text{(ii)} \quad \text{set } \tau_{k+1} = \min\{2 \tau_k,10^r\} \text{ for next time step}.
\end{split}\]
The parameter $r\ge 0$ defines an upper bound for the step-sizes
and avoids a numerical overflow. The ideas lead to the following algorithm.

\begin{algorithm}[Gradient descent for \fvk\ functional]\label{alg:fvk_grad_flow}
Choose $(u_h^0,w_h^0)\in \cA_h$, an integer $N_{\rm max}>0$, 
stopping tolerances $\veps_\stop,\veps_{\rm N}>0$, 
and an initial step-size $\tau_1>0$, set $k=1$. \\
(1a) Repeatedly decrease $\tau_k$ until the Newton scheme terminates within $N_{\rm max}$ steps
and tolerance $\veps_{\rm N}$ to determine $d_t w_h^k \in \cF_{h,w}^0$ such that
\[\begin{split}
(d_t w_h^k, v_h)_\verti &=  -\g^2 \big(D_h^2 w_h^k,D_h^2 v_h\big) \\
& \qquad - 2 \big(|\nabla w_h^k|^2\nabla w_h^k + \tveps(u_h^{k-1})\nabla w_h^{k-1/2},\nabla v_h\big)_h .
\end{split}\]
for all $v_h\in \cF_{h,w}^0$. \\
(1b) Compute $d_t u_h^k \in \cF_{h,u}^0$ such that 
\[
\big(d_t \tveps(u^k_h),\tveps(z_h)\big)_\hor = - \big(\tveps(u_h^k),\tveps(z_h)\big) 
- \big(\nabla w_h^k\otimes \nabla w_h^k,\tveps(z_h)\big)_h
\]
for all $z_h\in \cF_{h,u}^0$. \\
(2) Stop if $\|d_t w_h^k\|_\verti + \|d_t \tveps(u^k_h)\|_\hor  \le \veps_\stop \min\{1,\tau_k\}$; 
otherwise, define 
\[
\tau_{k+1} = \min\big\{ 2 \tau_k,10^r\big\},
\]
increase $k\to k+1$, and continue with~(1).
\end{algorithm}

Note that in the algorithm a stopping criterion is used that is proportional 
to the step-size which is important in the case of small step-sizes. To 
illustrate the performance of the algorithm and features of the mathematical
model we follow~\cite{Bart17} and consider the compression of a plate along one of its sides. Particularly,
we let $\o = (-1/2,1/2)\times (0,1)$ and compress the side $\{0\} \times [-1/2,1/2]$
by~10 percent which defines the boundary data for the in-plane displacement~$u$. 
We use homogeneous clamped boundary conditions along the same side
for the deflection~$w$. The results shown in Figure~\ref{fig:fvk_exp} 
confirm the formation of wrinkling structures. These depend on both the 
numerical resolution and the thickness parameter. Moreover, the energy
decay shown in the bottom plot of Figure~\ref{fig:fvk_exp} indicates that 
these are only obtained within
a reasonable number of iterations if an adaptive step-size strategy is used. 

\begin{figure}[p]
\begin{center}
\includegraphics[height=3.8cm,width=4.3cm]{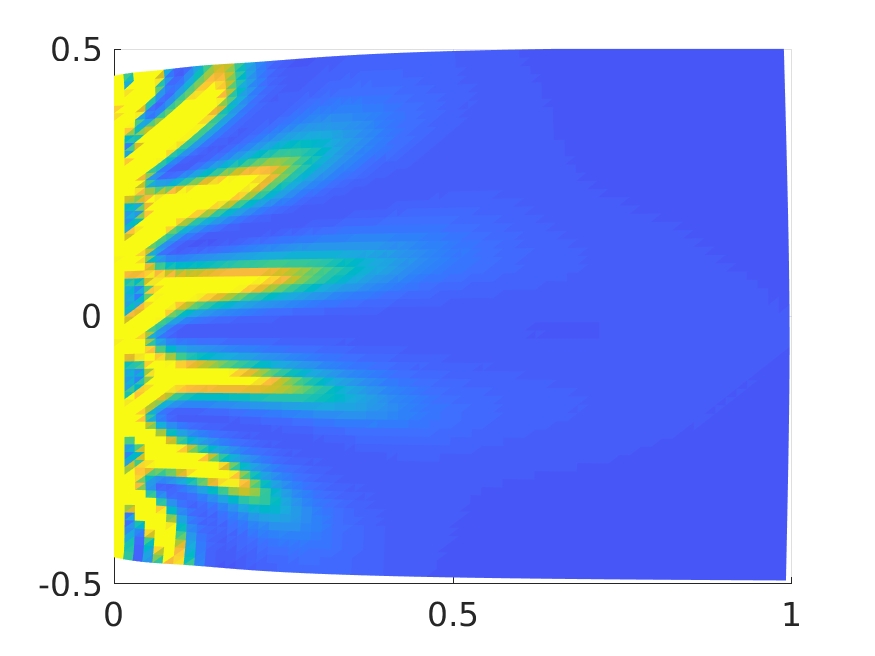} \hspace{-4mm}
\includegraphics[height=3.8cm,width=4.3cm]{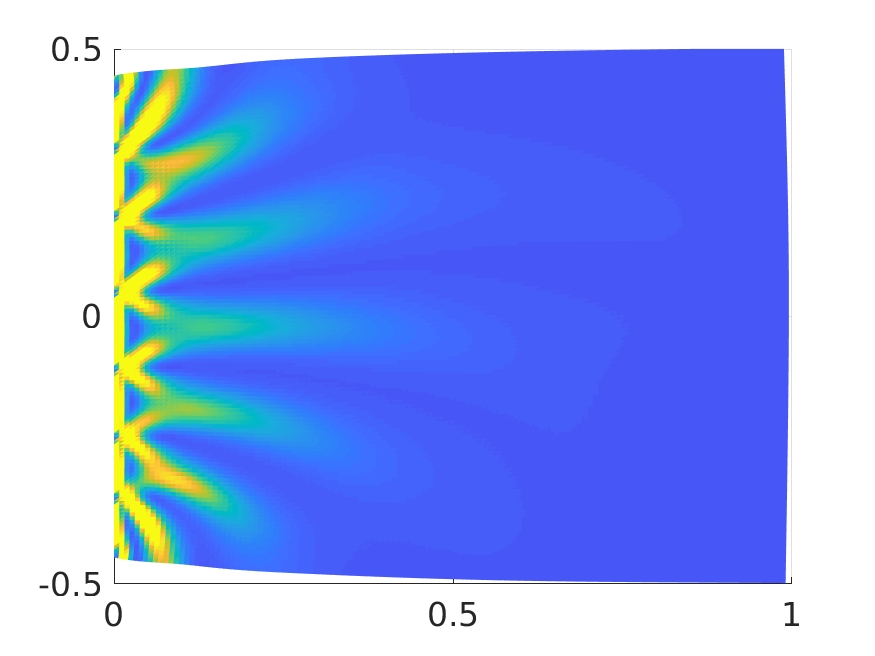} \hspace{-4mm}
\includegraphics[height=3.8cm,width=4.3cm]{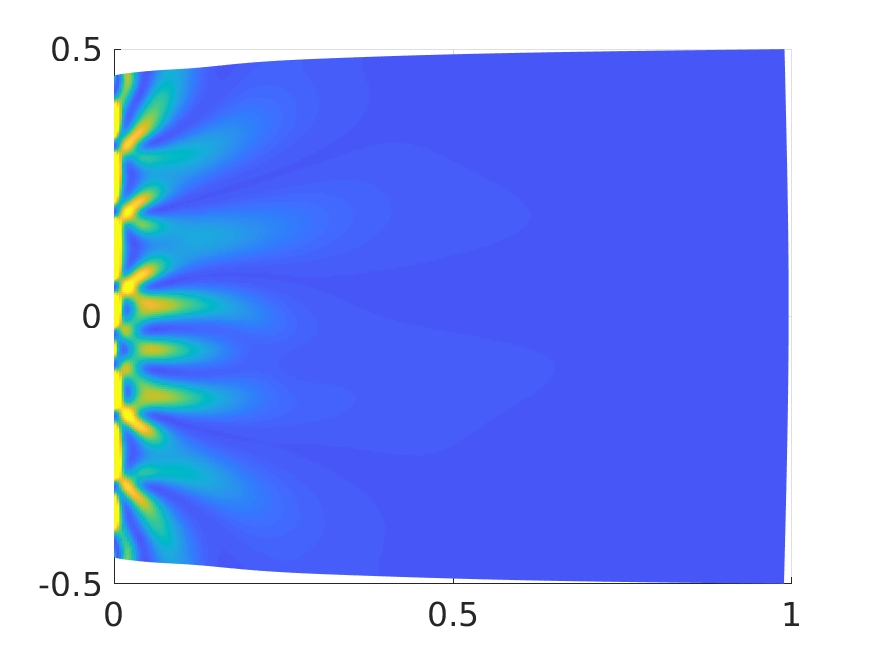} \\[.5cm]
\includegraphics[height=3.8cm,width=4.3cm]{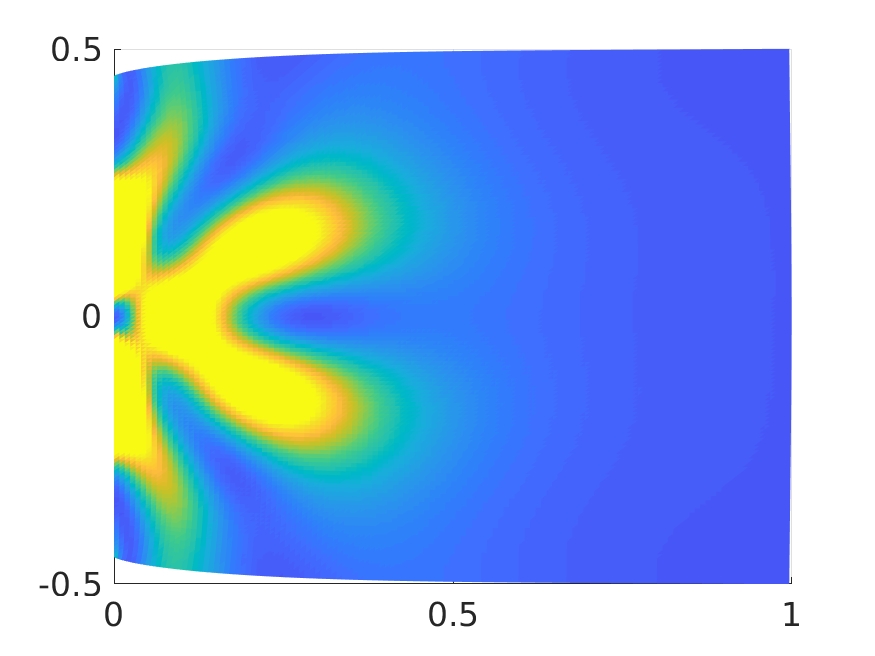} \hspace{-4mm}
\includegraphics[height=3.8cm,width=4.3cm]{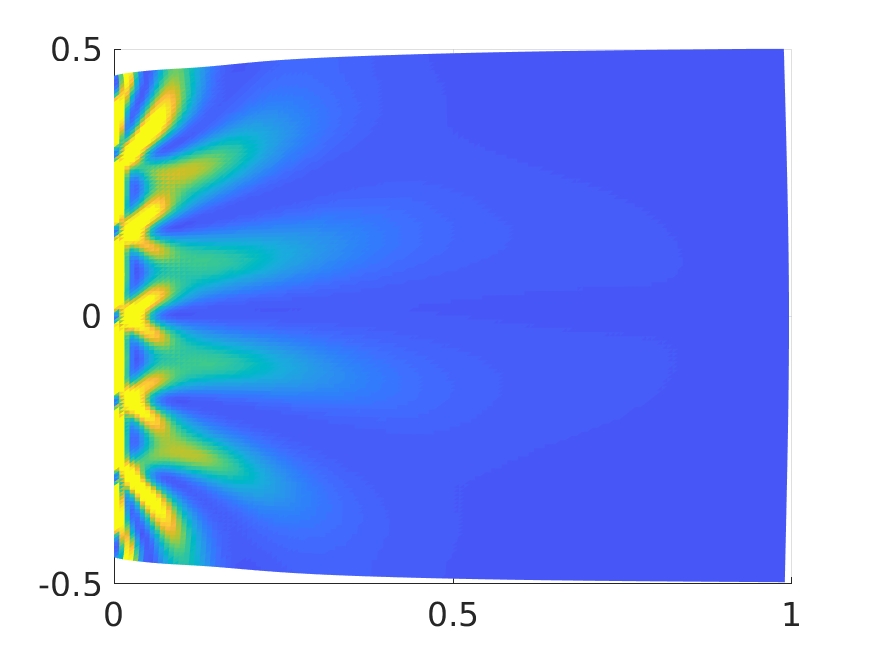} \hspace{-4mm}
\includegraphics[height=3.8cm,width=4.3cm]{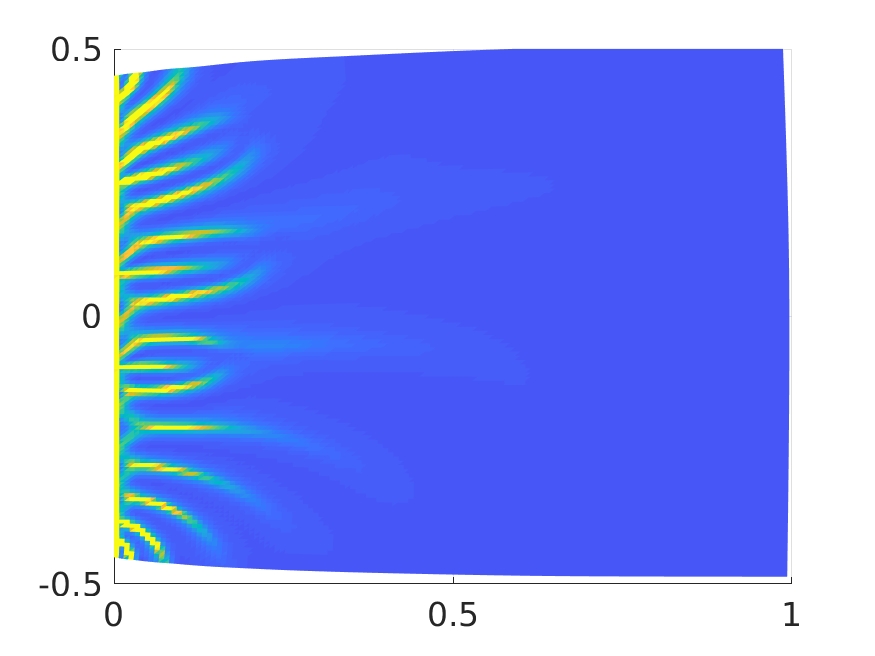} \\[.5cm]
\includegraphics[scale=.68]{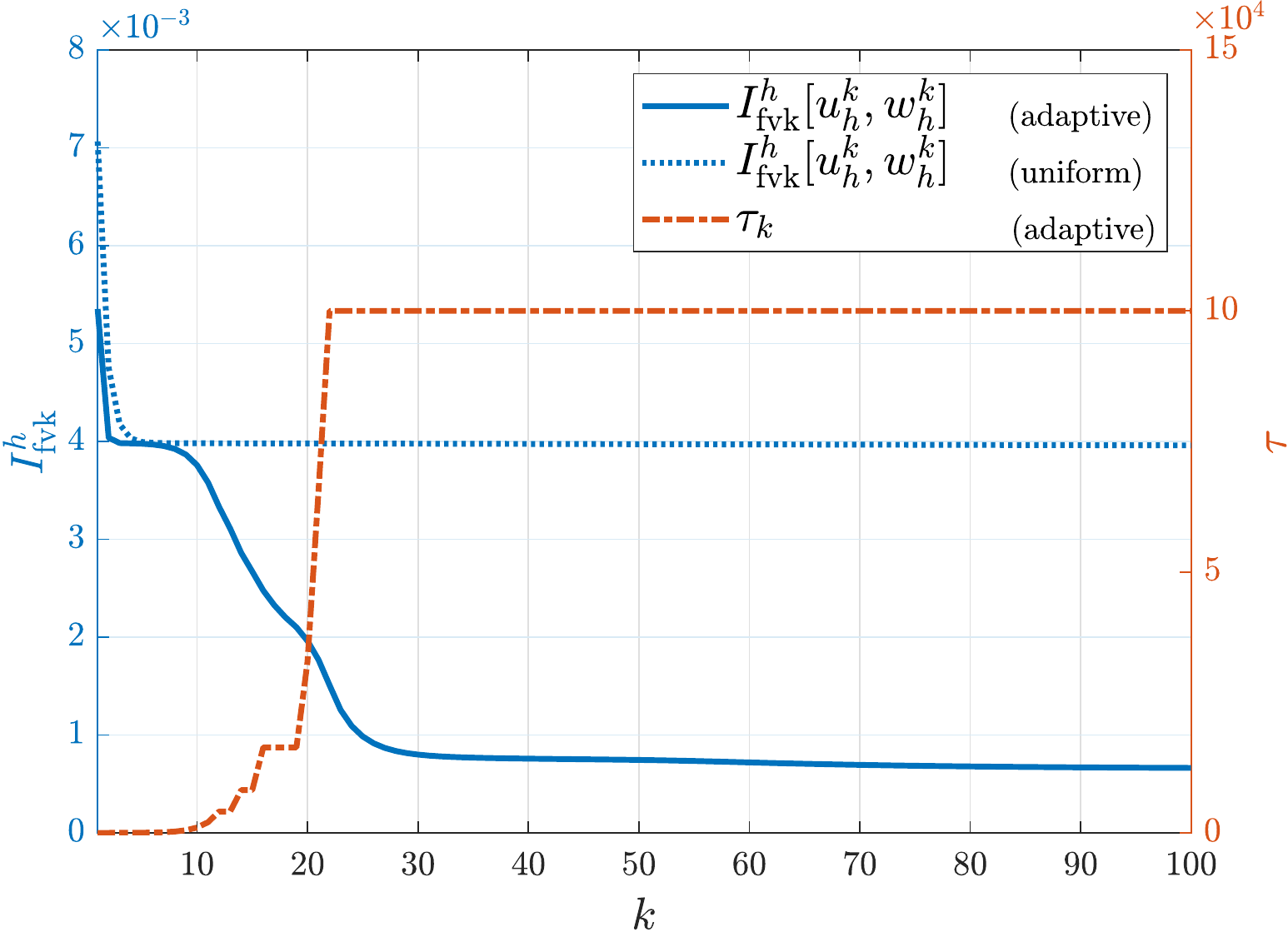} \\[.5cm]
\end{center}
\caption{\label{fig:fvk_exp}
In-plane deformation and bending energy density $|D^2 w_h|$ of the numerical 
approximations for the \fvk\ model with compressive boundary conditions
for different triangulations with mesh-sizes $h=2^{-\ell}$, $\ell=6,7,8$, and 
fixed $\d=1/200$ (top); 
for $\d =1/40$, $1/160$, and $1/640$ on a fixed triangulation with 
$h=2^{-7}$ (middle); energies and step-sizes for uniform and adaptive 
time-stepping (bottom).}
\end{figure}

\section{Conclusions}
We have addressed in this article the accurate finite element discretization
and reliable iterative solution of nonlinear models for describing large 
deformations of elastic rods and plates. To avoid unjustified regularity 
assumptions we have adopted the concept of $\Gamma$-convergence and thereby
proved convergence of approximations. To compute stationary configurations
of low and possibly minimal energy we have used gradient flow discretizations
that are guaranteed to decrease the elastic energy and preserve inextensibility
and isometry constraints appropriately. These two concepts leave a theoretical
gap in the sense that the framework of $\Gamma$-convergence is concerned
with global energy minimizers while gradient flows can only be expected to
determine stationary configurations. However, $\Gamma$-convergence goes 
beyond global energy minimization and gradient flows typically avoid unstable
critical configurations. These properties are convincingly confirmed by 
our numerical experiments which did not employ any a~priori knowledge
about expected solutions. The state of the art of reliable methods for 
computing nonlinear bending phenomena leaves open several aspects such as
the development of optimal preconditioners, adaptive refinement of stress
concentrations, determination of convergence rates, effective
combination with membrane effects, or simulation of realistic dynamic models
open. We believe that the methods presented in this article can be useful
in investigating them in future research. 

\bigskip \noindent
{\em Acknowledgments.}
The author wishes to thank his coworkers in various collaborations
that led to the results presented in this article. 

\bibliographystyle{abbrv}
\bibliography{bib_handbook}

\begin{thebibliography}{10}

\bibitem{Alou97}
F.~Alouges.
\newblock A new algorithm for computing liquid crystal stable configurations:
  the harmonic mapping case.
\newblock {\em SIAM J. Numer. Anal.}, 34(5):1708--1726, 1997.

\bibitem{Antm05-book}
S.~S. Antman.
\newblock {\em Nonlinear problems of elasticity}, volume 107 of {\em Applied
  Mathematical Sciences}.
\newblock Springer, New York, second edition, 2005.

\bibitem{AudPom10-book}
B.~Audoly and Y.~Pomeau.
\newblock {\em Elasticity and geometry}.
\newblock Oxford University Press, Oxford, 2010.

\bibitem{BaGaNu07}
J.~W. Barrett, H.~Garcke, and R.~N\"{u}rnberg.
\newblock A parametric finite element method for fourth order geometric
  evolution equations.
\newblock {\em J. Comput. Phys.}, 222(1):441--462, 2007.

\bibitem{BaGaNu12}
J.~W. Barrett, H.~Garcke, and R.~N\"{u}rnberg.
\newblock Parametric approximation of isotropic and anisotropic elastic flow
  for closed and open curves.
\newblock {\em Numer. Math.}, 120(3):489--542, 2012.

\bibitem{Bart05}
S.~Bartels.
\newblock Stability and convergence of finite-element approximation schemes for
  harmonic maps.
\newblock {\em SIAM J. Numer. Anal.}, 43(1):220--238, 2005.

\bibitem{Bart13a}
S.~Bartels.
\newblock Approximation of large bending isometries with discrete {K}irchhoff
  triangles.
\newblock {\em SIAM J. Numer. Anal.}, 51(1):516--525, 2013.

\bibitem{Bart13b}
S.~Bartels.
\newblock Finite element approximation of large bending isometries.
\newblock {\em Numer. Math.}, 124(3):415--440, 2013.

\bibitem{Bart13c}
S.~Bartels.
\newblock A simple scheme for the approximation of the elastic flow of
  inextensible curves.
\newblock {\em IMA J. Numer. Anal.}, 33(4):1115--1125, 2013.

\bibitem{Bart15-book}
S.~Bartels.
\newblock {\em Numerical methods for nonlinear partial differential equations},
  volume~47 of {\em Springer Series in Computational Mathematics}.
\newblock Springer, Cham, 2015.

\bibitem{Bart16}
S.~Bartels.
\newblock Projection-free approximation of geometrically constrained partial
  differential equations.
\newblock {\em Math. Comp.}, 85(299):1033--1049, 2016.

\bibitem{Bart17}
S.~Bartels.
\newblock Numerical solution of a {F}\"{o}ppl--von {K}\'{a}rm\'{a}n model.
\newblock {\em SIAM J. Numer. Anal.}, 55(3):1505--1524, 2017.

\bibitem{BBMN18}
S.~Bartels, A.~Bonito, A.~H. Muliana, and R.~H. Nochetto.
\newblock Modeling and simulation of thermally actuated bilayer plates.
\newblock {\em J. Comput. Phys.}, 354:512--528, 2018.

\bibitem{BaBoNo17}
S.~Bartels, A.~Bonito, and R.~H. Nochetto.
\newblock Bilayer plates: model reduction, {$\Gamma$}-convergent finite element
  approximation, and discrete gradient flow.
\newblock {\em Comm. Pure Appl. Math.}, 70(3):547--589, 2017.

\bibitem{BarHor15}
S.~Bartels and P.~Hornung.
\newblock Bending paper and the {M}\"{o}bius strip.
\newblock {\em J. Elasticity}, 119(1-2):113--136, 2015.

\bibitem{BarPal19-in-prep-pre}
S.~Bartels and C.~Palus.
\newblock Stable iterative solution of bilayer bending problems.
\newblock ({I}n preparation), 2019.

\bibitem{BarRei18-pre}
S.~Bartels and P.~Reiter.
\newblock Stability of a simple scheme for the approximation of elastic knots
  and self-avoiding inextensible curves.
\newblock {\em arXiv}, (arXiv:1804.02206), 2018.

\bibitem{BarRei19-in-prep-pre}
S.~Bartels and P.~Reiter.
\newblock Numerical solution of a bending-torsion model for elastic rods.
\newblock ({I}n preparation), 2019.

\bibitem{BaReRi18}
S.~Bartels, P.~Reiter, and J.~Riege.
\newblock A simple scheme for the approximation of self-avoiding inextensible
  curves.
\newblock {\em IMA J. Numer. Anal.}, 38(2):543--565, 2018.

\bibitem{BelKoh14}
P.~Bella and R.~V. Kohn.
\newblock Metric-induced wrinkling of a thin elastic sheet.
\newblock {\em J. Nonlinear Sci.}, 24(6):1147--1176, 2014.

\bibitem{Bergetal08}
M.~Bergou, M.~Wardetzky, S.~Robinson, B.~Audoly, and E.~Grinspun.
\newblock {Discrete Elastic Rods}.
\newblock {\em ACM Transactions on Graphics (SIGGRAPH)}, 27(3):63:1--63:12, aug
  2008.

\bibitem{Blat13}
S.~Blatt.
\newblock The energy spaces of the tangent point energies.
\newblock {\em J. Topol. Anal.}, 5(3):261--270, 2013.

\bibitem{BlaRei15}
S.~Blatt and P.~Reiter.
\newblock Regularity theory for tangent-point energies: the non-degenerate
  sub-critical case.
\newblock {\em Adv. Calc. Var.}, 8(2):93--116, 2015.

\bibitem{BoNoNt-in-prep-pre}
A.~Bonito, R.~H. Nochetto, and D.~Ntogkas.
\newblock {DG} approach to large bending plate deformations with isometry
  constraint.
\newblock ({I}n preparation), 2019.

\bibitem{Brae07-book}
D.~Braess.
\newblock {\em Finite elements}.
\newblock Cambridge University Press, Cambridge, third edition, 2007.

\bibitem{Ciar80}
P.~G. Ciarlet.
\newblock A justification of the von {K}\'arm\'an equations.
\newblock {\em Arch. Rational Mech. Anal.}, 73(4):349--389, 1980.

\bibitem{Ciar97-book}
P.~G. Ciarlet.
\newblock {\em Mathematical elasticity. {V}ol. {II}}, volume~27 of {\em Studies
  in Mathematics and its Applications}.
\newblock North-Holland Publishing Co., Amsterdam, 1997.
\newblock Theory of plates.

\bibitem{CiGrKe05}
P.~G. Ciarlet, L.~Gratie, and S.~Kesavan.
\newblock Numerical analysis of the generalized von {K}\'arm\'an equations.
\newblock {\em C. R. Math. Acad. Sci. Paris}, 341(11):695--699, 2005.

\bibitem{ConMag08}
S.~Conti and F.~Maggi.
\newblock Confining thin elastic sheets and folding paper.
\newblock {\em Arch. Ration. Mech. Anal.}, 187(1):1--48, 2008.

\bibitem{CoOlTo17}
S.~Conti, H.~Olbermann, and I.~Tobasco.
\newblock Symmetry breaking in indented elastic cones.
\newblock {\em Math. Models Methods Appl. Sci.}, 27(2):291--321, 2017.

\bibitem{Dalm93-book}
G.~Dal~Maso.
\newblock {\em An introduction to {$\Gamma$}-convergence}, volume~8 of {\em
  Progress in Nonlinear Differential Equations and their Applications}.
\newblock Birkh\"{a}user Boston, Inc., Boston, MA, 1993.

\bibitem{DeDzEl05}
K.~Deckelnick, G.~Dziuk, and C.~M. Elliott.
\newblock Computation of geometric partial differential equations and mean
  curvature flow.
\newblock {\em Acta Numer.}, 14:139--232, 2005.

\bibitem{DzKuSc02}
G.~Dziuk, E.~Kuwert, and R.~Sch\"{a}tzle.
\newblock Evolution of elastic curves in {$\Bbb R^n$}: existence and
  computation.
\newblock {\em SIAM J. Math. Anal.}, 33(5):1228--1245, 2002.

\bibitem{FHMP16}
L.~Freddi, P.~Hornung, M.~G. Mora, and R.~Paroni.
\newblock A corrected {S}adowsky functional for inextensible elastic ribbons.
\newblock {\em J. Elasticity}, 123(2):125--136, 2016.

\bibitem{FrJaMu02c}
G.~Friesecke, R.~D. James, and S.~M{\"u}ller.
\newblock The {F}\"oppl-von {K}\'arm\'an plate theory as a low energy
  {$\Gamma$}-limit of nonlinear elasticity.
\newblock {\em C. R. Math. Acad. Sci. Paris}, 335(2):201--206, 2002.

\bibitem{FrJaMu02b}
G.~Friesecke, R.~D. James, and S.~M\"{u}ller.
\newblock A theorem on geometric rigidity and the derivation of nonlinear plate
  theory from three-dimensional elasticity.
\newblock {\em Comm. Pure Appl. Math.}, 55(11):1461--1506, 2002.

\bibitem{FrJaMu06}
G.~Friesecke, R.~D. James, and S.~M\"{u}ller.
\newblock A hierarchy of plate models derived from nonlinear elasticity by
  gamma-convergence.
\newblock {\em Arch. Ration. Mech. Anal.}, 180(2):183--236, 2006.

\bibitem{FrJaMu02a}
G.~Friesecke, S.~M\"{u}ller, and R.~D. James.
\newblock Rigorous derivation of nonlinear plate theory and geometric rigidity.
\newblock {\em C. R. Math. Acad. Sci. Paris}, 334(2):173--178, 2002.

\bibitem{GonMad99}
O.~Gonzalez and J.~H. Maddocks.
\newblock Global curvature, thickness, and the ideal shapes of knots.
\newblock {\em Proc. Natl. Acad. Sci. USA}, 96(9):4769--4773, 1999.

\bibitem{Horn11}
P.~Hornung.
\newblock Approximation of flat {$W^{2,2}$} isometric immersions by smooth
  ones.
\newblock {\em Arch. Ration. Mech. Anal.}, 199(3):1015--1067, 2011.

\bibitem{KPPRS18-pre}
J.~Kraus, C.-M. Pfeiler, D.~Praetorius, M.~Ruggeri, and B.~Stiftner.
\newblock Iterative solution and preconditioning for the tangent plane scheme
  in computational micromagnetics.
\newblock {\em arXiv}, (1808.10281), 2018.

\bibitem{LanSin96}
J.~Langer and D.~A. Singer.
\newblock Lagrangian aspects of the {K}irchhoff elastic rod.
\newblock {\em SIAM Rev.}, 38(4):605--618, 1996.

\bibitem{MorMul03}
M.~G. Mora and S.~M\"{u}ller.
\newblock Derivation of the nonlinear bending-torsion theory for inextensible
  rods by {$\Gamma$}-convergence.
\newblock {\em Calc. Var. Partial Differential Equations}, 18(3):287--305,
  2003.

\bibitem{MulOlb14b}
S.~M{\"u}ller and H.~Olbermann.
\newblock Almost conical deformations of thin sheets with rotational symmetry.
\newblock {\em SIAM J. Math. Anal.}, 46(1):25--44, 2014.

\bibitem{NasSof96}
S.~G. Nash and A.~Sofer.
\newblock Preconditioning reduced matrices.
\newblock {\em SIAM J. Matrix Anal. Appl.}, 17(1):47--68, 1996.

\bibitem{OHar03-book}
J.~O'Hara.
\newblock Energy of knots and conformal geometry.
\newblock volume~33 of {\em Series on Knots and Everything}, pages xiv+288.
  World Scientific Publishing Co., Inc., River Edge, NJ, 2003.

\bibitem{Pakz04}
M.~R. Pakzad.
\newblock On the {S}obolev space of isometric immersions.
\newblock {\em J. Differential Geom.}, 66(1):47--69, 2004.

\bibitem{Pant03}
O.~Pantz.
\newblock On the justification of the nonlinear inextensional plate model.
\newblock {\em Arch. Ration. Mech. Anal.}, 167(3):179--209, 2003.

\bibitem{PozSti17}
P.~Pozzi and B.~Stinner.
\newblock Curve shortening flow coupled to lateral diffusion.
\newblock {\em Numer. Math.}, 135(4):1171--1205, 2017.

\bibitem{Rivi95}
T.~Rivi\`ere.
\newblock Everywhere discontinuous harmonic maps into spheres.
\newblock {\em Acta Math.}, 175(2):197--226, 1995.

\bibitem{Rolf90-book}
D.~Rolfsen.
\newblock {\em Knots and links}, volume~7 of {\em Mathematics Lecture Series}.
\newblock Publish or Perish, Inc., Houston, TX, 1990.
\newblock Corrected reprint of the 1976 original.

\bibitem{SaNeBi16}
O.~Sander, P.~Neff, and M.~B\^{i}rsan.
\newblock Numerical treatment of a geometrically nonlinear planar {C}osserat
  shell model.
\newblock {\em Comput. Mech.}, 57(5):817--841, 2016.

\bibitem{Schm07a}
B.~Schmidt.
\newblock Minimal energy configurations of strained multi-layers.
\newblock {\em Calc. Var. Partial Differential Equations}, 30(4):477--497,
  2007.

\bibitem{Schm07b}
B.~Schmidt.
\newblock Plate theory for stressed heterogeneous multilayers of finite bending
  energy.
\newblock {\em J. Math. Pures Appl. (9)}, 88(1):107--122, 2007.

\bibitem{SchEbe01}
O.~Schmidt and K.~Eberl.
\newblock Thin solid films roll up into nanotubes.
\newblock {\em Nature}, 410:168, 2001.

\bibitem{SchUhl82}
R.~Schoen and K.~Uhlenbeck.
\newblock A regularity theory for harmonic maps.
\newblock {\em J. Differential Geom.}, 17(2):307--335, 1982.

\bibitem{ShRoSw07}
E.~Sharon, B.~Roman, and H.~Swinney.
\newblock Geometrically driven wrinkling observed in free plastic sheets and
  leaves.
\newblock {\em Physical review. E, Statistical, nonlinear, and soft matter
  physics}, 75:046211, 04 2007.

\bibitem{Smela-etal93}
E.~Smela, O.~Ingan\"os, Q.~Pei, and I.~Lundstr\"om.
\newblock Electrochemical muscles: Micromachining fingers and corkscrews.
\newblock {\em Advanced Materials}, 5(9):630--632, 1993.

\bibitem{StSzvdM13}
P.~Strzelecki, M.~Szuma\'{n}ska, and H.~von~der Mosel.
\newblock On some knot energies involving {M}enger curvature.
\newblock {\em Topology Appl.}, 160(13):1507--1529, 2013.

\bibitem{StrVdM12}
P.~Strzelecki and H.~von~der Mosel.
\newblock Tangent-point self-avoidance energies for curves.
\newblock {\em J. Knot Theory Ramifications}, 21(5):1250044, 28, 2012.

\bibitem{ThCoSw04}
J.~M.~T. Thompson, B.~D. Coleman, and D.~Swigon.
\newblock Theory of self-contact in kirchhoff rods with applications to
  supercoiling of knotted and unknotted dna plasmids.
\newblock {\em Philosophical Transactions of the Royal Society of London.
  Series A: Mathematical, Physical and Engineering Sciences},
  362(1820):1281--1299, 2004.

\bibitem{Venk04}
S.~C. Venkataramani.
\newblock Lower bounds for the energy in a crumpled elastic sheet---a minimal
  ridge.
\newblock {\em Nonlinearity}, 17(1):301--312, 2004.

\bibitem{VdM98}
H.~von~der Mosel.
\newblock Minimizing the elastic energy of knots.
\newblock {\em Asymptot. Anal.}, 18(1-2):49--65, 1998.

\bibitem{WBHZG07}
M.~Wardetzky, M.~Bergou, D.~Harmon, D.~Zorin, and E.~Grinspun.
\newblock Discrete quadratic curvature energies.
\newblock {\em Comput. Aided Geom. Design}, 24(8-9):499--518, 2007.

\end{thebibliography}

\end{document}